\documentclass[11pt,a4paper]{amsart}
\usepackage[utf8]{inputenc}

\usepackage[mono=false,osf]{libertine}
\usepackage{libertinust1math}
\usepackage[T1]{fontenc}
\usepackage[scaled=0.88]{beramono}
\renewcommand{\mathsf}[1]{\text{\normalfont\sffamily#1}}
\usepackage[cal=euler,scr=boondoxo,bb=pazo]{mathalfa}

\usepackage{xcolor-material}
\usepackage[unicode, psdextra, colorlinks=true, linkcolor=GoogleBlue, citecolor=GoogleGreen, urlcolor=GoogleBlue, linktocpage]{hyperref}

\usepackage{amsmath,amsfonts,amssymb,amsthm,amscd}

\usepackage{trimclip}

\usepackage{setspace}

\usepackage{tikz}
\usetikzlibrary{cd}
\usepackage{mathrsfs}
\usepackage{multirow}
\usepackage{stmaryrd}
\SetSymbolFont{stmry}{bold}{U}{stmry}{m}{n} 
\usepackage{bm,bbm}
\usepackage{mathtools}
\usepackage{etoolbox}
\usepackage{todonotes}
\usepackage{enumitem}

\usepackage[margin=3cm]{geometry}

\tikzcdset{scale cd/.style={every label/.append style={scale=#1},
    cells={nodes={scale=#1}}}}

\makeatletter
\def\l@subsection{\@tocline{2}{0pt}{2.5pc}{5pc}{}}
\makeatother

\makeatletter
\newcommand\RedeclareMathOperator{%
  \@ifstar{\def\rmo@s{m}\rmo@redeclare}{\def\rmo@s{o}\rmo@redeclare}%
}
\newcommand\rmo@redeclare[2]{%
  \begingroup \escapechar\m@ne\xdef\@gtempa{{\string#1}}\endgroup
  \expandafter\@ifundefined\@gtempa
     {\@latex@error{\noexpand#1undefined}\@ehc}%
     \relax
  \expandafter\rmo@declmathop\rmo@s{#1}{#2}}
\newcommand\rmo@declmathop[3]{%
  \DeclareRobustCommand{#2}{\qopname\newmcodes@#1{#3}}%
}
\@onlypreamble\RedeclareMathOperator
\makeatother

\newcommand{\dmerge}[4]{
	\draw (#1,#2) .. controls (#1,#4*0.5+#2*0.5) and (#3*0.5+#1*0.5,#4*0.5+#2*0.5) .. (#3*0.5+#1*0.5,#4);
	\draw (#3,#2) .. controls (#3,#4*0.5+#2*0.5) and (#3*0.5+#1*0.5,#4*0.5+#2*0.5) .. (#3*0.5+#1*0.5,#4);
}
\newcommand{\dsplit}[4]{
	\draw (#3*0.5+#1*0.5,#2) .. controls (#3*0.5+#1*0.5,#4*0.5+#2*0.5) and (#3,#4*0.5+#2*0.5) .. (#3,#4);
	\draw (#3*0.5+#1*0.5,#2) .. controls (#3*0.5+#1*0.5,#4*0.5+#2*0.5) and (#1,#4*0.5+#2*0.5) .. (#1,#4);
}
\newcommand{\opbox}[5]{
	\draw (#1,#2) rectangle (#3,#4);
	\node[] at (#3*0.5+#1*0.5,#4*0.5+#2*0.5) {#5};
}
\newcommand{\ntxt}[3]{
	\node[text height=1.2ex,text depth=.25ex] at (#1,#2) {#3};
}
\newcommand{\crosin}[4]{
	\draw (#1,#2) .. controls (#1,#4*0.6+#2*0.4) and (#3,#4*0.4+#2*0.6) .. (#3,#4);
	\draw (#3,#2) .. controls (#3,#4*0.6+#2*0.4) and (#1,#4*0.4+#2*0.6) .. (#1,#4);
}
\newcommand{\strcros}[4]{
	\draw (#1,#2) -- (#3,#4);
	\draw (#3,#2) -- (#1,#4);
}
\newcommand{\strex}[2]{
	\draw (#1-0.3,#2-0.3) -- (#1+0.3,#2+0.3);
	\draw (#1+0.3,#2-0.3) -- (#1-0.3,#2+0.3);
}

\usepackage[capitalise]{cleveref}

\theoremstyle{plain}
\newtheorem{thm}{Theorem}[section]
\newtheorem{thmintro}{Theorem}

\newtheorem{prop}[thm]{Proposition}
\newtheorem{lem}[thm]{Lemma}
\newtheorem{cor}[thm]{Corollary}
\newtheorem{conj}[thm]{Conjecture}

\theoremstyle{definition}
\newtheorem{defn}[thm]{Definition}

\theoremstyle{remark}
\newtheorem{rmk}[thm]{Remark}
\newtheorem{expl}[thm]{Example}

\AddToHook{env/prop/begin}{\crefalias{thm}{prop}}
\AddToHook{env/lem/begin}{\crefalias{thm}{lem}}
\AddToHook{env/cor/begin}{\crefalias{thm}{cor}}
\AddToHook{env/conj/begin}{\crefalias{thm}{conj}}
\AddToHook{env/defn/begin}{\crefalias{thm}{defn}}
\AddToHook{env/notation/begin}{\crefalias{thm}{notation}}
\AddToHook{env/rmk/begin}{\crefalias{thm}{rmk}}
\AddToHook{env/expl/begin}{\crefalias{thm}{expl}}

\Crefname{thm}{Theorem}{Theorems}
\Crefname{thmintro}{Theorem}{Theorems}
\Crefname{lem}{Lemma}{Lemmata}
\Crefname{prop}{Proposition}{Propositions}
\Crefname{cor}{Corollary}{Corollaries}
\Crefname{conj}{Conjecture}{Conjectures}
\Crefname{defn}{Definition}{Definitions}
\Crefname{notation}{Notation}{Notations}
\Crefname{rmk}{Remark}{Remarks}

\numberwithin{equation}{section}

\usepackage{mysymbols}

\title{Stratifying quiver Schur algebras via ersatz parity sheaves}
\author{Ruslan Maksimau}
\address[Ruslan Maksimau]{Laboratoire Analyse, Géométrie et Modélisation, CY Cergy Paris Université, 95302 Cergy-Pontoise, France}
\email{ruslan.maksimau@cyu.fr}

\author{Alexandre Minets}
\address[Alexandre Minets]{Max Planck Institute for Mathematics, Bonn, Germany}
\email{minets@mpim-bonn.mpg.de}

\begin{document}

\begin{abstract}
	We propose an extension of the theory of parity sheaves, which allows for non-locally constant sheaves along strata.
	Our definition is tailored for proving the existence of (proper, quasihereditary, etc) stratifications of $\Ext$-algebras. 
	We use this to study quiver Schur algebra $A(\alpha)$ for the cyclic quiver of length $2$. 
	We find a polynomial quasihereditary structure on $A(\alpha)$ compatible with the categorified PBW basis of McNamara and Kleshchev-Muth, and sharpen their results to arbitrary characteristic.
	We also prove that semicuspidal algebras of $A(n\delta)$ are polynomial quasihereditary covers of semicuspidal algebras of the corresponding KLR algebra $R(n\delta)$, and compute them diagrammatically.
\end{abstract}
\date{\today}
\maketitle

\tableofcontents

\section*{Introduction}
Let $\Gamma = (I,H)$ be a quiver of affine type, and $U_q(\mathfrak{g}_\Gamma)$ the corresponding affine quantum group.
The categorification of its negative half $U_q(\mathfrak{g}_\Gamma^-)$ is provided by (graded, projective) modules over the quiver Hecke (KLR) algebras $R(\alpha)$, $\alpha\in \bbZ_{\geq 0}^I$~\cite{KL_DACQ2009}.
As $U_q(\mathfrak{g}_\Gamma^-)$ admits several important bases (PBW, canonical, their dual versions), one is naturally lead to ask what these bases are categorified by.

The answer is given by the notion of ($\mathscr{B}$-)properly stratified algebras~\cite{Kle_AHWC2015}. 
It is known that the algebras $R(\alpha)$ are properly stratified in characteristic $0$ or $p>0$ big enough~\cite{McN_RKAI2017,KM_SKAA2017}.
Their building blocks are the so-called \textit{semicuspidal} algebras $C(k\alpha')$, where $\alpha'$ runs over the positive roots of $\mathfrak{g}_\Gamma$.
The strata are parameterised by Kostant partitions of $\alpha$, and induction from simples in $C(k\alpha')$'s gives rise to proper standard modules.
These modules categorify the dual PBW basis~\cite{McN_RKAI2017}.
Taking their projective covers in $R(\alpha)$, one obtains a categorification of the canonical basis. 

While standard objects should be independent of the characteristic of the ground field, indecomposable projectives are typically not.
In particular, their decategorification recovers the $p$-canonical basis of $U_q(\mathfrak{g}_\Gamma)$~\cite{grojnowski1999affine,jensen2017p}, which should capture subtle modular information. 
It is therefore desirable to extend the stratification result above to small characteristic.

At the same time, not all stratifications are born equal.
Ideally, we want to have as few simples in each stratum as possible. 
For instance, when all strata of a properly stratified algebra $A$ are graded polynomial algebras, we say that $A$ is \textit{polyhereditary}.\footnote{This name is non-standard; one usually calls such algebras \textit{polynomial quasihereditary}. We opted for a contraction to save ink; see \cref{subs:polyher} for definitions}
It is known that imaginary semicuspidal algebras $C(k\delta)$, where $\delta$ is the indecomposable imaginary root of $\mathfrak{g}_\Gamma$, are not polynomial, or even polyhereditary~\cite{KM_AZAI2019}, so one is tempted to look for their polyhereditary covers.

\subsection*{Ersatz parity sheaves}
Recall that KLR algebras can be realized as convolution algebras.
In characteristic $0$, Kato~\cite{kato2017algebraic} considered varieties $X$ with finitely many $G$-orbits $\bbO_\lambda$, and proved polyheredity for $G$-equivariant convolution algebras $\Ext^*(L,L)$, $L = \bigoplus_\lambda \IC(\bbO_\lambda)$.
This result was extended in~\cite{McN_RTGE2017} to $\overline{\bbF}_p$-coefficients, by replacing $\IC$-sheaves with parity sheaves~\cite{JMW_PS2014}.
This immediately implies polyheredity for KLR algebras of Dynkin quivers.
However, the quivers of affine type have infinitely many isomorphism classes of representations, so these results do not apply.

We circumvent this by generalizing the theory of parity sheaves.
Let $X = \bigsqcup_\lambda X_\lambda$ be an algebraic stratification.
The axiomatics of~\cite{JMW_PS2014} are set up so that parity sheaves are locally constant on each stratum; we drop this condition.
Instead, on each $X_\lambda$ we pick a collection of sheaves with no extensions between them, which we call an \textit{evenness theory}.
This bears almost no influence on basic constructions; in particular, there exists at most one parity extension $\parsh(F)$ of an indecomposable sheaf $F\in\Ev(X_\lambda)$ to $X$. 
When such extension exists, we call it \textit{ersatz parity sheaf}.
It follows from our axioms that the collection of ersatz parity sheaves is itself an evenness theory $\Ev(X)$ on $X$; we say it is \textit{ersatz-complete} if all ersatz parity sheaves $\parsh(F)$ exist.

Given an evenness theory $\Ev$ on $X$, take the direct sum $\cL$ of all indecomposable sheaves in $\Ev$ and consider the Ext-algebra $A_\cL = \Ext^*(\cL,\cL)$.
We define a notion of \textit{polyheredity} of evenness theories, which ensures that $A_\cL$ is polyhereditary (\cref{lem:Ev-to-alg-qher}).
Our main technical result allows us to glue polyhereditary theories, and thus polyhereditary algebras:
\begin{thmintro}[{\cref{thm:quasi-her-inher}}]\label{thm:A}
	Let $X = \bigsqcup_\lambda X_\lambda$, equipped with a polyhereditary evenness theory $\Ev(X_\lambda)$ on each $X_\lambda$.
	If $\Ev(X)$ is ersatz-complete, then it is also polyhereditary.
\end{thmintro}

\subsection*{Quiver Schur for Kronecker quiver}
Let $\Gamma = \bullet\rightrightarrows \bullet$ be the Kronecker quiver.
The imaginary semicuspidal algebras for KLR algebras of $\Gamma$ were computed in~\cite{MakMin_KLR2023} in arbitrary characteristic.
Here, we study imaginary semicuspidal algebras of the quiver Schur algebras of $\Gamma$.
Since quiver Hecke and Schur algebras only differ for quivers with cycles, we consider a ``seminilpotent'' variant $A^\m(\alpha)$ of quiver Schur algebra; see \cref{subs:polyrep} for definitions. 

We define a Harder-Narasimhan-type stratification of $\Rep_\alpha \Gamma$, and show that the associated evenness theory is precisely the Karoubi envelope of the additive category of flag sheaves (\cref{cor:even-are-flag}); in particular, ersatz parity sheaves in this case are closely related to geometric extensions as defined in~\cite{hone2025geometric} (\cref{rmk:geom-Ext}).
This quickly implies
\begin{thmintro}[{\cref{thm:Schur-poly-qher}}]\label{thm:B}
	The quiver Schur algebra $A_\Gamma^\m(\alpha)$ is polyhereditary.
\end{thmintro}

Finally, similarly to~\cite{MakMin_KLR2023} we provide the imaginary semicuspidal algebra with an explicit diagrammatic description by considering torsion sheaves on $\bbP^1$. 
\begin{thmintro}[{\cref{thm:curve-Schur-poly,thm:End(P+)}}]
	The imaginary semicuspidal algebra of $A_\Gamma^\m(n\delta)$ is Morita-equivalent to the extended curve Schur algebra of $\bbP^1$ (\cref{sec:schur-dotted}). 
	It provides a polyhereditary cover of the imaginary semicuspidal algebra of $R(n\delta)$, as computed in~\cite{MakMin_KLR2023}.
\end{thmintro}

As a byproduct, we obtain a simpler polyhereditary cover of the semicuspidal algebra of $R(n\delta)$ in characteristic $0$ (\cref{subs:Savage}), and extend the stratification result of~\cite{McN_RKAI2017,KM_SKAA2017} to arbitrary characteristic (\cref{subs:KLR-prop-str}).
We also compute integral cohomology of the moduli stack of torsion sheaves on $\bbP^1$, thus correcting some mistakes in~\cite{MakMin_KLR2023} (\cref{subs:erratum}).

\subsection*{Modular representation theory}
Semicuspidal quotients feature prominently in the study of blocks of $\bbF_p[\fkS_d]$, or more generally of Ariki-Koike algebras at roots of unity (higher level).
The typical strategy is to use derived equivalences to reduce a given question from arbitrary blocks to \textit{RoCK blocks}; see~\cite{chuang2002symmetric,chuang2008derived} for $\bbF_p[\fkS_d]$, and~\cite{lyle2022rouquier,webster2024rock} for higher level.
Thanks to the foundational result of Brundan-Kleshchev~\cite{brundan2009blocks}, blocks of Ariki-Koike algebras (at roots of unity) are equivalent to blocks of cyclotomic KLR algebras for cyclic quivers.

For the symmetric groups, RoCK blocks were realized inside the category of semicuspidal modules of KLR algebra of a cyclic quiver~\cite{evseev2018blocks}, thus proving a non-abelian version of Broué conjecture (for symmetric groups). 
In the higher level case, $C(n\delta)$ only captures a certain piece of RoCK blocks of Ariki-Koike~\cite{muth2024skew}.
A presentation of higher RoCK blocks as modules over quotients of KLRW algebras was recently established in~\cite{munasinghe2025steadied}.
We hope that our methods can be used to study these quotients geometrically in positive characteristic, and for example realize them as $\Ext$-algebras of ersatz parity sheaves in some suitable evenness theory.

\subsection*{Organisation}
In \cref{sec:prereq}, we recall Schur diagrammatics and convolution algebras in Borel-Moore homology.
In \cref{sec:Quiver-Schur}, we fix our notations for quiver Schur algebras, and introduce the seminilpotent version for the Kronecker quiver.
We define and study evenness theories \cref{sec:ersatz-parity}, culminating in the proof of \cref{thm:A}.
We construct a special evenness theory for representations of the Kronecker quiver using Mackey filtration on flag sheaves, and deduce \cref{thm:B} in \cref{sec:rep-cS}.
\cref{sec:semicusp-quot} relates our results back to semicuspidal quotients, and explains how one could extend our approach to other affine quivers.
The imaginary semicuspidal algebra is computed in \cref{sec:schur-dotted}.
Finally, \cref{subs:erratum} discusses integral cohomology of stacks of torsion sheaves on curves, and provides erratum for~\cite{MakMin_KLR2023}.

\subsection*{Acknowledgements}
This project took its current shape after A.M. tried to explain the results of~\cite{MakMin_KLR2023} to D. Juteau at Colloque tournant in Luminy, and was met with confusion at how the mismatch of $\bbZ$-lattices therein could ever occur.
We are grateful to him for the patient ear.
The authors are also grateful to Max Planck Institute for Mathematics in Bonn and CY Cergy Paris Université for their hospitality and financial support.

\medskip
\section{Prerequisites}\label{sec:prereq}
Let $\bbk$ be a field. 
All varieties we consider are over $\bbC$; the field $\bbk$ plays the role of the coefficient ring for (co)homology and sheaves.

\subsection{Coloured compositions}
Let us fix a finite set $I$; we refer to its elements as \textit{colours}.
\begin{defn}\label{def:comps}
	\begin{enumerate}
		\item Let $n\in \bbZ_{\geq 0}$. A \textit{quasi-composition} of $n$ of length $k$ is a $k$-tuple of non-negative integers $\lambda=(\lambda^1,\lambda^2,\ldots,\lambda^k)$ which sum up to $n$. We further say that $\lambda$ is a \textit{composition} if all $\lambda^j$'s are positive, and a \emph{partition} if $\lambda^1\geqslant\lambda^2\geqslant\ldots\geqslant\lambda^k>0$.
		\item Let $\vecn = (n_i)_i \in \bbZ_{\geq 0}^I$. An \textit{$I$-composition} of $\vecn$ is a $k$-tuple $\Ibe = (\Ibe^1,\Ibe^2,\ldots,\Ibe^k)$ of nonzero elements of $\bbZ_{\geq 0}^I$, which sum up to $\vecn$. We call $\ell(\beta)\coloneqq k$ the \textit{length} of $\beta$. For each $i\in I$, we have $\Ibe_i=(\Ibe_i^1,\ldots,\Ibe_i^k)$ a quasi-composition of $n_i$. 
		\item An \textit{$I$-coloured composition} of $n$ of length $k$ is a word $\cla = (\lambda^1i^1,\ldots,\lambda^ki^k)$, where $i^j\in I$ and $\sum_j\lambda^j = n$. Equivalently, it is a pair $\cla = (\lambda,(i^j)_j)$, where $\lambda=(\lambda^1,\ldots,\lambda^k)$ is a composition of $n$, and $(i^1,\ldots, i^k)$ is a $k$-tuple of colours.
	\end{enumerate}
\end{defn}

Denote the set of ($I$-)compositions of $\vecn$ by $\Comp(\vecn)$, and the set of $I$-coloured compositions on $n$ by $I^{(n)}$; our convention is that for $n=0$ both sets consist of one (trivial) composition of length $0$.
$I$-coloured compositions can be seen as $I$-compositions, where each $\Ibe^j$ is supported in one single colour, and so we have $I^{(n)}\subset \Comp(\vecn)$.

Let $|\cdot|:\bbZ^I_{\geq 0}\to \bbZ_{\geq 0}$ be the map sending $(n_i)_i$ to $\sum_{i} n_i$. Then each $I$-composition $\Ibe=(\Ibe^1,\ldots,\Ibe^k)$ of $\vecn$ yields a composition $|\Ibe|=(|\Ibe^1|,\ldots,|\Ibe^k|)$ of $|\vecn|$.

Each quasi-composition $\lambda$ yields a parabolic subgroup $\fkS_\lambda \coloneqq \prod_{j=1}^k\fkS_{\lambda^j}$ of $\fkS_n$.
Let us denote $\fkS_{\vecn} = \prod_{i\in I}\fkS_{n_i}$; then each $I$-composition $\Ibe$ yields a parabolic subgroup $\fkS_{\Ibe} = \prod_{i\in I}\fkS_{\Ibe_i}$ of $\fkS_{\vecn}$.
For $I$-coloured compositions, we set $\fkS_\cla \coloneqq \fkS_\lambda$, viewed as a subgroup of $\fkS_n$.
When talking about double cosets of the form $\fkS_\lambda\backslash \fkS_n / \fkS_\mu$, we will always implicitly identify them with the shortest length representatives. 

For two compositions $\lambda,\mu\in\Comp(n)$, we say that $\mu$ is a \textit{refinement} of $\lambda$ if $\fkS_\mu\subset \fkS_\lambda$, and denote it by $\mu\vDash \lambda$.
We have the same definition for $I$-compositions.
For two $I$-coloured compositions $\cla$, $\cmu$ we say that $\cmu$ is a refinement of $\cla$ if it is a refinement as $I$-compositions.

\subsection{Schur diagrammatics}\label{subs:sp-mer-diag}
Let us recall the widely used diagrammatic language for permutations and compositions.

\begin{defn}
	Let $n$ be a positive integer.
	A \textit{Schur diagram} of weight $n$ is a planar diagram, which around every point looks in one of the following ways:
	
	\begin{center}
		\setlength{\tabcolsep}{24pt}
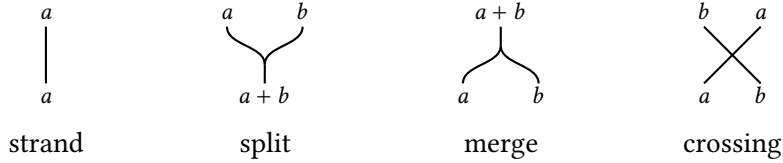
		\begin{tabular}{ c c c c } 
			\tikz[thick,xscale=.25,yscale=.25,font=\footnotesize]{\draw (2,0) -- (2,3); \ntxt{2}{-0.7}{$a$} \ntxt{2}{3.7}{$a$}} & 
			\tikz[thick,xscale=.25,yscale=.25,font=\footnotesize]{\dsplit{0}{1}{4}{3} \draw (2,0) -- (2,1); \ntxt{0}{3.7}{$a$} \ntxt{4}{3.7}{$b$} \ntxt{2}{-0.7}{$a+b$}} & 
			\tikz[thick,xscale=.25,yscale=.25,font=\footnotesize]{\dmerge{0}{0}{4}{2}\draw (2,2) -- (2,3); \ntxt{0}{-0.7}{$a$} \ntxt{4}{-0.7}{$b$} \ntxt{2}{3.7}{$a+b$}} & 
			\tikz[thick,xscale=.25,yscale=.25,font=\footnotesize]{\strcros{0}{0}{3}{3} \ntxt{0}{-0.7}{$a$} \ntxt{3}{-0.7}{$b$} \ntxt{0}{3.7}{$b$} \ntxt{3}{3.7}{$a$}} \\ 
		 	strand & split & merge & crossing \\ 
		\end{tabular}
	\end{center}
	We label each strand with a positive integer, called \textit{thickness}; we say a strand is \textit{thin} if it has thickness $1$. Thicknesses add up on splits and merges, and are required to sum up to $n$ on each horizontal slice.
	Strands go from bottom to top, caps or cups are not allowed.
\end{defn}

We will always impose the following associativity relations on splits and merges:
\begin{equation}\label{eq:pic-asso}
	\begin{aligned}
	\tikz[thick,xscale=.2,yscale=.2,font=\footnotesize]{
	\dsplit{0}{3}{4}{5}
	\draw (8,3) -- (8,5);
	\dsplit{2}{1}{8}{3}
	\draw (5,0) -- (5,1);
	\ntxt{0}{6}{$a$}
	\ntxt{4}{6}{$b$}
	\ntxt{8}{6}{$c$}
	\ntxt{5}{-1}{$a+b+c$}
	
	\ntxt{10}{2}{$=$}
	
	\dsplit{16}{3}{20}{5}
	\draw (12,3) -- (12,5);
	\dsplit{12}{1}{18}{3}
	\draw (15,0) -- (15,1);
	\ntxt{12}{6}{$a$}
	\ntxt{16}{6}{$b$}
	\ntxt{20}{6}{$c$}
	\ntxt{15}{-1}{$a+b+c$}
	
	\ntxt{21}{2}{,}
	}\qquad\qquad
	\tikz[thick,xscale=.2,yscale=.2,font=\footnotesize]{
	\dmerge{0}{0}{4}{2}
	\draw (8,0) -- (8,2);
	\dmerge{2}{2}{8}{4}
	\draw (5,4) -- (5,5);
	\ntxt{0}{-1}{$a$}
	\ntxt{4}{-1}{$b$}
	\ntxt{8}{-1}{$c$}
	\ntxt{5}{6}{$a+b+c$}
	
	\ntxt{10}{2}{$=$}
	
	\dmerge{16}{0}{20}{2}
	\draw (12,0) -- (12,2);
	\dmerge{12}{2}{18}{4}
	\draw (15,4) -- (15,5);
	\ntxt{12}{-1}{$a$}
	\ntxt{16}{-1}{$b$}
	\ntxt{20}{-1}{$c$}
	\ntxt{15}{6}{$a+b+c$}
	
	\ntxt{21}{2}{.}
	}
\end{aligned}
\end{equation}

This allows us to make sense of more general splits and merges:
$$
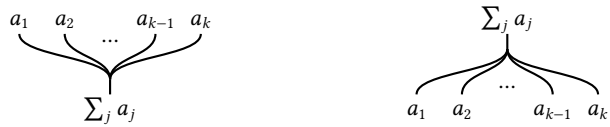
\tikz[thick,xscale=.2,yscale=.2,font=\footnotesize]{
	\draw (5,0) -- (5,1);
	\dsplit{-1}{1}{11}{4}
	\dsplit{2}{1}{8}{4}
	\ntxt{-1}{5}{$a_1$}
	\ntxt{2}{5}{$a_2$}
	\ntxt{5}{3.5}{$\cdots$}
	\ntxt{8}{5}{$a_{k-1}$}
	\ntxt{11}{5}{$a_k$}
	\ntxt{5}{-1}{$\sum_j a_j$}
}\qquad\qquad\qquad
\tikz[thick,xscale=.2,yscale=.2,font=\footnotesize]{
	\draw (5,3) -- (5,4);
	\dmerge{-1}{0}{11}{3}
	\dmerge{2}{0}{8}{3}
	\ntxt{-1}{-1}{$a_1$}
	\ntxt{2}{-1}{$a_2$}
	\ntxt{5}{0.5}{$\cdots$}
	\ntxt{8}{-1}{$a_{k-1}$}
	\ntxt{11}{-1}{$a_k$}
	\ntxt{5}{5}{$\sum_j a_j$}
}
$$

In particular, for any $\lambda,\mu\in \Comp(n)$ with $\mu \vDash \lambda$ we can draw the corresponding split and merge diagrams, with thicknesses given by components of $\lambda$ and $\mu$.
Furthermore, if $\lambda = (\lambda^1\ldots,\lambda^k)\in\Comp(n)$ and $\sigma\in \fkS_k$, then for any reduced decomposition of $\sigma$ we can draw the corresponding diagram which only has crossings, and the thicknesses of strands are given by $\lambda^j$'s.

Let $\lambda,\mu\in \Comp(n)$, $\ell(\lambda) = k$, $\ell(\mu) = l$, and $w\in\fkS_\lambda\backslash \fkS_n/\fkS_\mu$.
Define the compositions $\lambda'\vDash \lambda$, $\mu'\vDash \mu$ by 
\begin{align*}
	\lambda' = (\lambda'^{11},\ldots,\lambda'^{1l},\ldots,\lambda'^{kl}), \quad \lambda'^{ij} = \#([\lambda^{i-1}+1,\lambda^i]\cap w[\mu^{j-1}+1,\mu^j]),\\
	\mu' = (\mu'^{11},\ldots,\mu'^{1k},\ldots,\mu'^{lk}), \quad \mu'^{ij} = \#([\mu^{i-1}+1,\mu^i]\cap w^{-1}[\lambda^{j-1}+1,\lambda^j]).
\end{align*}
We have $\fkS_{\lambda'}=\fkS_\lambda\cap w\fkS_\mu w^{-1}$, $\fkS_{\mu'}=\fkS_\mu\cap w^{-1}\fkS_\lambda w$.
Denoting the length of $\lambda'$ by $k'$, the permutation $w\in\fkS_n$ induces a permutation $\sigma\in\fkS_{k'}$, which transforms $\lambda'$ into $\mu'$.
Fix a reduced decomposition of $\sigma$.

\begin{defn}
	Denote by $D_w$ the Schur diagram characterized by the following properties:
	\begin{itemize}
		\item $D_w$ is a concatenation of there parts: top, middle and bottom;
		\item The bottom part is the split corresponding to $\lambda' \vDash \lambda$, the top part is the merge corresponding to $\mu' \vDash \mu$;
		\item The middle part is the crossing diagram corresponding to $\sigma$.
	\end{itemize}
\end{defn}

\begin{expl}
	Let $\lambda=(2,2,1)$, $\mu=(3,2)$, and take $w=s_3s_4$. Then $\lambda'=(2,2,1)$, $\mu'=(2,1,2)$, and $D_w$ is as follows:
	$$
	
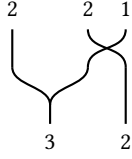
\tikz[thick,xscale=.25,yscale=.25,font=\footnotesize]{
	\draw (2,0) -- (2,1);
	\dsplit{0}{1}{4}{3}
	\draw (6,0) -- (6,3);
	\crosin{4}{3}{6}{5}
	\draw (0,3) -- (0,5);
	\ntxt{2}{-1}{$3$}
	\ntxt{6}{-1}{$2$}
	\ntxt{0}{6}{$2$}
	\ntxt{4}{6}{$2$}
	\ntxt{6}{6}{$1$}
	}
	$$
\end{expl}

\begin{rmk}
	We do not impose $\fkS_n$-relations on crossings.
	However, in situations of our interest they hold ``up to lower terms'', see e.g. proof of~\cite[Th.~4.10]{MakMin_KLR2023}.
\end{rmk}

Given a set of colours $I$, we can define Schur diagrams with thicknesses valued in $\bbZ^I_{\geq 0}$ in a straightforward way.
All the considerations above extend accordingly after replacing compositions by $I$-compositions; for example, for any $\Ibe,\Iga\in \Comp(\vecn)$ and $w\in \fkS_{\Ibe}\backslash \fkS_{\vecn}/\fkS_{\Iga}$, we get a Schur diagram $D_w$.
More precisely, in this case for the definition of $D_w$ we should take $\Ibe'$ and $\Iga'$ defined by
\begin{equation}\label{eq:betapr-gammapr}
	\begin{aligned}
		\Ibe'_i = (\Ibe'^{11}_i,\ldots,\Ibe'^{1l}_i,\ldots,\Ibe'^{kl}_i), \quad \Ibe'^{tr}_i = \#([\Ibe^{t-1}_i+1,\Ibe^t_i]\cap w_i[\Iga^{r-1}_i+1,\Iga^r_i]),\\
		\Iga'_i = (\Iga'^{11}_i,\ldots,\Iga'^{1k}_i,\ldots,\Iga'^{lk}_i), \quad \Iga'^{tr}_i = \#([\Iga^{t-1}_i+1,\Iga^t_i]\cap w^{-1}_i[\Ibe^{r-1}_i+1,\Ibe^r_i]).
	\end{aligned}
\end{equation}
In some cases, we will only want to consider $I$-coloured compositions; this corresponds to only allowing strands, splits and merges of pure colour.

\subsection{Symmetric polynomials}
We will work with various polynomial rings.
Define $\Pol_n=\Bbbk[X_1,\ldots,X_n]$, and for any quasi-composition $\lambda$ set $\Pol_\lambda=\Pol_n^{\fkS_\lambda}$. 
More generally, for an $I$-composition $\Ibe$ of $\vecn = (n_i)_i$ we define $\Pol_{\vecn} = \bigotimes_{i\in I} \Pol_{n_i}$, and $\Pol_{\Ibe} = \Pol_{\vecn}^{\fkS_{\Ibe}}$.
Note that we have obvious identifications 
\[
	\Pol_\lambda\simeq \bigotimes_{j=1}^k \Pol_{\lambda^j}^{\fkS_{\lambda^j}},\qquad \Pol_{\Ibe} \simeq \bigotimes_{i\in I} \Pol_{n_i}^{\fkS_{{\Ibe}_i}}.
\]

For $r\in [1,n-1]$, denote by $\partial_r$ the \textit{Demazure operator} 
\[
	\partial_r\colon \Pol_n\to \Pol_n, \qquad
	P\mapsto \frac{P-s_r(P)}{X_r-X_{r+1}}.
\]

For any $w\in\fkS_n$ we define $\partial_w\coloneqq\partial_{k_1}\ldots\partial_{k_r}$, where $w=s_{k_1}\ldots s_{k_r}$ is a reduced expression.
Since Demazure operators satisfy braid relations, $\partial_w$ is independent of choices.

For positive integers $a$, $b$ with $a+b=n$ consider the shuffle permutation
\begin{equation}\label{eq:shuffle-perm}
	w_{a,b}(i)=
	\begin{cases}
	i+b & \mbox{ if } 1\leqslant i\leqslant a,\\
	i-a & \mbox{ if } a< i\leqslant n.
	\end{cases}
\end{equation}
We abbreviate $\partial_{a,b}\coloneqq\partial_{w_{a,b}}$.
For any $P\in\Pol_n^{\fkS_a\times\fkS_b}$, we have $\partial_{a,b}P\in\Pol_n^{\fkS_n}$.

\subsection{Convolution algebras}\label{subs:conv-algs}
Let us briefly recall the geometric setup of~\cite{CG_RTCG2010}.
Let $\pi:Y\to X$ be a proper map of algebraic varieties with smooth source.
We define $Z\coloneqq Y\times_X Y$, and consider its Borel-Moore homology $A = H_*(Z) = H_*(Z,\bbk)$ with coefficients in $\bbk$.
Consider the following correspondence:
\[
	\begin{tikzcd}
		Z\times Z\ar[d,hook] & Y\times_X Y\times_X Y\ar[d,hook]\ar[r,"p_{13}"]\ar[l,"p_{12}\times p_{23}"'] & Z \\
		Y^2\times Y^2 & Y\times Y\times Y\ar[l,"\Id_Y\times \Delta_Y\times \Id_Y"'] & 
	\end{tikzcd}
\]
One equips $A$ with the associative product $p_{13*}(p_{12}\times p_{23})^!$, where $(p_{12}\times p_{23})^!$ is the Gysin pullback along the diagonal embedding $\Delta_Y$.
Note that this product depends not only on $Z$, but also on the choice of the map $\pi$.
We have a natural action of $A$ on $H_*(Y)\simeq H^{2\dim Y - *}(Y)$ via a similar correspondence.
Both $A$ and $H_*(Y)$ are naturally $H^*(X)$-modules; both algebra structure on $A$ and $A$-action on $H_*(Y)$ are $H^*(X)$-linear. 

In the case when $Y = \sqcup_i Y_i$ has several connected components, we have an obvious decomposition $A = A_{ij} \coloneqq H_*(Z_{ij})$, $Z_{ij} = Y_i\times_X Y_j$.
We equip $A$ with a grading by setting
\[
	A = \bigoplus_k A_k,\qquad A_k \coloneqq \bigoplus_{i,j} H_{\dim Y_i + \dim Y_j - k}(Z_{ij}).
\]
Let us consider the constructible sheaf $\ona{IC}_Y \coloneqq \bigoplus_i \bbk_{Y_i}[\dim Y_i]$.
The following statement is proved in~\cite[\S~8.6]{CG_RTCG2010}:

\begin{prop}\label{prop:Ext-description}
	We have an isomorphism of graded algebras $A\simeq \Ext^*(\pi_*\ona{IC}_Y,\pi_*\ona{IC}_Y)$.
\end{prop}

When $X$, $Y$ are $G$-varieties for a linear algebraic group $G$, and $\pi$ is $G$-equivariant, we can replace Borel-Moore homology by $G$-equivariant Borel-Moore homology.
Everything we listed above remains true.

\medskip
\section{Quiver Schur algebras}
\label{sec:Quiver-Schur}

\subsection{Quivers and flags}
Let $\Gamma=(I,H)$ be a quiver, where $I$ denotes the set of vertices and $H$ the set of edges.
We denote by $s,t:H\to I$ the maps associating to an edge its source and target respectively.
Write $Q_I = \bbZ^I$, $Q^+_I = \bbZ^I_{\geq 0}$; note that $Q_I$ has an obvious basis $\{\alpha_i,\,i\in I\}$ over $\bbZ$.
For $\vecn = \sum_i n_i\alpha_i \in Q_I$, we write $|\vecn| \coloneqq \sum_i n_i$.
Given an $I$-graded vector space $V = \bigoplus_i V_i$, its dimension vector $\dim_I(V) = \sum_i \dim(V_i)\alpha_i$ is naturally an element of $Q^+_I$.

\begin{defn}
	Let $V$ be a finite-dimensional $\bbC$-vector space.
	A \text{quasi-flag} of length $k$ in $V$ is a sequence of vector spaces
	$$
	V^\bullet=(\{0\}=V^0\subset V^1\subset\cdots\subset V^k=V);
	$$
	we write $\ell(V^\bullet) \coloneqq k$. For a quasi-flag $V^\bullet$ we say that it is  
	\begin{itemize}
		\item a \emph{flag}, if $V^{j-1}\neq V^j$ for all $1\leq j\leq k$,
		\item a \emph{full flag}, if $\dim V^r/V^{r-1}=1$ for all $1\leq j\leq k$.
	\end{itemize}
	When $V$ is an $I$-graded vector space, we say that a flag $V^\bullet$ is \textit{$I$-homogeneous} if $V^r=\bigoplus_{i\in I}V^r_i$ for all $1\leq r\leq k$, where $V_i^r = V^r\cap V_i$.
\end{defn}

We denote the variety of all (not necessarily full) $I$-homogeneous flags in $V$ by $\bfF_\vecn$. Note that $(\{0\}=V^0_i\subset V^1_i\subset\ldots\subset V^k_i=V_i)$ is a quasi-flag in $V_i$, but not always a flag.

Given an $I$-homogeneous flag $V^\bullet$, denote $\Ibe^r \coloneqq \dim_I V^r/\dim_I V^{r-1}$ for $1\leqslant r\leqslant \ell(\Ibe)$ and call the composition $\Ibe=(\Ibe^1,\ldots,\Ibe^k) \in \Comp(\vecn)$ the \textit{type} of $V^\bullet$.
Denoting by $\bfF_{\Ibe}$ the variety of $I$-homogeneous flags of type $\Ibe$, we have $\bfF_\vecn=\coprod_{\Ibe\in\Comp(\vecn)}\bfF_\Ibe$.

\subsection{Quiver Schur algebras}\label{subs:qschur}
Let $V = \bigoplus_i V_i$ be an $I$-graded vector space of dimension $\vecn$.
Set $E_\vecn=\bigoplus_{h\in H}{\Hom}(V_{s(h)},V_{t(h)})$.
The group $G_\vecn=\prod_{i\in I}GL(V_i)$ acts on $E_\vecn$, as well as on $\bfF_\Ibe$ for all $\Ibe\in\Comp(\vecn)$, in a natural way.

Denote by $\widetilde{\bfF}^\s_{\Ibe}$ the closed subvariety of pairs $(x,V^\bullet)\in E_\vecn\times \bfF_\Ibe$ such that $x$ preserves the flag $V^\bullet$, in other words for each arrow $h\in H$ and $1\leq j\leq \ell(\Ibe)$ we have $x_h(V^j_{s(h)})\subset V^j_{t(h)}$.
This subvariety is preserved by the action of $G_\vecn$, and is known to be smooth.
We will also write $\widetilde{\bfF}^\s_\vecn=\coprod_{\Ibe\in \Comp(\vecn)} \widetilde \bfF^\s_\Ibe$.

For $\Ibe,\Iga\in \Comp(\vecn)$, we define the \textit{quiver Steinberg varieties}
\[
	\bfZ^\s_{\Ibe,\Iga}=\widetilde{\bfF}^\s_{\Ibe}\times_{E_\vecn}\widetilde{\bfF}^\s_{\Iga},\qquad \bfZ^\s_{\vecn}= \widetilde{\bfF}^\s_{\vecn}\times_{E_\vecn}\widetilde{\bfF}^\s_{\vecn} =\coprod_{\Ibe,\Iga\in \Comp(\vecn)}\bfZ^\s_{\Ibe,\Iga}.
\]

As explained in~\cref{subs:conv-algs}, we have a convolution algebra structure on $H_*^{G_\vecn}(\bfZ^\s_\vecn,\bbk)$.

\begin{defn}
	The \textit{(standard) quiver Schur algebra} of $\Gamma$ is the algebra $A^\s(\vecn)\coloneqq H_*^{G_\vecn}(\bfZ^\s_\vecn,\bbk)$.
\end{defn}

\begin{rmk}\label{rmk:stack-notations}
	We often find it convenient to use the language of stacks, instead of varieties with group actions.
	Namely, consider the quotient stacks 
	\[
		\Rep_\vecn \coloneqq [E_\vecn/G_\vecn],\quad \Fl^\s_\Ibe \coloneqq [\widetilde{\bfF}^\s_{\Ibe}/G_\vecn].
	\]
	In this notation, we have $H_*^{G_\vecn}(\bfZ^\s_{\Ibe,\Iga})\equiv H_*(\Fl_\Ibe\times_{\Rep_\vecn} \Fl_\Iga)$.
\end{rmk}

It is also customary (see~\cite{SW_QSAF2014}) to consider smaller Steinberg varieties.
Namely, denote by $\widetilde{\bfF}^\n_{\Ibe}\subset \widetilde{\bfF}^\s_{\Ibe}$ the variety of pairs $(x,V^\bullet)\subset E_\vecn\times \bfF_\Ibe$ such that $x$ \textit{strictly} preserves the flag $V^\bullet$, that is for each arrow $h\in H$ and $1\leq j\leq \ell(\Ibe)$ we have $x_h(V^r_{s(h)})\subset V^{r-1}_{t(h)}$.
When $\Ibe$ is the type of a full flag, we have $\widetilde{\bfF}^\n_{\Ibe} = \widetilde{\bfF}^\s_{\Ibe}$, but in general the inclusion is strict.

Anologously to the above, for any $\Ibe,\Iga\in\Comp(\vecn)$ we define
\[
	\bfZ^\n_{\Ibe,\Iga}=\widetilde{\bfF}^\n_{\Ibe}\times_{E_\vecn}\widetilde{\bfF}^\n_{\Iga},\qquad \bfZ^\n_{\vecn}= \widetilde{\bfF}^\n_{\vecn}\times_{E_\vecn}\widetilde{\bfF}^\n_{\vecn} =\coprod_{\Ibe,\Iga\in \Comp(\vecn)}\bfZ^\n_{\Ibe,\Iga},
\]
and the convolution algebra $H_*^{G_\vecn}(\bfZ^\n_\vecn,\bbk)$.

\begin{defn}
	The \textit{nilpotent quiver Schur algebra} of $\Gamma$ is the algebra $A^\n(\vecn) \coloneqq H_*^{G_\vecn}(\bfZ^\n_\vecn,\bbk)$.
\end{defn}

Let $\star\in\{\s,\n\}$.
For each $\Ibe\in\Comp(\vecn)$ we define an idempotent $1_\Ibe\in A^\star(\vecn)$ as the fundamental class of the diagonal subvariety $\widetilde{\bfF}^\star_\Ibe\subset \bfZ^\star_{\Ibe,\Ibe}$.
We have $1=\sum_{\Ibe\in \Comp(\vecn)}1_\Ibe$ in $A^\star(\vecn)$.

Let $\Iga\vDash\Ibe$.
We have a map $f:\bfF_\Iga\to \bfF_\Ibe$ obtained by forgetting subspaces of the flag.
Consider the subvariety $\widetilde\bfF^\star_{\Iga,\Ibe}\subset \widetilde\bfF^\star_\Iga$ of pairs $(x,V^\bullet)\in E_\vecn\times \bfF_\Iga$ such that $x$ (strictly, if $\star=\n$) preserves $f(V^\bullet)$.
We have a map $\widetilde\bfF^\star_{\Iga,\Ibe}\to \widetilde\bfF^\star_\Ibe$ induced by $f$, and so a closed embedding of $\widetilde\bfF^\star_{\Iga;\Ibe}$ into both $\bfZ^\star_{\Iga,\Ibe}$ and $\bfZ^\star_{\Ibe,\Iga}$.
We denote the corresponding fundamental classes by $S^\star_{\Iga,\Ibe}$ and $M^\star_{\Ibe,\Iga}$ respectively.

\begin{defn}\label{def:quiver-Schur-SM}
	We call the element $S^\star_{\Iga,\Ibe}\in 1_\Iga A^\star(\vecn)1_\Ibe$ a \textit{split}, and the element $M^\star_{\Ibe,\Iga}\in 1_\Ibe A^\star(\vecn)1_\Iga$ a \textit{merge}.
\end{defn}

\begin{rmk}
	Note that $\widetilde\bfF^\s_{\Iga,\Ibe} = \widetilde\bfF^\s_\Iga$, but $\widetilde\bfF^\n_{\Iga,\Ibe} \subset \widetilde\bfF^\n_\Iga$ is typically a strict inclusion.
\end{rmk}

Let $\eta\vDash\Iga\vDash\Ibe$.
It is clear from the definitions that splits, merges and idempotents satisfy the following relations:
\begin{gather*}
	1_\Iga S^\star_{\Iga,\Ibe} = S^\star_{\Iga,\Ibe}1_\Ibe=S^\star_{\Iga,\Ibe},\qquad
	1_\Ibe M^\star_{\Ibe,\Iga} = M^\star_{\Ibe,\Iga}1_\Iga = M^\star_{\Ibe,\Iga},\\
	S^\star_{\eta,\Iga}S^\star_{\Iga,\Ibe}=S^\star_{\eta,\Ibe},\qquad 
	M^\star_{\Ibe,\Iga}M^\star_{\Iga,\eta}=M^\star_{\Ibe,\eta}.
\end{gather*}
In particular, this means that we can represent them using Schur diagrammatics:

\[
	\tikz[thick,xscale=.25,yscale=.25,font=\footnotesize]{
		\ntxt{-2}{1.5}{$1_\beta=$}
		\ntxt{6}{1}{,}
		\draw (0,0) -- (0,3); 
		\draw (2,0) -- (2,3);
		\draw (5,0) -- (5,3); 
		\ntxt{0}{-0.7}{$\Ibe^1$}
		\ntxt{0}{3.7}{$\Ibe^1$} 
		\ntxt{2}{-0.7}{$\Ibe^2$}
		\ntxt{2}{3.7}{$\Ibe^2$} 
		\ntxt{3.5}{1}{$\ldots$} 
		\ntxt{5}{-0.7}{$\Ibe^k$}
		\ntxt{5}{3.7}{$\Ibe^k$}
	}\qquad 
	\tikz[thick,xscale=.25,yscale=.25,font=\footnotesize]{
		\ntxt{-3}{1.5}{$S^\star_{\Iga,\Ibe}=$}
		\ntxt{9}{1}{,}
		\draw (1.5,0) -- (1.5,1);
		\draw (6,0) -- (6,3);
		\dsplit{0}{1}{3}{3}
		\dsplit{4}{1}{8}{3}
		\ntxt{1.5}{-0.7}{$\Ibe^1$}
		\ntxt{6}{-0.7}{$\Ibe^k$}
		\ntxt{0}{3.7}{$\Iga^1$}
		\ntxt{8}{3.7}{$\Iga^l$}
		\ntxt{3.5}{0.5}{$\cdots$}
		\ntxt{3.5}{3.7}{$\cdots$}
	}\qquad
	\tikz[thick,xscale=.25,yscale=.25,font=\footnotesize]{
		\ntxt{-3}{1.5}{$M^\star_{\Ibe,\Iga}=$}
		\ntxt{9}{1}{.}
		\draw (1.5,2) -- (1.5,3);
		\draw (6,0) -- (6,3);
		\dmerge{0}{0}{3}{2}
		\dmerge{4}{0}{8}{2}
		\ntxt{1.5}{3.7}{$\Ibe^1$}
		\ntxt{6}{3.7}{$\Ibe^k$}
		\ntxt{0}{-0.7}{$\Iga^1$}
		\ntxt{8}{-0.7}{$\Iga^l$}
		\ntxt{3.5}{2.5}{$\cdots$}
		\ntxt{3.5}{-0.7}{$\cdots$}
	}
\]

Recall that $A^\star(\vecn)$ acts by convolution on
\[
	H^*_{G_\vecn}(\widetilde{\bfF}_\vecn^\star,\bbk) = \bigoplus_{\Ibe \in \Comp(\vecn)} H^*_{G_\vecn}(\widetilde{\bfF}_\Ibe^\star,\bbk) \simeq \bigoplus_{\Ibe\in\Comp(\vecn)} \Pol_\Ibe;
\]
we denote this module by $\APol^\star_\vecn$, and call it the \textit{polynomial representation}.
We will write out the action of $A^\star(\vecn)$ on $\APol^\star_\vecn$ more explicitly in \cref{subs:polyrep}.
For now, note that the diagonal inclusion $\widetilde{\bfF}^\star_\Ibe\subset \bfZ_{\Ibe,\Ibe}^\star$ induces an algebra homomorphism
\[
	\Pol_\Ibe \to 1_\Ibe A^\star(\vecn)1_\Ibe, \qquad P\mapsto [\widetilde{\bfF}^\star_\Ibe]\cup P.
\]
Diagrammatically, we will draw elements of these polynomial subalgebras as coupons:
\[
	\tikz[thick,xscale=.3,yscale=.3,font=\footnotesize]{
			\draw (0,0) -- (0,2); 
			\draw (1,0) -- (1,2);
			\draw (3,0) -- (3,2); 
			\draw (0,3) -- (0,5); 
			\draw (1,3) -- (1,5);
			\draw (3,3) -- (3,5);
			\opbox{-0.5}{2}{3.5}{3}{$P$}
			\ntxt{0}{-0.7}{$\Ibe^1$}
			\ntxt{1}{-0.7}{$\Ibe^2$}
			\ntxt{2}{0.5}{$\ldots$} 
			\ntxt{3}{-0.7}{$\Ibe^k$}	
			\ntxt{4}{2}{$,$}	
		}
	\qquad\qquad 
	\tikz[thick,xscale=.3,yscale=.3,font=\footnotesize]{
			\draw (0,0) -- (0,1); 
			\draw (1,0) -- (1,1);
			\draw (3,0) -- (3,1); 
			\draw (0,2) -- (0,3); 
			\draw (1,2) -- (1,3);
			\draw (3,2) -- (3,3);
			\draw (0,4) -- (0,5); 
			\draw (1,4) -- (1,5);
			\draw (3,4) -- (3,5);
			\opbox{-0.5}{3}{3.5}{4}{$Q$}
			\opbox{-0.5}{1}{3.5}{2}{$P$}
			\ntxt{0}{-0.7}{$\Ibe^1$}
			\ntxt{1}{-0.7}{$\Ibe^2$}
			\ntxt{2}{2.7}{$\ldots$} 
			\ntxt{3}{-0.7}{$\Ibe^k$}
			\ntxt{4.5}{2.5}{$=$}
			\draw (6,0) -- (6,1.75); 
			\draw (7,0) -- (7,1.75);
			\draw (9,0) -- (9,1.75); 
			\draw (6,3.25) -- (6,5); 
			\draw (7,3.25) -- (7,5);
			\draw (9,3.25) -- (9,5);
			\opbox{5.5}{1.75}{9.5}{3.25}{$PQ$}
			\ntxt{6}{-0.7}{$\Ibe^1$}
			\ntxt{7}{-0.7}{$\Ibe^2$}
			\ntxt{8}{0.5}{$\ldots$} 
			\ntxt{9}{-0.7}{$\Ibe^k$}	
		}
\]

\begin{rmk}\label{rmk:KLR-not-Schur}
	Let $\Comp_0(\vecn)\subset \Comp(\vecn)$ be the subset of all compositions $\Ibe = (\Ibe^1,\ldots,\Ibe^k)$ satisfying $|\Ibe^i| = 1$ for all $i$; in other words, $\Comp_0(\vecn)$ is the set of types of full flags.
	Denote $1_{\mathrm{KLR}} \coloneqq \sum_{\Ibe \in \Comp_0(\vecn)} 1_\Ibe$, and consider
	\[
		A_0(\vecn) \coloneqq 1_{\mathrm{KLR}} A^\s(\vecn) 1_{\mathrm{KLR}} = \bigoplus_{\Ibe,\Iga\in \Comp_0(\vecn)} H_*^{G_\vecn}(\bfZ_{\Ibe,\Iga},\bbk).
	\]
	This is the quiver Hecke algebra of $\Gamma$ as defined in~\cite{KL_DACQ2009}.
	If $\Gamma$ has no edge loops, we clearly have $\widetilde{\bfF}^\n_{\Ibe} = \widetilde{\bfF}^\s_{\Ibe}$ for any $\Ibe\in \Comp_0(\vecn)$.
	In particular, in this case $1_{\mathrm{KLR}} A^\s(\vecn) 1_{\mathrm{KLR}} = 1_{\mathrm{KLR}} A^\n(\vecn) 1_{\mathrm{KLR}}$.
\end{rmk}

\subsection{Basis of $A(\vecn)$}\label{subs:qschur-basis}
Let us recall a basis of $A(\vecn)$. Since this construction works for both $A^\s(\vecn)$ and $A^\n(\vecn)$, we will temporarily drop the superscripts.
We describe a basis of $1_\Ibe A(\vecn)1_\Iga$ for fixed $ \Ibe,\Iga\in \Comp(\vecn)$.
Given $w\in \fkS_{\Ibe}\backslash \fkS_{\vecn}/\fkS_{\Iga}$, recall the split-merge diagram $D_w$ from~\cref{subs:sp-mer-diag}.
We can interpret it inside $A(\vecn)$ by sending splits to splits, merges to merges, and crossings to ``naive'' crossings, given by a composition of a merge with a split:
\begin{equation}\label{eq:dumb-cross}
	\begin{aligned}
	\tikz[thick,xscale=.25,yscale=.25,font=\footnotesize]{
		\strcros{0}{0}{4}{4}
		\ntxt{0}{-0.7}{$\beta^1$}
		\ntxt{4}{-0.7}{$\beta^2$}
		\ntxt{0}{4.5}{$\beta^2$}
		\ntxt{4}{4.5}{$\beta^1$}
		\ntxt{6}{2}{$\coloneqq$}
		\dmerge{8}{0}{12}{1.7}
		\draw (10,1.7) -- (10,2.3);
		\dsplit{8}{2.3}{12}{4}
		\ntxt{8}{-0.7}{$\beta^1$}
		\ntxt{12}{-0.7}{$\beta^2$}
		\ntxt{8}{4.5}{$\beta^2$}
		\ntxt{12}{4.5}{$\beta^1$}
	}
	\end{aligned}
\end{equation}

Recall the intermediate composition $\Ibe'$ from~\eqref{eq:betapr-gammapr}.
For every $P\in \Pol_{\Ibe'}$, we consider the element $\psi_{w}^P\in 1_\Ibe A(\vecn)1_\Iga$, which corresponds to the diagram obtained from $D_w$ by placing a coupon with label $P$ on the horizontal line separating the splits and the crossings.

For each $\Ibe'$, fix a basis $B_{\Ibe'}$ of $\Pol_{\Ibe'}$.
The following statement is proved in~\cite[Th.~3.11]{SW_QSAF2014},~\cite[Th.~3.25]{Prz_QSAC2019}.
\begin{prop}
	\label{prop-basis-QS-nocol}
	The set 
	$$
	\{\psi_{w}^P,\,w\in \fkS_{\Ibe}\backslash \fkS_{\vecn}/\fkS_{\Iga},P\in B_{\Ibe'}\}
	$$
	is a basis of $1_\Ibe A(\vecn)1_\Iga$.
\end{prop}

\subsection{Fourier transform}\label{subs:FS-trans}
Let us briefly explain what happens to the quiver Schur algebra when we change the orientation of $\Gamma$.

Let $G$ be an algebraic group, $X$ a $G$-variety, and $E$ a $G$-equivariant vector bundle on $X$.
We say that a complex of sheaves is \textit{monodromic} if it is locally constant along the orbits of the scaling $\bbG_m$-action on the fibers of $E$, and denote by $D^b_\mon(E/G) = D^b_{G,\mon}(E,\bbk)$ the $G$-equivariant derived category of monodromic complexes of $\bbk$-vector spaces on $E$.
The \textit{Fourier-Sato transform} is an equivalence of categories 
\[
	\Theta_E : D^b_\mon(E/G) \to D^b_\mon(E^\vee/G);
\]
see~\cite[2.7]{AHJR} for the precise definition.
The following proposition summarizes some properties of $\Theta_E$, found in \textit{loc.cit.}:

\begin{prop}\label{prop:Fourier-props}
	\begin{enumerate}
		\item Identifying $X$ with the zero section in $E$, we have $\Theta_E(\bbk_X) \simeq \bbk_{E^\vee}[\dim E]$;
		\item If $E = E_1\oplus E_2$, then $\Theta_E = \Theta_{E_1}\boxtimes\Theta_{E_2}$;
		\item Let $\bbD$ be Verdier duality functor. We have 
		\[
			\bbD\Theta_E = \Theta_E\bbD(-1),\qquad \Theta_{E^\vee}\Theta_E = (-1), 
		\]
		where $(-1)$ is the auto-equivalence of $D^b_\mon(E/G)$ induced by multiplication by $(-1)$ along the fibers of $E \to X$.
	\end{enumerate}
\end{prop}

Fix a map $\diamond:H\to \{\s,\n\}$.
We can generalize the considerations in \cref{subs:qschur} and define a quiver Schur algebra $A^\diamond(\vecn)$ starting with flag varieties $\bfF^\diamond_\vecn$, where we impose strict condition on $h$ if $\diamond(h) = \n$, and the non-strict one otherwise. Considering the constant functions $\s:h\mapsto \s$ and $\n:h\mapsto \n$, we recover our previous notations.

The convolution algebra $H_*^{G_\vecn}(\bfZ^\diamond_\vecn)$ can be alternatively seen as $\Ext^*(p_*\ona{IC}_{\widetilde{\bfF}_\vecn^\diamond},p_*\ona{IC}_{\widetilde{\bfF}_\vecn^\diamond})$, see~\cref{prop:Ext-description}.
Given a collection of arrows $H'\subset H$, consider the quiver $\Gamma'$, obtained from $\Gamma$ by inverting the arrows in $H'$. Note that $E^\Gamma_\vecn$ and $E^{\Gamma'}_\vecn$ are dual vector bundles over the base $B = \bigoplus_{h\in H\setminus H'}{\Hom}(V_{s(h)},V_{t(h)})$, so that we have Fourier-Sato equivalences $\Theta_\vecn:D^b_\mon(E^\Gamma_\vecn/G_\vecn)\to D^b_\mon(E^{\Gamma'}_\vecn/G_\vecn)$.

\begin{lem}\label{lm:Fourier}
	Denote by $\Theta(\diamond)$ the function obtained from $\diamond$ by inverting the values on $H'$.
	We have an isomorphism of algebras $A_\Gamma^\diamond(\vecn) \simeq A_{\Gamma'}^{\Theta(\diamond)}(\vecn)$.
\end{lem}
\begin{proof}
	To ease the notational burden, we will only consider $\diamond = \s$, the general proof being completely analogous.
	Let $E^\Gamma_\Ibe = \prod_j E_{\Ibe^j}$, and $E^0_\Ibe\subset E^\Gamma_\Ibe$ is the subspace of all representations where the components corresponding to the edges in $H'$ vanish.
	Observe that we have a $G_\vecn$-equivariant map $q:\widetilde{\bfF}^\s_\Ibe\to \bfF_\Ibe\times E^\Gamma_\Ibe$ which remembers the flag $V^\bullet$ and takes a representation $x$ to its associated graded with respect to $V^\bullet$.
	The preimage of $\bfF_\Ibe\times E^0_\Ibe$ under this map is precisely $\widetilde{\bfF}^{\Theta(\s)}_\Ibe$.
	Consider the following commutative diagram:
	\[
		\begin{tikzcd}
			E^{\Gamma'}_{\vecn} &\widetilde{\bfF}^{\Theta(\s)}_\Ibe\ar[l,"p'"']\ar[d,"q'"]\ar[r,equal,"\iota"] & \widetilde{\bfF}^{\Theta(\s)}_\Ibe\ar[r,hook]\ar[d,"q_0"]& \widetilde{\bfF}^{\s}_\Ibe\ar[r,"p"]\ar[d,"q"] & E_\vecn \\
			& \bfF_\Ibe\times E^{\Gamma'}_{\Ibe} & \bfF_\Ibe\times E^0_{\Ibe}\ar[r,hook]\ar[l,hook',"\iota"'] & \bfF_\Ibe\times E^\Gamma_{\Ibe} & 
		\end{tikzcd}
	\]
	and denote $d_\s = \dim \widetilde{\bfF}_\Ibe^\s$, $d_{\Theta(\s)} = \dim \widetilde{\bfF}_\Ibe^{\Theta(\s)}$. We have 
	\begin{align*}
		\Theta(p_*\bbk[d_\s]) &\simeq \Theta(p_*q^*\bbk[d_\s]) \\&\simeq p'_*(q')^*\Theta(\bbk[d_\s]) \simeq p'_*(q')^*\iota_*\bbk[d_{\Theta(\s)}] \simeq p'_*\iota_*q^*_0\bbk[d_{\Theta(\s)}] \simeq p'_*\bbk[d_{\Theta(\s)}],
	\end{align*}
	where the isomorphism between two lines is due to Lusztig~\cite[\S~10.2]{Lus1993}.
	Recalling the Ext-algebra description of $A^\diamond(\vecn)$, we may conclude.
\end{proof}

\begin{cor}
	Fourier-Sato transform $\Theta_\vecn$ induces an isomorphism of algebras $A^\s_\Gamma(\vecn)\simeq A^\n_{\Gamma^{\rm op}}(\vecn)$.
	Up to signs, this isomorphism sends $1_\Ibe$ to $1_\Ibe$, $S^\s_{\Iga,\Ibe}$ to $S^\n_{\Iga,\Ibe}$, $M^\s_{\Ibe,\Iga}$ to $M^\n_{\Ibe,\Iga}$ and coupons to coupons.
\end{cor}
\begin{proof}
	Coupons are sent to coupons by equivariance.
	The rest of the second statement is left as an exercise to the reader; the correction signs might appear because of \cref{prop:Fourier-props}(3). 
\end{proof}

\subsection{Seminilpotent version for Kronecker quiver}\label{subs:polyrep}
From now on, let $\Gamma=1 \rightrightarrows 0$ be the Kronecker quiver, where we denote the two edges by $h_\s$ and $h_\n$.
Let $\vecn = n_0\alpha_0+n_1\alpha_1$.
Consider the function $\m:H\to \{\s,\n\}$ given by $\m(h_\s) = \s$, $\m(h_\n) = \n$, and the corresponding convolution algebra $A^\m(\vecn)$.
More explicitly, it is the convolution algebra $H_*^{G_\vecn}(\bfZ^\m_\vecn,\bbk)$, where
\[
	\bfZ^\m_{\Ibe,\Iga}=\widetilde{\bfF}^\m_{\Ibe}\times_{E_\vecn}\widetilde{\bfF}^\m_{\Iga},\qquad \bfZ^\m_{\vecn}= \widetilde{\bfF}^\m_{\vecn}\times_{E_\vecn}\widetilde{\bfF}^\m_{\vecn} =\coprod_{\Ibe,\Iga\in \Comp(\vecn)}\bfZ^\m_{\Ibe,\Iga},
\]
and $\widetilde{\bfF}^\m_{\Ibe}\subset \widetilde{\bfF}^\s_{\Ibe}$ is the variety of pairs $(x,V^\bullet)\subset E_\vecn\times \bfF_\Ibe$ such that $h_\s(V^j_1)\subset V^j_0$ and $h_\n(V^j_1)\subset V^{j-1}_0$ for $1\leq j\leq \ell(\Ibe)$.
As in \cref{rmk:stack-notations}, we will sometimes write $\Fl^\m_\Ibe \coloneqq [\widetilde{\bfF}^\m_{\Ibe}/G_\vecn]$.

\begin{defn}
	We call $A^\m(\vecn)$ the \textit{seminilpotent quiver Schur algebra}.
\end{defn}

By reversing various arrows and using~\cref{lm:Fourier}, we get isomorphisms
\[
	A^\m_{1 \rightrightarrows 0}(\vecn) \simeq A^\s_{1 \rightleftarrows 0}(\vecn) \simeq A^\n_{1 \leftrightarrows 0}(\vecn).
\]
We define the idempotents, split, merge, and polynomial elements in $A^\m(\vecn)$ analogously to \cref{subs:qschur}, and they are preserved under the isomorphisms above (up to signs).
However, recall the polynomial representation
\[
	\APol^\star_\vecn = \sum_{\Ibe\in\Comp(\vecn)} \Pol_\Ibe = \sum_{\Ibe\in\Comp(\vecn)} \bbk[u_1,\ldots,u_{n_0},v_1,\ldots,v_{n_1}]^{\fkS_\Ibe}.
\]
While it does not depend on the choice of $\star$ as a vector space, the action of the generators of $A^\star(\vecn)$ will vary.

\begin{prop}\label{prop:poly-rep-quiver-Schur}
	Let $\Ibe = (\Ibe^1,\Ibe^2)\in \Comp(\vecn)$, and $\Ibe^1 = a\alpha_0+b\alpha_1$, $\Ibe^2 = c\alpha_0+d\alpha_1$.
	Denote
	\[
		K_1 = \prod_{i=1}^a\prod_{j=b+1}^{b+d}(u_i-v_j), \qquad K_2 = \prod_{i=a+1}^{a+c}\prod_{j=1}^{b}(v_j-u_i).
	\]
	The action of the split-merge operators $S_{\Ibe,\vecn}:\Pol_\vecn\to \Pol_\Ibe$, $M_{\vecn,\Ibe}:\Pol_\Ibe\to \Pol_\vecn$ on $\APol^\star_\vecn$ is given by the following table:
	\begin{center}
		\setlength{\tabcolsep}{12pt}
		\renewcommand{\arraystretch}{1.2}
		\begin{tabular}{ c | c c } 
			& \tikz[thick,xscale=.25,yscale=.15]{\dsplit{0}{1}{4}{3} \draw (2,0) -- (2,1); } & \tikz[thick,xscale=.25,yscale=.15]{\dmerge{0}{0}{4}{2}\draw (2,2) -- (2,3); } \\\hline
			$A^\s$ & $P\mapsto P$ & $P\mapsto (-1)^{ad+bc}\partial^u_{a,c}\partial^v_{b,d}(K_1K_2P)$\\
			$A^\m$ & $P\mapsto (K_1P)$ & $P\mapsto (-1)^{bc}\partial^u_{a,c}\partial^v_{b,d}(K_2P)$\\
			$A^\n$ & $P\mapsto (K_1K_2P)$ & $P\mapsto \partial^u_{a,c}\partial^v_{b,d}P$\\
		\end{tabular}
	\end{center}
\end{prop}
\begin{proof}
	The formulas for $A^\s$ and $A^\n$ can be found in~\cite[Th.~4.7]{Prz_QSAC2019} and~\cite[Prop.~3.4]{SW_QSAF2014} respectively.
	For $A^\m$, one obtains them completely analogously via equivariant localization.
\end{proof}

From this point on, we will only be interested in the Kronecker quiver and the seminilpotent version of Schur algebra.
We thus drop the superscripts and write $A(\vecn)$, $\widetilde{\bfF}_\Ibe$ and so on, and use the corresponding conventions for the diagrammatic calculus and the polynomial representation.

\subsection{Representations of the Kronecker quiver}\label{subs:reps-Kronecker-quiver}
Let us recall some basic facts about representations of the Kronecker quiver; see e.g.~\cite[App.~A]{MakMin_KLR2023}.
Let $\Root^+\subset \bbZ_{\geqslant 0}^I$ be the set of positive roots for $\widehat{\mathfrak{sl}}_2$.
There exists an indecomposable representation of $\Gamma$ of dimension $\vecn$ iff $\vecn\in\Root^+$.
Denote $\delta \coloneqq \alpha_0 + \alpha_1$.
The set $\Root^+$ breaks into real and imaginary roots:
\[
	\Root^+ = \Root^+_{\rm re} \sqcup \Root^+_{\rm im}, \qquad \Root^+_{\rm re} = \{ \alpha_i + k\delta : i\in I, k\geq 0 \},\quad \Root^+_{\rm im} = \{ k\delta : k\geq 1 \}.
\]
When $\vecn\in \Root^+_{\rm re}$, there exists a unique indecomposable representation of dimension $\vecn$; we denote it by $M_\vecn$.
When $\vecn = n\delta$ is an imaginary root, we consider the open substack $\Rep^\reg_{n\delta}\subset \Rep_{n\delta}$ of all representations which do not contain any $M_\vecn$, $\vecn\in \Root^+_{\rm re}$ as a summand.
It is well known that $\Rep^\reg_{n\delta}$ is isomorphic to the moduli stack $\cT_n(\bbP^1)$ of torsion coherent sheaves on $\bbP^1$ of length $n$, and the isomorphism is given by
\begin{equation}\label{eq:tor-iso-reg}
	\cT_n(\bbP^1)\xra{\sim}\Rep^\reg_{n\delta}\Gamma,\quad \cE\mapsto \left(\begin{tikzcd}H^0(\cE)\ar[r,shift left=.5ex,"\cdot y_1"]\ar[r,shift right=.5ex,"\cdot y_2"'] & H^0(\cE\otimes\cO_{\bbP^1}(1))\end{tikzcd}\right).
\end{equation}
Let $\infty\in\bbP^1$ be the point corresponding to the representation in $\Rep_\delta\simeq \cT_1(\bbP^1)$ with vanishing $h_\n$. 
Using the description above, define $\Rep^\reg_{n\delta^\bullet}\subset\Rep^\reg_{n\delta}$ to be the closed substack consisting of sheaves supported at $\infty$, and $\Rep^\reg_{n\delta^\circ}\subset\Rep^\reg_{n\delta}$ the open substack of sheaves supported away from $\infty$.
It is clear that
\[
	\Rep^\reg_{n\delta^\bullet}\simeq[\mathcal{N}_{\fkgl_n}/GL_n],\qquad \Rep^\reg_{n\delta^\circ}\simeq[\fkgl_{n}/GL_n],
\]
and we have a stratification
\begin{equation}\label{eq:repreg-str}
	\Rep^\reg_{n\delta} = \bigsqcup_{k=0}^n \Rep^\reg_{k\delta^\bullet}\times \Rep^\reg_{(n-k)\delta^\circ}.
\end{equation}

\medskip
\section{Ersatz parity sheaves}
\label{sec:ersatz-parity}
In this section we introduce a geometric framework, which allows us to easily prove polyheredity of convolution algebras.
Our construction takes inspiration from~\cite{JMW_PS2014,McN_RTGE2017}.

\subsection{Evenness theories}
Let $Y$ be a \textit{quotient stack}, that is $[Y'/G]$, where $G$ is an algebraic group, and $Y'$ a $G$-variety; we only require an existence of such presentation, and do not fix one.
Let $D^b(Y)$ denote the bounded constructible derived category of $\bbk$-sheaves on $Y$; see~\cite{bernstein2006equivariant} for its definition and properties.

\begin{defn}
\label{def:eveness}
An \emph{evenness theory} on $Y$ is a full subcategory $\Ev = \Ev(Y)\subset D^b(Y)$ such that 
\begin{itemize}
\item $\Ev(Y)$ is stable by even shifts, direct sums, taking direct summands and Verdier duality;
\item For any $\calF,\calG \in \Ev(Y)$ the $\Bbbk$-module $\Hom(\calF,\calG)$ is finite-dimensional, and $\Ext^{\rm odd}(\calF,\calG)=0$.
\end{itemize}

Given an evenness theory, we denote by $\Par = \Par(Y)$ the full subcategory of $D^b(Y)$ whose objects are direct sums of elements in $\Ev(Y)$ and $\Ev(Y)[1]$.
We call elements of $\Ev(Y)$ \emph{even}, elements of $\Ev(Y)[1]$ \emph{odd}, and elements of $\Par(Y)$ \emph{parity complexes}.
\end{defn}

\begin{rmk}
\label{rmk:parity-Krull-Schmidt}
$\Hom$-finiteness in the second assumption of \cref{def:eveness} implies that the category $\Par(Y)$ is Krull-Schmidt~\cite{Karoub}.
\end{rmk}

\begin{rmk}\label{lem:prod-of-Ev}
	Let $X$, $Y$ be two stacks, equipped with evenness theories $\Ev(X)$, $\Ev(Y)$.
	Then
	\[
		\Ev(X\times Y) \coloneqq \left\{ \text{ direct sums and summands of } \cF\boxtimes \cG , \cF\in \Ev(X),\cG\in \Ev(Y) \text{ }\right\}
	\]
	is an evenness theory on $X\times Y$.
\end{rmk}

Let $X=\coprod_{\lambda\in \Lambda}X_\lambda$ be a stratification into a finite number of locally closed substacks, such that
\begin{equation}\label{eq:strat-cond}
	\parbox{20em}{\centering There exists a partial order on $\Lambda$, such that \\$X_{\leqslant \lambda}\coloneqq \sqcup_{\mu\leqslant \lambda} X_\mu$ is closed in $X$ for all $\lambda\in \Lambda$.}
\end{equation}
Denote by $i_\lambda\colon X_\lambda\to X$ the inclusion maps, and fix an evenness theory on each stratum $X_\lambda$.

\begin{defn}
Let $\calF\in D^b(X)$. 
We say that $\calF$ is \emph{$*$-even} if $i_\lambda^*\calF\in \Ev(X_\lambda)$ for all $\lambda\in \Lambda$, and $\calF$ is \emph{$!$-even} if $i_\lambda^!\calF\in \Ev(X_\lambda)$ for all $\lambda\in \Lambda$.
\end{defn}

\begin{lem}
\label{lem:Ext-*!-even}
Let $\calF,\calG\in D^b(X)$.
Assume that $\calF$ is $*$-even and $\calG$ is $!$-even.
Then we have a (non-canonical) isomorphism of $\bbk$-vector spaces
$$
\Ext^*(\calF,\calG)\simeq \bigoplus_{\lambda\in \Lambda}\Ext^*(i_\lambda^*\calF,i_\lambda^!\calG).
$$ 
In particular, $\Ext^{\rm odd}(\calF,\calG)=0$.
\end{lem}
\begin{proof}
The proof is by induction on the number $k$ of strata with $i_\lambda^*\calF\ne 0$.
If $k=1$, we have $\calF= i_{\lambda!}i_\lambda^*\calF$ for some stratum $\lambda\in \Lambda$ and so by adjunction
$$
\Ext^*(\calF,\calG)=\Ext^*((i_{\lambda})_!i_\lambda^*\calF,\calG)\simeq \Ext^*(i_\lambda^*\calF,i_{\lambda}^!\calG).
$$

Now, assume that $k>1$.
By~\eqref{eq:strat-cond}, we can find an open union of strata $j:U\hookrightarrow X$ such that $j^*\calF\ne 0$.
We get a long exact sequence 
$$
\ldots\to \Ext^n(i_*i^*\calF,\calG)\to \Ext^n(\calF,\calG)\to \Ext^n(j_!j^!\calF,\calG)\to\ldots
$$
The complexes $i_*i^*\calF$ and $j_!j^!\calF$ are clearly $*$-even.
By induction assumption, all terms of the long exact sequence with odd $n$ vanish. Moreover, we have
$$
\Ext^*(i_*i^*\calF,\calG)\simeq \bigoplus_{X_\lambda\subset U}\Ext^*(i_\lambda^*\calF,i_\lambda^!\calG),\qquad \Ext^*(j_!j^!\calF,\calG)\simeq \bigoplus_{X_\lambda\subset X\backslash U}\Ext^*(i_\lambda^*\calF,i_\lambda^!\calG).
$$ 
This proves the statement.
\end{proof}

\begin{cor}\label{cor:gluing-evenness}
	The full subcategory $\Ev_\Lambda(X)\subset D^b(X)$ of complexes which are both $*$-even and $!$-even is an evenness theory on $X$.
\end{cor}

\begin{lem}
\label{lem:Par-open-surj}
Let $\calF,\calG\in \Par_\Lambda(X)$.
For any open union of strata $j:U\hookrightarrow X$, the map $\Ext^*(\calF,\calG)\to \Ext^*(j^*\calF,j^*\calG)$ is surjective.
In particular, if $\calF\in \Par_\Lambda(X)$ is indecomposable, then $j^*\calF\in \Par_\Lambda(U)$ is indecomposable (or zero).
\end{lem}
\begin{proof}
It is enough to assume $\calF,\calG\in\Ev_\Lambda(X)$.
Recall that $j^* = j^!$.
Consider a piece of long exact sequence:
\[
	\ldots\to \Ext^{2k}(\calF,\calG)\to \Ext^{2k}(j_!j^!\calF,\calG)\to\Ext^{2k+1}(i_*i^*\calF,\calG)\to\ldots
\]
Since $i_*i^*\calF\in \Ev_\Lambda(X)$, the last term vanishes, which proves the first claim.

For the second claim, let $\calF$ be indecomposable.
The ring $\End(\calF)$ is local by the Krull-Schmidt property, and surjects on $\End(j^*\calF)$.
Therefore $\End(j^*\calF)$ is also local (or zero), hence $j^*\calF$ is indecomposable.
\end{proof}

\begin{lem}
\label{lem:maxisupport-indecomp}
	Assume that $\calF\in \Par_\Lambda(X)$ is indecomposable. Then there exists a unique $\lambda\in \Lambda$ such that $i_\lambda^*\calF\ne 0$ and the support of $\calF$ is in $X_{\leqslant \lambda}$.
\end{lem}
\begin{proof}
	Let $X'$ be the smallest closed union of strata in $X$ containing the support of $\calF$; such $X'$ exists by~\eqref{eq:strat-cond}.
	We claim that $X'= X_{\leqslant \lambda}$ for some $\lambda\in \Lambda$.
	Indeed, assume $X'$ contains two open strata $X_{\mu_1}$, $X_{\mu_2}$.
	Then the restriction of $\calF$ to $X_{\mu_1} \sqcup X_{\mu_2}$ should be indecomposable by \cref{lem:Par-open-surj}, and so either $i^*_{\mu_1}\calF$ or $i^*_{\mu_2}\calF$ is zero.
\end{proof}

\begin{lem}
	\label{lem:par-unique}
	Let $\lambda\in \Lambda$, and $\calF,\calG\in \Par_\Lambda(X)$ two indecomposable complexes supported on $X_{\leqslant\lambda}$, with $i^*_\lambda\calF\simeq i^*_\lambda\calG$.
	Then $\calF\simeq \calG$.
	\end{lem}
	\begin{proof}
	By \cref{lem:Par-open-surj}, the complex $i^*_\lambda\calF\simeq i^*_\lambda\calG$ is indecomposable; denote it by $\cP$.
	The ring $\End(\cP)$ is local by \cref{rmk:parity-Krull-Schmidt}. Denote its residue field by $\bbF$, and the open inclusion $X_\lambda\hookrightarrow X_{\leqslant\lambda}$ by $j_\lambda$.
	Again by \cref{lem:Par-open-surj}, we have surjections 
	\[
		\Hom(\calF,\calG)\twoheadrightarrow \Hom(j^*\calF,j^*\calG)=\End(\cP)\twoheadrightarrow \bbF.
	\]
	Thus we can lift $1\in\bbF$ to a morphism $\phi\colon \calF\to \calG$. Exchanging the roles of $\calF$ and $\calG$, we get a morphism $\psi\colon\calG\to \calF$.
	Then $\psi\circ\phi$ is invertible in $\End(\calF)$ and $\phi\circ\psi$ is invertible in $\End(\calG)$, so $\phi$ is an isomorphism.
	\end{proof}

\begin{defn}
\label{def:strong-even}
We say that an evenness theory $\Ev(Y)$ on $Y$ is \emph{finitary} if the indecomposables in $\Par(Y)$ up to shift are labelled by a finite set $\Mu$.
A finitary evenness theory is \emph{balanced} if each class $\mu\in \Mu$ contains a Verdier self-dual representative $\parsh(\mu)$.
\end{defn}

Let $X=\coprod_{\lambda\in \Lambda}X_\lambda$ as before, and assume that the evenness theory on each $X_\lambda$ is finitary.
Denote by $\Mu_\lambda$ the labelling set of isomorphism classes (up to a shift) of indecomposable objects in $\Par(X_\lambda)$.
For $\mu\in \Mu_\lambda$, let $\parsh_\lambda(\mu)$ be the corresponding indecomposable parity sheaf with an arbitrary shift; our convention is that for a balanced evenness theory we pick the Verdier self-dual one.  

\begin{defn}
	Let $\lambda\in\Lambda$, $\mu\in \Mu_\lambda$.
	We call an indecomposable sheaf $\calF\in \Par_\Lambda(X)$ satisfying $i^*_\lambda\calF\simeq \parsh_\lambda(\mu)$ an \textit{ersatz parity sheaf}, and denote it by $\parsh(\lambda,\mu) = \parsh^X(\lambda,\mu)$.
\end{defn}

\cref{lem:par-unique} shows that if an ersatz parity sheaf $\parsh(\lambda,\mu)$ exists, it is unique.
Note that each $\parsh(\lambda,\mu)$ can be either even or odd.

\begin{defn}
	We say that the evenness theory $\Ev_\Lambda$ is \textit{ersatz-complete} if all ersatz parity sheaves $\parsh(\lambda,\mu)$ exist. 
\end{defn}

By \cref{lem:maxisupport-indecomp}, every indecomposable in $\Par_\Lambda(X)$ must be of the form $\parsh(\lambda,\mu)$ up to a shift.
Moreover, if $\Ev(X_\lambda)$ is balanced, the complex $\parsh(\lambda,\mu)$ is automatically Verdier self-dual.
We obtain

\begin{prop}\label{prop:gluing-fin-bal}
	Let $X = \bigsqcup_{\lambda\in \Lambda} X_\lambda$ be a stratification satisfying~\eqref{eq:strat-cond} with a finitary evenness theory on each stratum.
	Assume that $\Ev_\Lambda(X)$ is ersatz-complete.
	Then $\Ev_\Lambda(X)$ is finitary with labelling set $\Lambda_\Mu \coloneqq \{ (\lambda,\mu) : \lambda\in \Lambda, \mu \in \Mu_\lambda \}$.
	If each $\Ev(X_\lambda)$ is balanced, then $\Ev_\Lambda(X)$ is balanced as well.
\end{prop}

\subsection{Proper stratifications}
Let $Y$ be a stack with a finitary evenness theory, whose labelling set we denote by $\Mu$.
Fix a finite totally ordered set $(\Xi,\leqslant)$, together with a map $\varrho : \Mu\to \Xi$.
The map $\varrho$ induces a preorder on $\Mu$ with $\nu\leqslant \mu$ if and only if $\varrho(\nu)\leqslant \varrho(\mu)$.

Given $\xi\in\Xi$, we denote $\parsh(\xi) \coloneqq \bigoplus_{\varrho(\mu) = \xi}\parsh(\mu)$.
For each $\calF,\calG\in \Par(Y)$, consider the following subspaces of $\Ext^*(\calF,\calG)$:
\begin{align*}
	\Ext^*_{\leqslant\xi}(\calF,\calG) & = \sum\nolimits_{\zeta\leqslant \xi} \Ext^*(\calF,\parsh(\zeta))\circ \Ext^*(\parsh(\zeta),\calG)\subset \Ext^*(\calF,\calG),\\
	\Ext^*_{<\xi}(\calF,\calG) & = \sum\nolimits_{\zeta< \xi} \Ext^*(\calF,\parsh(\zeta))\circ \Ext^*(\parsh(\zeta),\calG)\subset \Ext^*(\calF,\calG).
\end{align*}
We further denote $\Ext^*_{\xi}(\calF,\calG)=\Ext^*_{\leqslant\xi}(\calF,\calG)/\Ext^*_{<\xi}(\calF,\calG)$.

Let $\scrB$ be a class of Noetherian Laurentian $\bbk$-algebras.
It is known~\cite[Lem.~2.6]{Kle_AHWC2015} that all such algebras are semiperfect.

\begin{defn}\label{defn-quasi-hered-Par}
Let $\Ev(Y)$ be a finitary evenness theory on $Y$.
A map $\varrho:\Mu\to \Xi$ is \emph{$\scrB$-properly stratifying} if the following conditions are satisfied for all $\xi\in \Xi$:

\begin{enumerate}[label=(\arabic*)]
\item\label{eq:pspar1} $B_{\xi}\coloneqq\Ext^*_{\xi}(\parsh(\xi),\parsh(\xi))^\op$ belongs to $\scrB$;
\item\label{eq:pspar2} For each $\calF\in \Par(Y)$, the right $B_\xi$-module $\Ext^*_{\xi}(\calF,\parsh(\xi))$ is finitely generated and flat;
\item\label{eq:pspar3} For each $\calF,\calG\in \Par(Y)$, the canonical map 
\begin{equation}\label{eq:propstr-cond3}
	\Ext^*_{\xi}(\calF,\parsh(\xi))\otimes_{B_\xi}\Ext^*_{\xi}(\parsh(\xi),\calG)\to \Ext^*_\xi(\calF,\calG)	
\end{equation}
is an isomorphism of vector spaces.
\end{enumerate}
We say that $\Ev(Y)$ is \emph{$\scrB$-properly stratified} if it admits a $\scrB$-properly stratifying preorder.
\end{defn}

\begin{expl}
	If $|\Xi| = 1$, the only non-trivial condition in \cref{defn-quasi-hered-Par} is that $\Ext^*(\parsh,\parsh)^\op\in\scrB$, where $\parsh = \bigoplus_{\mu\in\Mu} \parsh(\mu)$.
\end{expl}

Let us recall a parallel notion of $\scrB$-properly stratified algebras.
\begin{defn}[{\cite{Kle_AHWC2015}}]
Let $A$ be a Noetherian Laurentian graded unital $\bbk$-algebra.
A two-sided homogeneous ideal $J \subset A$ is \textit{$\scrB$-properly stratifying} if it satisfies the
following conditions:
\begin{enumerate}[label=(\roman*)]
	\item\label{eq:psalg1} $\Hom_A(J,A/J)=0$;
	\item\label{eq:psalg2} $J$ is projective as a left $A$-module, and $\End_A(J)^\op$ is Morita-equivalent to an algebra in $\scrB$;
	\item\label{eq:psalg3} $J$ is finitely generated and flat as a right $\End_A(J)^\op$-module.
\end{enumerate}

The algebra $A$ is \emph{$\scrB$-properly stratified} if there exists a finite
chain of two-sided ideals $\{0\}=J_m\subset J_{m-1}\subset \ldots\subset J_1\subset J_0=A$ such that $J_r/J_{r+1}$ is $\scrB$-properly stratifying in $A/J_{r+1}$ for all $0\leqslant r<m$.
\end{defn}

We say that $\calL\in\Par(Y)$ is \emph{full} if it contains all indecomposables $\parsh(\mu)$, $\mu\in \Mu$ as direct factors up to a shift. 

\begin{prop}\label{lem:Ev-to-alg-qher}
	Let $\Ev(Y)$ be a $\scrB$-properly stratified evenness theory.
	The algebra $A_\calL\coloneqq\Ext^*(\calL,\calL)^\op$ is $\scrB$-properly stratified for any full sheaf $\calL\in \Par(Y)$.
\end{prop}

\begin{proof}
	Let us first check that $A_\calL$ is Noetherian Laurentian.
	Laurentian condition follows from finite-dimensionality of $\Hom$-spaces in $\Ev(Y)$.
	Recall that $Y$ is assumed to be a quotient stack $[Y'/G]$.
	Denote $\Ext$-functors in the non-equivariant category $D^b(Y')$ by $\Ext^*_{Y'}$.
	By Serre spectral sequence 
	\[
		H^*_G(\pt)\otimes_\bbk \Ext^*_{Y'}(\calL,\calL) \Rightarrow \Ext^*(\calL,\calL),
	\]
	the algebra $\Ext^*(\calL,\calL)$ is finitely generated over a quotient of $H^*_G(\pt)$.
	Since $H^*_G(\pt)$ is Noetherian, so is $A_\calL$.

	Fix a properly stratifying $\varrho:\Mu\to \Xi$. 
	For each $\xi\in \Xi$, consider the ideal $J_{\leqslant \xi}=\Ext^*_{\leqslant \xi}(\calL,\calL)$. 
	It suffices to show that $J_{\leqslant \xi}/J_{<\xi}$ is $\scrB$-properly stratifying in $A/J_{<\xi}$ for all $\xi$.

	Let us write $A = A_\calL$, $J=J_{\leqslant\xi}$, $A' = A/J_{<\xi}$.
	Given a graded $A$-module $N$, denote the corresponding graded $A'$-module $N/J_{<\xi}N$ by $N'$; in particular, $J'=J_{\leqslant\xi}/J_{<\xi}$. 
	Let $e\in A$ be an idempotent projecting to a direct summand of $\calL$ isomorphic to a shift of $\parsh(\xi)$.
	By abuse of notation, we denote the image of $e$ in $A'$ by $e$ as well.
	We have $J=AeA+J'$, and so $J'=A'eA'$.
	
	It is a standard fact (see e.g.~\cite[Th.~2.14]{Mak_CBKA2015} for the proof of a similar statement) that the functor $\bfE=\Ext^*(\calL,-)$ induces an equivalence between $\Par(Y)$ and the category of finitely generated projective (left) $A$-modules.
	In particular, for any $\calF,\calG\in \Par(Y)$ we have an isomorphism of graded $\bbk$-modules 
	\[
		\Ext^*(\calF,\calG)\simeq \Hom^*_A(\bfE(\calF),\bfE(\calG)),
	\]
	where $\Hom^i_A(-,-) \coloneqq \Hom_A(-,-[i])$.
	By the definition of $J'$, we also have 
	\begin{equation}
	\label{eq:Hom-F'}
	\Hom_{A'}(\bfE(\calF)',\bfE(\calG)')\simeq \Ext^*(\calF,\calG)/\Ext^*_{<\mu}(\calF,\calG).
	\end{equation}
	
	Consider $P_\xi \coloneqq \bfE(\parsh(\xi))$.
	We have $P_\xi\simeq Ae$, $P'_\xi=A'e$ and $\End_{A'}(P_\xi')\simeq eA'e$.
	Applying~\ref{eq:pspar3} to $\calF=\calG=\calL$ and using~\eqref{eq:Hom-F'}, we see that $A'e\otimes_{eA'e}eA'\simeq A'eA'$. This shows that $J'=A'eA'$ is a direct sum of (shifted) copies of $P_\xi'=A'e$.
	In particular, $J'$ is projective, and $\End_{A'}(J')^\op$ is Morita-equivalent to
	\[
		\End_{A'}(P_\xi')^\op\simeq \Ext^*_{\xi}(\parsh(\xi))^\op=B_\xi.
	\]
	Since $B_\xi\in \mathscr{B}$ by~\ref{eq:pspar1}, we verified~\ref{eq:psalg2}.

	Let us write $J' = P_\xi'\otimes_\bbk V$, with $V$ graded vector space.
	Then 
	\begin{align*}
		\Hom_{A'}(J',A'/J') & = \Hom_{A'}(P_\xi',A'/J')\otimes V = \Hom_{A'}(A'e,A'/A'eA')\otimes V \\&= e(A'/A'eA')\otimes V=0,
	\end{align*}
	which proves~\ref{eq:psalg1}.

	Finally, the right $\End_{A'}(P_\xi')^\op$-module $P_\xi'$ is the same as the right $B_\xi$-module $\Ext^*_\xi(\calL,\parsh(\xi))$, so it is flat and finitely generated over $\End_{A'}(P_\xi')^\op$ by~\ref{eq:pspar2}.
	Since these properties are preserved by Morita-equivalence, this proves~\ref{eq:psalg3}; and so we may conclude.
\end{proof}

\subsection{Gluing stratifications}
Let us consider $X=\coprod_{\lambda\in \Lambda}X_\lambda$ with a finitary evenness theory $\Ev(X_\lambda)$ for each $X_\lambda$.
Denote the labelling set of $\Ev(X_\lambda)$ by $\Mu_\lambda$.
Assume that we have a $\scrB$-properly stratifying map $\varrho_\lambda: \Mu_\lambda\to \Xi_\lambda$ for each $\lambda$.
Consider an arbitrary refinement of the order from~\eqref{eq:strat-cond} to a total order on $\Lambda$, and equip
\[
	\Xi\coloneqq \bigsqcup_\lambda \Xi_\lambda = \{ (\lambda,\xi) : \xi\in\Xi_\lambda\}
\]
with the lexicographic order.
Let $\Mu_\Lambda = \bigsqcup_\lambda \Mu_\lambda$, and denote the disjoint union of maps $\varrho_\lambda$ by $\varrho:\Mu\to \Xi$.
The following is our main technical result.
\begin{thm}\label{thm:quasi-her-inher}
If the induced evenness theory $\Ev_\Lambda(X)$ on $X$ is ersatz-complete, then it is also $\scrB$-properly stratifying with respect to $\varrho$.
\end{thm}

We will need some preparation before proceeding with the proof.
Denote by $\lambda_0$ the minimal element of $\Lambda$, and by $\xi_0$ the maximal element of $\Xi_{\lambda_0}$.
Consider the open embedding $j:U=X\backslash X_{\lambda_0}\hookrightarrow X$ and its closed complement $i:X_{\lambda_0}\hookrightarrow X$.

\begin{lem}
	For $\calF,\calG\in \Par_\Lambda(X)$, there is a short exact sequence
	$$
	0\to \Ext^*_{\leqslant\xi_0}(\calF,\calG)\to  \Ext^*(\calF,\calG) \to  \Ext^*(j^*\calF,j^*\calG)\to 0.
	$$
\end{lem}
\begin{proof}
	Analogously to the proof of \cref{lem:Par-open-surj}, we have a short exact sequence
	$$
	0\to \Ext^*(\calF,i_*i^!\calG)\to \Ext^*(\calF,\calG)\to \Ext^*(\calF,j_*j^*\calG)\to 0.
	$$
	We need to show that the image of $\Ext^*(\calF,i_!i^!\calG)$ in $\Ext^*(\calF,\calG)$ is $\Ext^*_{\leqslant\xi_0}(\calF,\calG)$.
	Since $i^!\calG\in \Par(X_{\lambda_0})$, this image clearly belongs to $\Ext^*_{\leqslant\xi_0}(\calF,\calG)$.
	On the other hand, each element of $\Ext^*_{\leqslant\xi_0}(\calF,\calG)$ vanishes under the map $\Ext^*(\calF,\calG)\to \Ext^*(j^*\calF,j^*\calG)$, so it has to be in the image of $\Ext^*(\calF,i_!i^!\calG)$.
\end{proof}

\begin{lem}
\label{lem:Ext-lam-mu-not-lam0}
Let $\lambda\ne\lambda_0$, and $\xi\in\Xi_\lambda$.
For any $\calF,\calG\in \Par_\Lambda(X)$, we have $\Ext^*_{\xi}(\calF,\calG)\simeq \Ext^*_{\xi}(j^*\calF,j^*\calG)$. 
\end{lem}
\begin{proof}
By the lemma above, we have $\Ext^*(j^*\calF,j^*\calG)\simeq \Ext^*(\calF,\calG)/\Ext^*_{\leqslant\xi}(\calF,\calG)$. 
For any $(\lambda',\mu')\in \Lambda_M$ with $\lambda'\ne \lambda_0$, consider the ersatz parity sheaf $\parsh^X(\lambda',\mu')$, resp. $\parsh^U(\lambda',\mu')$ on $X$, resp. $U$.
By \cref{lem:Par-open-surj} we have $j^*\parsh^X(\lambda',\mu')=\parsh^U(\lambda',\mu')$.
This implies that for any $\xi\in\Xi_\lambda$, $\lambda\ne \lambda_0$, the subspace $\Ext^*_{\leqslant \xi}(j^*\calF,j^*\calG)\subset \Ext^*(j^*\calF,j^*\calG)$ is the image of $\Ext^*_{\leqslant \xi}(\calF,\calG)\subset \Ext^*(\calF,\calG)$.
The claim follows. 
\end{proof}

\begin{lem}
	\label{lem:Ext-factors}
	Let $\calF,\calG\in\Par_\Lambda(X)$.
	\begin{enumerate}
		\item If $\calG\in i_*\Par(X_{\lambda_0})$, the natural map $\Ext^*(i_*i^*\calF,\calG)\to \Ext^*(\calF,\calG)$ is an isomorphism;
		\item If $\calF\in i_*\Par(X_{\lambda_0})$, the natural map $\Ext^*(\calF,\calG)\to \Ext^*(\calF,i_!i^!\calG)$ is an isomorphism;
		\item The isomorphisms in (1) and (2) restrict to isomorphisms on $\Ext^*_{\leqslant \xi}$ for any $\xi\in \Xi_{\lambda_0}$. 
		In particular, we get $\Ext^*_{\xi}(\calF,\calG)\simeq \Ext^*_{\xi}(i_*i^*\calF,\calG)$ in (1) and $\Ext^*_{\xi}(\calF,\calG)\simeq \Ext^*_{\xi}(\calF,i_!i^!\calG)$ in (2).
	\end{enumerate}
\end{lem}
\begin{proof}
	The first two statements follow from long exact sequences as in \cref{lem:Ext-*!-even} together with $j^*i_* = 0$.
	Let us prove the last statement for the isomorphism in (1), for (2) the proof is analogous. 
	The inclusion $\Ext^*_{\leqslant \xi}(i_*i^*\calF,\calG)\subset \Ext^*_{\leqslant \xi}(\calF,\calG)$ is immediate.
	For the opposite direction, note that any extension in $\Ext^*(\calF,\parsh(\xi))$ must come from $\Ext^*(i_*i^*\calF,\parsh(\xi))$ by (1).
\end{proof}

\begin{lem}
	\label{lem:Ext-factors2}
	Let $\calF,\calG\in\Par_\Lambda(X)$.
	The canonical map $\Ext^*(i_*i^*\calF,i_!i^!\calG)\to\Ext^*(\calF,\calG)$ induces an isomorphism $\Ext^*_{\leqslant\xi}(i_*i^*\calF,i_!i^!\calG)\simeq \Ext^*_{\leqslant\xi}(\calF,\calG)$ for every $\xi\in \Xi_{\lambda_0}$.
	In particular, $\Ext^*_{\xi}(i_*i^*\calF,i_!i^!\calG)\simeq \Ext^*_{\xi}(\calF,\calG)$.
\end{lem}

\begin{proof}
First, note that we have $\Ext^*_{\leqslant\xi}(i_*i^*\calF,\calG)\simeq \Ext^*_{\leqslant\xi}(i_*i^*\calF,i_!i^!\calG)$ by \cref{lem:Ext-factors}.
It remains to prove $\Ext^*_{\leqslant\xi}(i_*i^*\calF,\calG)\simeq \Ext^*_{\leqslant\xi}(\calF,\calG)$.
Since $\calF,\calG,i_*i^*\calF\in\Par_\Lambda(X)$, we have a short exact sequence 
$$
0\to \Ext^*(i_*i^*\calF,\calG)\to \Ext^*(\calF,\calG)\to  \Ext^*(j_!j^!\calF,\calG)\to 0.
$$
This implies the inclusion $\Ext^*_{\leqslant\xi}(i_*i^*\calF,\calG)\subset \Ext^*_{\leqslant\xi}(\calF,\calG)$.
The surjectivy follows from \cref{lem:Ext-factors}(1) applied to $\calG = \parsh(\xi)$.
\end{proof}

\begin{proof}[Proof of \cref{thm:quasi-her-inher}]	
We reason by induction on the number of elements in $\Lambda$.
The case of one element is trivial. 
By the induction assumption, the evenness theory on $U$ induced by the stratification $U=\bigsqcup_{\lambda\in\Lambda\backslash \{\lambda_0\}}X_\lambda$ is $\scrB$-properly stratified.
If $\lambda\ne \lambda_0$, all conditions in \cref{defn-quasi-hered-Par} follow immediately from \cref{lem:Ext-lam-mu-not-lam0}.
From now on, assume $\lambda=\lambda_0$ and $\xi\in\Xi_{\lambda_0}$.
Let us write 
\[
	\parsh_{\lambda_0}(\xi) = \bigoplus_{\mu\in\varrho^{-1}(\xi)}\parsh_{\lambda_0}(\mu)\in \Par(X_{\lambda_0}).
\]
For condition~\ref{eq:pspar1}, we have 
$$
B_{\xi}=\Ext^*_{\xi}(\parsh(\xi),\parsh(\xi))^\op= \Ext_\xi(\parsh_{\lambda_0}(\xi),\parsh_{\lambda_0}(\xi))^\op,
$$	
and so $B_{\xi}\in \scrB$ by proper stratification on $X_{\lambda_0}$.
For condition~\ref{eq:pspar2}, let $\calF\in\Par_\Lambda(X)$.
Then by \cref{lem:Ext-factors} we have 
$$
\Ext^*_{\xi}(\parsh(\xi),\calF)\simeq \Ext^*_{\xi}(\parsh(\xi),i_!i^!\calF)=\Ext^*_{\xi}(\parsh_{\lambda_0}(\xi),i^!\calF),
$$
and so $\Ext^*_{\xi}(\parsh(\xi),\calF)$ is a flat $B_{\xi}$-module by proper stratification on $X_{\lambda_0}$.
Finally, let us consider the map
$$
\Ext^*_{\xi}(\calF,\parsh(\xi))\otimes_{B_{\xi}}\Ext^*_{\xi}(\parsh(\xi),\calG)\to \Ext^*_{\xi}(\calF,\calG)
$$
for $\calF,\calG\in \Par_\Lambda(X)$. 
By \cref{lem:Ext-factors} we have
\begin{gather*}
	\Ext^*_{\xi}(\calF,\parsh(\xi))\simeq \Ext^*_{\xi}(i_*i^*\calF,\parsh(\xi))=\Ext^*_{\xi}(i^*\calF,\parsh_{\lambda_0}(\xi)),\\
	\Ext^*_{\xi}(\parsh(\xi),\calG)\simeq \Ext^*_{\xi}(\parsh(\xi),i_!i^!\calG)= \Ext^*_{\xi}(\parsh_{\lambda_0}(\xi),i^!\calG).
\end{gather*}
Moreover, by \cref{lem:Ext-factors2} we have  
$$
\Ext^*_{\xi}(\calF,\calG)\simeq \Ext^*_{\xi}(i_*i^*\calF,i_!i^!\calG)=\Ext^*_{\xi}(i^*\calF,i^!\calG).
$$
Therefore~\ref{eq:pspar3} follows from the proper stratification on $X_{\lambda_0}$.
\end{proof}

\begin{rmk}
	Ersatz-completeness was used exactly once, in the proof of \cref{lem:Ext-lam-mu-not-lam0}, which is our crucial inductive step.
\end{rmk}

\begin{rmk}\label{rmk:exotic-sheaves}
	\cref{def:eveness} makes sense for any triangulated category $\mathcal{C}$.
	Most of our statements are essentially properties of recollements.
	We only use $\mathcal{C} = D^b(Y)$ for Krull-Schmidt property in \cref{rmk:parity-Krull-Schmidt}, and Noetherianity in \cref{lem:Ev-to-alg-qher}.
	This suggests that one can use analogous arguments to deduce $\scrB$-properly stratified structure of $\Ext$-algebras in more exotic categories of sheaves. 
\end{rmk}

\subsection{Polyheredity}\label{subs:polyher}
Let us specialize our discussion to the particular case when $\scrB = \scrP$ is the class of (finitely generated positively graded) polynomial algebras, and $\varrho$ is the identity map.
In the literature, one says \textit{polynomial quasihereditary} instead of $\scrP$-properly stratified, but we will use the term \textit{polyhereditary} for brevity.

\begin{defn}
	Let $\Ev(Y)$ be a finitary evenness theory on $Y$ with labelling set $\Mu$.
	A total order on $\Mu$ is \emph{polyhereditary} if the following conditions are satisfied for all $\mu\in \Mu$:
	
	\begin{enumerate}[label=(\arabic*)]
	\item $B_{\mu}\coloneqq\Ext^*_{\mu}(\parsh(\mu),\parsh(\mu))^\op$ is a polynomial algebra;
	\item For each $\calF\in \Par(Y)$, $\Ext^*_{\mu}(\calF,\parsh(\mu))$ is free of finite rank over $B_\mu$;
	\item For each $\calF,\calG\in \Par(Y)$, the map \eqref{eq:propstr-cond3} is an isomorphism.
	\end{enumerate}
	$\Ev(Y)$ is \emph{polyhereditary} if it admits a polyhereditary order.
\end{defn}

\begin{defn}
	Let $A$ be a Noetherian Laurentian graded unital $\bbk$-algebra.
	A two-sided homogeneous ideal $J \subset A$ is \textit{polyheredity} if:
	\begin{enumerate}[label=(\roman*)]
		\item $\Hom_A(J,A/J)=0$;
		\item $J$ is a direct sum of (shifted) copies of an indecomposable projective $P$ as a left $A$-module, and $\End_A(P)^\op$ is a polynomial algebra;
		\item $P$ is free finite rank over $\End_A(P)^\op$.
	\end{enumerate}
	The algebra $A$ is \emph{polyhereditary} if it admits a chain of polyheredity ideals.
\end{defn}

The following is an immediate consequence of \cref{lem:Ev-to-alg-qher,thm:quasi-her-inher}.
\begin{cor}\label{cor:polyher-case}
	\begin{enumerate}
		\item Let $\Ev(Y)$ be a polyhereditary evenness theory.
		The algebra $\Ext^*(\calL,\calL)$ is polyhereditary for any full sheaf $\calL\in \Par(Y)$.
		\item Let $X=\coprod_{\lambda\in \Lambda}X_\lambda$ with a polyhereditary evenness theory $\Ev(X_\lambda)$ for each $X_\lambda$.
		If $\Ev_\Lambda(X)$ is ersatz-complete, then it is polyhereditary.
	\end{enumerate}	
\end{cor}

\subsection{Comparison with usual parity sheaves}\label{rmk:classical-parity}
Let $Y$ be a smooth stack with $H^{\rm odd}(Y,\bbk) = 0$. 
Consider the full subcategory $\Ev^{\mathrm{c}}(Y) \subset D^b(Y)$ of direct sums of even shifts of the constant sheaf $\bbk_Y$.
One easily checks that this is an evenness theory.

Next, let $X = \bigsqcup_{\lambda\in \Lambda} X_\lambda$ be a stratification satisfying~\eqref{eq:strat-cond}, such that $H^{\rm odd}(X_\lambda,\bbk) = 0$ for all $\lambda$. 
By \cref{cor:gluing-evenness}, we can glue the evenness theories $\Ev^{\mathrm{c}}(X_\lambda)$ to an evenness theory $\Ev^{\mathrm{pc}}_\Lambda$ on $X$; we call it the \textit{piecewise constant} evenness theory. 
If $\pi_1(X_\lambda) = 0$ for all $\lambda\in \Lambda$, then $\Ev^{\mathrm{pc}}_\Lambda$ is precisely the category of even complexes of zero pariversity in the terminology of~\cite[Def.~2.4]{JMW_PS2014}, 
and ersatz parity sheaves are parity sheaves.

Let $G$ be an algebraic group, and $Y$ a $G$-variety with finitely many orbits $Y_\lambda$.
Consider
\[
	X = [Y/G] = \bigsqcup_{\lambda\in\Lambda} X_\lambda,\qquad X_\lambda = [Y_\lambda/G].
\]
Assume that for any point $y\in Y$ the stabilizer $G_y$ is connected, and $H^{\rm odd}_{G_y}(\pt) = 0$.
This is equivalent to requiring that $\pi_1(X_\lambda) = 0$ and $H^{\rm odd}(X_\lambda) = 0$.
If all parity sheaves on $X$ exist, and $H^*_{G_y}(\pt)$ is a polynomial algebra for all $y\in Y$, \cref{cor:polyher-case} recovers a result of McNamara.

\begin{cor}[{\cite[Th.~4.7]{McN_RTGE2017}}]\label{cor:McNa-result}
	Under the assumptions above, the graded algebra $\Ext^*(\cL,\cL)$ is polyhereditary for any full sheaf $\cL\in \Par^{\mathrm{pc}}_\Lambda(X)$.
\end{cor}
When $G$ is trivial, the existence of parity sheaves is known~\cite[Cor.~2.28]{JMW_PS2014}.
In general, the main method of constructing parity sheaves requires finding equivariant even resolutions of closures of the strata, see~\cite[Cor.~2.35]{JMW_PS2014}.

\begin{expl}\label{expl:compare-hw-orders}
	For each $a,b\in \bbZ$, consider the action of $T = \bbC^*$ on $\bbC^2 = \Spec \bbC[x,y]$, which scales coordinate $x$, resp. $y$ with weight $a$, resp. $b$, and denote by $[\bbC^2/_{(a,b)}T]$ the corresponding quotient stack.
	The constant sheaves
	\[
		\bbk_0\coloneqq \bbk_{\{0\}\times \{0\}},\quad \bbk_x \coloneqq \bbk_{\bbC\times \{0\}}, \quad \bbk_y \coloneqq \bbk_{\{0\} \times \bbC},\quad \bbk_{xy}\coloneqq \bbk_{\bbC^2}
	\]
	are $T$-equivariant for any values of $a,b$.

	Consider $Y_\n = \{(x,y): xy = 0\}\subset [\bbC^2/_{(-1,1)}T]$ and $Y_\m = [\bbC^2/_{(1,1)}T]$, together with the following stratifications:
	\[
	\tikz[xscale=.3,yscale=.3,font=\footnotesize]{
		\draw [fill=white,dotted] (-3,-3) rectangle (3,3);
		\draw [thick] (-3,0) -- (3,0);
		\draw [thick] (0,-3) -- (0,3);
		\fill (0,0) circle (7pt);

		\ntxt{0}{-4.2}{$Y_{\n} = (\{0\}\times \{0\}) \sqcup (\mathbb{C}^*\times \{0\}) \sqcup (\{0\}\times \mathbb{C}^*),$}
		\ntxt{19}{-4.2}{$Y_{\m} = (\{0\}\times \{0\}) \sqcup (\mathbb{C}^*\times \{0\}) \sqcup (\mathbb{C} \times \mathbb{C}^*).$}

		\draw [fill=lightgray,dotted] (16,-3) rectangle (22,3);
		\draw [thick] (16,0) -- (22,0);
		\fill (19,0) circle (7pt);
	}
	\]
	Consider evenness theory $\Ev^{\mathrm{pc}}_\Lambda$ associated to each stratification.
	In both cases all ersatz parity sheaves exist; namely, $\{\bbk_0, \bbk_x, \bbk_y\}$ for $Y_\n$ and $\{\bbk_0, \bbk_x, \bbk_{xy}\}$ for $Y_\m$.
	Moreover, Fourier-Sato transform $\Theta$ along the first coordinate of $\bbC^2$ exchanges the two sets (up to shifts):
	\[
		\bbk_0 \leftrightarrow \bbk_x,\qquad \bbk_x \leftrightarrow \bbk_0,\qquad \bbk_y\leftrightarrow \bbk_{xy}.
	\]
	Let $\cL = \bbk_0\oplus\bbk_x\oplus\bbk_{xy} \in \Ev^{\mathrm{pc}}_\Lambda(Y_\m)$, and consider the algebra $A = \Ext^*(\cL,\cL)$.
	It acquires two polyhereditary structures; one by \cref{lem:Ev-to-alg-qher} applied to $\cL$, and another by~\cref{cor:McNa-result} applied to $\cL' = \bbk_x\oplus\bbk_0\oplus\bbk_y \in \Ev^{\mathrm{pc}}_\Lambda(Y_\n)$, and then transported via $\Theta$.
	In the first case, the polyheredity order is $0<x<xy$; denote it by $\Omega_\m$.
	In the second case, we can take any total order refining the partial order $0<x, 0<y$ in terms of $\cL'$, which translates to $x<0, x<xy$ for $\cL$; denote the latter partial order by $\Omega_\n$.

	The algebra $A$ is linear over $H^*_T(\pt) = \bbk[u]$, and can be presented as a quiver with relations: 
	\begin{center}
	\begin{tabular}{rc}
	$
	\begin{tikzcd}
		&[-10pt] e_\delta\ar[dl,shift left=0.3ex,"s_1"]\ar[dr,shift left=0.3ex,"s_2"] &[-10pt] \\
		e_{10}\ar[ur,shift left=0.3ex,"m_1"]\ar[rr,shift left=0.3ex,"r_1"] & & e_{01}\ar[ul,shift left=0.3ex,"m_2"]\ar[ll,shift left=0.3ex,"r_2"]
	\end{tikzcd}
	$ & 
	\begin{tabular}{c}
		$s_1m_1 = ue_{10}$, $s_2m_2 = ue_{01}$, $m_1s_1 = m_2s_2 = ue_\delta$,\\
		$r_2r_1 = u^2e_{10}$, $r_1r_2 = u^2e_{01}$,\\
		$s_1m_2 = r_2$, $s_2m_1 = r_1$.
	\end{tabular}
	\end{tabular}
	\end{center}
	It follows that the first quotient in order $\Omega_\n$ is $A/Ae_\delta A \simeq \bbk[u]e_{10} \oplus \bbk[u]e_{01}$, whereas the first quotient in $\Omega_\m$ is $A/Ae_{10} A \simeq Z^e_{A_1}[u]$; here $Z^e_{A_1}$ is the extended zigzag algebra of type $A_1$, see~\eqref{ex:zigzag-A1-ext} for the definition.
\end{expl}

The example above essentially compares the stratifications of $A^\n(\delta)$ and $A^\m(\delta)$, up to killing off one equivariant parameter for simplicity.
The upshot is that the polyheredity structure arising from cyclic orientation $1\leftrightarrows 0$ does not have the semicuspidal algebra $C(\delta)$ as the top piece, while the one arising from Kronecker orientation $1\rightrightarrows 0$ does.

\begin{rmk}
	We could fit the stratification of $Y_\m$ above in the framework of~\cite{McN_RTGE2017} by adding a one-dimensional unipotent group acting on $\bbC^2$ by $a(x,y) = (x+ay,y)$. 
	However, this trick does not work when studying $A^\m(k\delta)$, $k>1$.
\end{rmk}

\medskip
\section{Stratification of seminilpotent quiver Schur}\label{sec:rep-cS}
In this section, we construct an ersatz-complete evenness theory on $\Rep_\alpha$, and deduce that the seminilpotent Schur algebra is polyhereditary.

\subsection{Stratification}
Let $\Root^+\subset \bbZ_{\geqslant 0}^I$ be the set of positive roots for $\widehat{\mathfrak{sl}}_2$, and $\Root^+_{\rm re}\subset\Root^+$ the subset of positive real roots. 
For every non-zero $\vecn=n_0\alpha_0+n_1\alpha_1$, set 
\[
	\theta(\vecn)=n_1/(n_0+n_1)\in [0,1].
\]
We write $\vecn\leqslant \vecn'$ if $\theta(\vecn)\leqslant \theta(\vecn')$, and $\vecn< \vecn'$ if the inequality is strict. 
This induces a preorder $\leqslant$ on $\Root^+$.
It is uniquely determined by the following conditions:
\begin{itemize}
	\item The preorder is convex (i.e. $\vecn\leqslant\vecn+\vecn'\leqslant\vecn'$ for $\vecn,\vecn',\vecn+\vecn'\in\Root^+$, $\vecn\leqslant\vecn'$), and $\alpha_0<\alpha_1$;
	\item The restriction of $\leqslant$ to $\Root^+_{\rm re}\cup \{\delta\}$ is a total order;
	\item The multiples of $\delta$ are equivalent with respect to $\leqslant$.
\end{itemize}
Explicitly, this preorder is as follows:
$$
\alpha_0<\alpha_0+\delta<\alpha_0+2\delta<\ldots <\delta\sim 2\delta\sim\ldots<\ldots<\alpha_1+2\delta<\alpha_1+\delta<\alpha_1.
$$

\begin{defn}
	Let $\vecn\in \bbZ^I_{\geqslant 0}$.
	A \emph{Kostant partition} $\kbeta$ of $\vecn$ is a composition of $\vecn$ of the form
	\[
		\kbeta=(i^1\beta^1,i^2 \beta^2,\ldots,i^r\beta^r),
	\]
	where $i^k\in \bbZ_{>0}$ and $\beta^1>\beta^2>\ldots>\beta^r$ are elements of $\Root^+_{\rm re}\cup \{\delta\}$.
	We denote the set of Kostant partitions of $\vecn$ by $\kp(\vecn)\subset \Comp(\vecn)$.
\end{defn}

We will need a slight variant of this notion. 
Let $\Root^+_\bullet\coloneqq \Root^+_{\rm re}\cup \{\delta^\bullet,\delta^\circ\}$ be the totally ordered set obtained from $\Root^+_{\rm re}\cup \{\delta\}$ by replacing $\delta$ with $\delta^\bullet>\delta^\circ$.
We regard $\delta^\bullet,\delta^\circ$ as two copies of the same element $\delta\in \bbZ^I_{\geqslant 0}$.
\begin{defn}
	Let $\vecn\in \bbZ^I_{\geqslant 0}$.
	A \emph{marked Kostant partition} $\mkbeta$ of $\vecn$ is a composition of $\vecn$ of the form
	\[
		\mkbeta=(i^1\beta^1,i^2 \beta^2,\ldots,i^r\beta^r),
	\]
	where $i^k\in \bbZ_{>0}$ and $\beta^1>\beta^2>\ldots>\beta^r$ are elements of $\Root^+_\bullet$.
	We denote the set of marked Kostant partitions of $\vecn$ by $\mkp(\vecn)$.
\end{defn}

The data of a marked Kostant partition $\mkbeta$ is equivalent to $(\kbeta,k^\bullet)$, where $\kbeta\in\kp(\vecn)$, and $k^\bullet$ is a non-negative integer not exceeding the coefficient of $\delta$ in $\kbeta$.
The bijection is given by
\begin{equation}\label{eq:mkbeta-vs-kbeta}
	\mkbeta = (i^1\beta^1,\ldots,k^\bullet\delta^\bullet,k^\circ\delta^\circ,\ldots,i^r\beta^r) \leftrightarrow (\kbeta = (i^1\beta^1,\ldots,(k+k')\delta,\ldots,i^r\beta^r), k^\bullet).
\end{equation}

Recall the stacky notations of \cref{rmk:stack-notations,subs:reps-Kronecker-quiver}.
For every $\vecn\in\Root^+_{\rm re}$, we have the unique indecomposable representation $M_\vecn\in\Rep_\vecn$.
The following lemma is well known, see e.g.~\cite[Ex.~3.34]{Sch_LHA2012}.
\begin{lem}\label{lem:Hom-Ext-ordered}
	Let $M$, $N$ be two indecomposable representations of $\Gamma$ with $\dim(M) < \dim(N)$. Then $\Ext^1(M,N) = \Hom(N,M) = 0$.
	If $\dim(M)\in\Root^+_{\rm re}$, then $\Ext^1(M,M) = 0$.\qed
\end{lem}
In particular, for any $\vecn\in\Root^+_{\rm re}$, $k\in \bbZ_{>0}$ we have an open substack
\[
	\Rep^\reg_{k\vecn}\coloneqq \{M_\vecn^{\oplus k}\} \simeq [\pt/GL_k] \subset \Rep_{k\vecn}.
\]
Recall that we have already defined the substacks $\Rep^\reg_{k\delta},\Rep^\reg_{k\delta^\bullet},\Rep^\reg_{k\delta^\circ}\subset \Rep_{k\delta}$ in \cref{subs:reps-Kronecker-quiver}.

\begin{defn}\label{def:stratas}
	For any $\mkbeta = (i^1\beta^1,i^2 \beta^2,\ldots,i^r\beta^r)\in\mkp(\vecn)$, define 
	\[
		\Rep_\mkbeta^\reg \coloneqq \prod_{j=1}^r \Rep_{i^j\beta^j}^\reg \subset \Rep_\mkbeta, \qquad 
		\Rep^\mkbeta_\vecn \coloneqq \Rep_\mkbeta^\reg\times_{\Rep_\mkbeta} \Fl^\s_\mkbeta.
	\]
	Define $\Rep_\kbeta^\reg$, $\Rep^\kbeta_\vecn$ for $\kbeta\in\kp(\vecn)$ analogously.
\end{defn}

By \cref{lem:Hom-Ext-ordered}, the natural map $q:\Rep^\mkbeta_\vecn\to \Rep_\mkbeta^\reg$ is bijective on $\bbC$-points.
Note that for $V = (V_1,\ldots, V_r)\in \Rep_\mkbeta^\reg$, we have a short exact sequence
\[
	0\to \Id_V + \bigoplus_{i<j} \Hom_{\bbk \Gamma}(V_j,V_i) \to \Aut_{\Rep^\mkbeta_\vecn}(V) \to \Aut_{\Rep_\mkbeta^\reg}(V) \to 0,
\]
where the first group is unipotent.
On the other hand, the forgetful map $\Rep^\mkbeta_\vecn\to \Rep_\vecn$ is a locally closed embedding.
By classification of indecomposable representations of the Kronecker quiver, we have
\begin{equation}\label{eq:stratas}
	\Rep_\vecn = \bigsqcup_{\kbeta\in\kp(\vecn)} \Rep^\kbeta_\vecn = \bigsqcup_{\mkbeta\in\mkp(\vecn)} \Rep^\mkbeta_\vecn.
\end{equation}
The stratification into $\Rep^\kbeta_\vecn$ satisfies the conditions~\eqref{eq:strat-cond} by~\cite[Prop.~3.7]{reineke2003harder}.
The same condition holds for each stratification $\Rep^\kbeta_\vecn = \sqcup_k \Rep^{(\kbeta,k)}_\vecn$, since it holds for~\eqref{eq:repreg-str}.

\subsection{Evenness theory on $\Rep_\vecn$}
We begin by defining an evenness theory on each $\Rep_\mkbeta^\reg$, $\mkbeta\in\mkp(\vecn)$.
Let $k^\bullet$, resp. $k^\circ$ be the multiplicity of $\delta^\bullet$, resp. $\delta^\circ$ in $\mkbeta$.
Then by definition, we have 
\[
	\Rep_\mkbeta^\reg \simeq [\mathcal{N}_{\fkgl_{k^\bullet}}/GL_{k^\bullet}] \times [\fkgl_{k^\circ}/GL_{k^\circ}] \times [\pt/G'] = [(\mathcal{N}_{\fkgl_{k^\bullet}}\times\fkgl_{k^\circ})/G],
\]
where $G'$ is a product of general linear groups, and $G \coloneqq G'\times GL_{k^\bullet} \times GL_{k^\circ}$.

It is known that $\mathcal{N}_{\fkgl_{k^\bullet}}\times \mathcal{N}_{\fkgl_{k^\circ}}$ has finitely many $GL_{k^\bullet}\times GL_{k^\circ}$ orbits, with stabilizers satisfying the conditions of \cref{rmk:classical-parity}.
In particular, we obtain piecewise constant evenness theory $\Ev^\mathrm{pc}$ on $[(\mathcal{N}_{\fkgl_{k^\bullet}}\times\mathcal{N}_{\fkgl_{k^\circ}})/G]$.

Let us consider the maps 
\begin{align*}
	&i:[(\mathcal{N}_{\fkgl_{k^\bullet}}\times\mathcal{N}_{\fkgl_{k^\circ}})/G]\hookrightarrow[(\mathcal{N}_{\fkgl_{k^\bullet}}\times\fkgl_{k^\circ})/G],\\
	&p:[(\mathcal{N}_{\fkgl_{k^\bullet}}\times\fkgl_{k^\circ})/G] \to [\mathcal{N}_{\fkgl_{k^\bullet}}/G].
\end{align*}
The first map $i$ is a closed embedding, and $p$ is a vector bundle, self-dual via trace form on $\fkgl_{k^\circ}$.
Consider the Fourier-Sato transform $\Theta_p$; then $\Theta_p\circ i_*(\Ev^\mathrm{pc})$ is an evenness theory on $\Rep_\mkbeta^\reg$.

\begin{rmk}\label{rmk:comp-to-tensor}
	Alternatively, we can define an evenness theory on $\Rep_\mkbeta^\reg$ by appealing to \cref{lem:prod-of-Ev}, and check that the result is the same.
\end{rmk}

Pulling back $\Theta_p\circ i_*(\Ev^\mathrm{pc})$ along $\Rep^\mkbeta_\vecn\to \Rep_\mkbeta^\reg$, we get an evenness theory $\Ev_\mkbeta\subset D^b(\Rep^\mkbeta_\vecn)$.
Finally, \cref{cor:gluing-evenness} gives rise to an evenness theory on $\Rep_\vecn$, which we will denote $\Ev_\vecn^\mkp$.

\begin{expl}
	Let $\alpha = k\delta$, and restrict $\Ev_\vecn^\mkp$ to the open regular stratum $\Rep^\reg_{k\delta}$.
	This is an evenness theory, with at most $p_2(k)$ ersatz parity sheaves, where $p_2(k)$ is the number of bipartitions of $k$.
\end{expl}

\subsection{Evenness of flag sheaves}\label{subs:evenness-flag}
The main source of sheaves in $\Ev_\vecn^\mkp$ is flag sheaves.

\begin{prop}\label{prop:flag-is-even}
	Let $\Ibe\in \Comp(\vecn)$, $p_\Ibe: \Fl^\m_\Ibe\to \Rep_\vecn$ the projection, and 
	\[
		\cL_\Ibe\coloneqq (p_\Ibe)_*\bbk.
	\]
	Then $\cL_\Ibe\in \Ev_\vecn^\mkp$.
\end{prop}

Before proceeding with the proof, we need some intermediary results.
\begin{defn}
	Let $M\in \Rep_\vecn$, and $\Ibe\in\Comp(\vecn)$.
	The variety $\Fl^\m_\Ibe(M)\coloneqq \Fl^\m_\Ibe\times_{\Rep_\vecn} \{M\}$ is called a \textit{(seminilpotent) Lusztig fiber}.
\end{defn}

We can analogously define $\Fl^\diamond_\Ibe(M)$ for any quiver $\Gamma = (I,H)$ and any function $\diamond:H\to \{\s,\n\}$ as in \cref{subs:FS-trans}.

\begin{prop}\label{prop:Lus-even}
	All seminilpotent Lusztig fibers for the Kronecker quiver have even cohomology.
\end{prop}
\begin{proof}
	For flags of KLR type (when $\Fl^\m_\Ibe = \Fl^\s_\Ibe$) this is proved in~\cite[App.~A]{MakMin_KLR2023}.
	The vanishing of odd cohomology is proven for Lusztig fibers $\Fl^\s_\Ibe(M)$, $\Fl^\n_\Ibe(M)$ for all Dynkin and affine quivers in~\cite{zhou2022affine}.
	The proof replaces the original quiver $\Gamma = (I,H)$ with its extended quiver on vertices $I\times [1,d]$, and Lusztig fibers with (extended) quiver Grassmannians.
	Two flavors of extended quiver are considered, depending on whether one works with $\Fl^\s$ or $\Fl^\n$; for the former, the copies of arrows of the original quiver preserve the second coordinate, while for the latter they decrease it by $1$.
	More generally, for any function $\diamond:H\to \{\s,\n\}$ one may consider the ``mixed'' extended quiver, where the behaviour of second coordinate depends on the value of $\diamond$.
	The proof in~\cite{zhou2022affine} then extends verbatim to show the evenness of $H^*(\Fl^\diamond_\Ibe(M))$.
\end{proof}

\begin{lem}\label{lem:base-of-induction}
	Let $\alpha\in \Root^+_\bullet$, $k>0$, and $\Iga\in \Comp(k\alpha)$.
	Then $r^*\cL_\Iga$ is even, where $r:\Rep^\reg_{k\alpha} \to \Rep_{k\alpha}$ is the inclusion.
\end{lem}
\begin{proof}
	Let $\alpha\in \Root^+_{\rm re}$, and $M_\vecn\in \Rep_\vecn$ the indecomposable representation.
	Since $\Rep^\reg_{k\alpha} = \{M_\vecn^{\oplus k}\}\simeq [\pt/GL_k]$, the sheaf $r^*\cL_\Iga$ is nothing else than the $H^*_{GL_k}(\pt)$-module $H^*_{GL_k}(\Fl^\m_\Iga(M_\vecn^{\oplus k}))$.
	As $\Fl^\m_\Iga(M_\vecn^{\oplus k})$ has even cohomology, this module is a direct sum of even shifts of the constant sheaf on $[\pt/GL_k]$ by degeneration of Leray spectral sequence.
	Therefore $r^*L_\Iga$ is even by the definition of $\Ev^\mkp$.

	Let $\alpha = \delta_\bullet$.
	In this case $\Rep^\reg_{k\delta_\bullet} = [\mathcal{N}_{\fkgl_k}/GL_k]$ has finitely many points, and the evenness theory is the locally constant one.
	In particular, the evenness of $r^*\cL_\Iga$ follows from \cref{prop:Lus-even} as before.

	Finally, let $\alpha = \delta_\circ$.
	Since all representations in $\Rep^\reg_{k\delta_\circ} \simeq [\fkgl_k/GL_k]$ are regular, we have $r^*\cL_\Iga = 0$ unless $\Iga_1^j \leq \Iga_0^j$ for all $1\leq j\leq \ell(\Iga)$.
	Note that the arrow $h_\n$ provides a canonical isomorphism $V_1\xra{\sim} V_0$ for any representation $(V_1\rightrightarrows V_0) \in \Rep^\reg_{k\delta_\circ}$.
	Let us denote by $\Gamma'\subset \Gamma$ the subquiver $1\xra{h_\n} 0$.
	We have the following fiber square:
	\begin{equation}\label{eq:tcd-restr}
		\begin{tikzcd}
			\left[ \fkgl_k/GL_k \right] \ar[r,equal]\ar[d] & \Rep^\reg_{k\delta_\circ}(\Gamma) \ar[r,hook,"j"]\ar[d,"p"]\arrow[dr, phantom, "\scalebox{1.3}{$\lrcorner$}", at start, shift right=1ex, color=black] & \Rep_{k\delta_\circ}(\Gamma)\ar[d,"p"] \\
			\left[\pt/GL_k \right] \ar[r,equal] & \Rep^\reg_{k\delta_\circ}(\Gamma') \ar[r,hook] & \Rep_{k\delta_\circ}(\Gamma') 
		\end{tikzcd}
	\end{equation}
	where $\Rep^\reg_{k\delta_\circ}(\Gamma')$ parameterizes the unique representation of $\Gamma'$ with $h_\n$ isomorphism.
	For any flag type $\Iga$, we need to check that $\Theta_p \circ j^* \cL_\Iga$ is even.
	Since Fourier-Sato transform is local in the target, we have
	\[
		\Theta_p \circ j^*(\cL_\Iga) = (j')^*\circ \Theta_p (\cL_\Iga) = (j')^* \cL^\n_\Iga,
	\]
	where $\cL^\n_\Iga$ is the flag sheaf of type $\Iga$ on $\Rep_{k\delta}(1\leftrightarrows 0)$ with nilpotent conditions, and
	\[
		j': \Rep^\reg_{k\delta}(1\leftrightarrows 0)\hookrightarrow \Rep_{k\delta}(1\leftrightarrows 0)
	\]
	is the open inclusion of representations where the bottom arrow is an isomorphism. 
	Analogously to~\eqref{eq:tcd-restr}, we have a fiber diagram
	\[
		\begin{tikzcd}
			\left[ \cN_{\fkgl_k}/GL_k \right] \ar[r,hook]\ar[d]\arrow[dr, phantom, "\scalebox{1.3}{$\lrcorner$}", at start, shift right=1ex, color=black] & \Rep^{\mathrm{nil}}_{k\delta}(1\leftrightarrows 0)\ar[d] \\
			\left[\pt/GL_k \right] \ar[r,hook] & \Rep_{k\delta_\circ}(\Gamma') 
		\end{tikzcd}
	\]
	Thus, we are again reduced to showing that the Lusztig fibers $\Fl^\n_\Iga(N)$ are even, for any $N\in \Rep^{\mathrm{nil}}_{k\delta}(1\leftrightarrows 0)$.
	This is known by~\cite[Lem.~3.40]{Mak_CBKA2015}.
\end{proof}

Consider Lusztig's \textit{restriction functor}:
\[
	\Delta_{\Ibe}=(q_{\Ibe})_!(p_{\Ibe})^*:D^b(\Rep_\alpha)\to D^b(\Rep_\beta),\qquad \Rep_\beta \xleftarrow{q_{\Ibe}} \Fl^\s_\Ibe \xrightarrow{p_{\Ibe}} \Rep_\alpha.
\]
Recall (see \cref{subs:sp-mer-diag}) that any $w\in \fkS_{\Ibe}\backslash \fkS_{\vecn}/\fkS_{\Iga}$ yields a Schur diagram $D_w$.
In particular, we obtain compositions $\Ibe^i(w)\in \Comp(\Ibe^i)$, $1\leq i\leq \ell(\Ibe)$.

\begin{lem}
	Let $\Ibe,\Iga \in \Comp(\alpha)$, and $M = |\fkS_{\Ibe}\backslash \fkS_{\vecn}/\fkS_{\Iga}|$.
	We have a sequence of exact triangles $(F_{i-1}\to F_i\to C_i[-2k_i]\to)$ in $D^b(\Rep_\Ibe)$, such that $F_0 = 0$, $F_M = \Delta_{\Ibe}(\cL_\Iga)$, $k_i\in\bbZ$, and
	\[
		\{C_i : 1\leq i\leq M\} = \left\{ L_{\Ibe^1(w)}\boxtimes\ldots\boxtimes L_{\Ibe^\ell(w)} : w\in \fkS_{\Ibe}\backslash \fkS_{\vecn}/\fkS_{\Iga} \right\}.
	\]
	We call this filtration of $\Delta_{\Ibe}(L_\Iga)$ the \textit{Mackey filtration}.
\end{lem}
\begin{proof}
	The statement is essentially due to Lusztig~\cite[(9.2.11)]{Lus1993}.
	Namely, the proof of Lem.~9.2.4 in~\emph{loc.cit.} constructs a stratification $\Fl_\Iga^\s\times_{\Rep_\alpha} \Fl_\Ibe^\s = \bigsqcup_w X_w$ into vector bundles $X_w\to \Fl^\s_{\Ibe^1(w)}\times \ldots \times \Fl^\s_{\Ibe^\ell(w)}$, $w\in \fkS_{\Ibe}\backslash \fkS_{\vecn}/\fkS_{\Iga}$.
	Restricting from $\Fl_\Iga^\s$ to $\Fl_\Iga^\m$, we obtain a stratification $\Fl_\Iga^\m\times_{\Rep_\alpha} \Fl_\Ibe^\s = \bigsqcup_w X_w$ into vector bundles $X^\m_w\to \Fl^\m_{\Ibe^1(w)}\times \ldots \times \Fl^\m_{\Ibe^\ell(w)}$.
	Consider the locally closed inclusions
	\[
		i_w: X^\m_w\hookrightarrow \Fl^\m_\Iga\times_{\Rep_\alpha} \Fl^\s_\Ibe.
	\]
	By recollement, we have a sequence of exact triangles 
	\begin{equation}\label{eq:interm-ses}
		G_{i-1}\to G_i\to D_i[-2m_i]\xra{+1}
	\end{equation}
	in $D^b(\Fl^\m_\Iga\times_{\Rep_\alpha} \Fl^\s_\Ibe)$, such that $G_0 = 0$, $G_M = \bbk$, $m_i\in\bbZ$, and $\{D_i\}_i = \{(i_w)_!\bbk_{X^\m_w}\}_w$.
	On the other hand, let us denote the natural projections from $\Fl^\m_\Iga\times_{\Rep_\alpha} \Fl^\s_\Ibe$ to $\Fl^\s_\Ibe$, resp. $\Fl^\m_\Iga$ by $p_1$, resp. $p_2$.
	We have
	\begin{align*}
		\Delta_{\Ibe}(\cL_\Iga) &= (q_{\Ibe})_!(p_{\Ibe})^*(p_{\Iga})_!(\bbk_{\Fl^\m_\Iga}) = (q_{\Ibe})_!(p_1)_!(p_2)^*(\bbk_{\Fl^\m_\Iga}) \\
		&= (p_1\circ q_{\Ibe})_!(\bbk_{\Fl^\m_\Iga\times_{\Rep_\alpha} \Fl^\s_\Ibe}).
	\end{align*}
	Pushing the sequence~\eqref{eq:interm-ses} down to $\Rep_\Ibe$, we arrive at the desired statement.
\end{proof}

\begin{prop}\label{prop:Mackey-splits}
	Mackey filtration splits, that is we have
	\[
		\Delta_{\beta}(\cL_\gamma) = \bigoplus_{w\in \fkS_{\Ibe}\backslash \fkS_{\vecn}/\fkS_{\Iga}} \cL_{\Ibe^1(w)}\boxtimes\ldots\boxtimes \cL_{\Ibe^\ell(w)}[-2k_w]
	\]
	for any $\beta,\gamma\in \Comp(\vecn)$.
\end{prop}
\begin{rmk}
	Note that the splitting argument in~\cite[{\S~9.2}]{Lus1993} relies on the theory of weights for perverse sheaves, and therefore does not apply when $\ona{char} \bbk > 0$.
\end{rmk}

We prove \cref{prop:flag-is-even,prop:Mackey-splits} simultaneously.

\begin{proof}[{Proof of \cref{prop:flag-is-even,prop:Mackey-splits}}]
	We proceed by induction on $|\vecn|$.
	Assume we know both \cref{prop:flag-is-even} and \cref{prop:Mackey-splits} for all $\vecn'$ with $|\vecn'|<N$, and let $|\vecn|=N$.
	All sheaves occuring in Mackey filtration of $\Delta_{\beta}(\cL_\gamma)$ are even by induction, which implies \cref{prop:Mackey-splits} for $\vecn$.
	Next, consider the following commutative diagram: 
	\[
		\begin{tikzcd}
			\Rep_\mkbeta^\reg\ar[d,"i"] & \Rep^\mkbeta_\vecn\ar[r,equal]\ar[l,"q"']\ar[d,"i"] & \Rep^\mkbeta_\vecn\ar[d,"i_\mkbeta"] \\
			\Rep_\mkbeta & \Fl^\s_{\mkbeta}\ar[r,"p_\mkbeta"]\ar[l,"q_\mkbeta"'] & \Rep_\vecn
		\end{tikzcd}
	\]
	The sheaf $\cL_\Iga\in D^b(\Rep_\vecn)$ is Verdier self-dual up to an even shift.
	Therefore $\cL_\Iga$ is even if and only if $q_! i_\mkbeta^*\cL_\Iga$ is even for any $\mkbeta\in\mkp(\vecn)$.
	By base change, we have
	\[
		q_! i_\mkbeta^*\cL_\Iga = q_! i^*p_\mkbeta^*\cL_\Iga = i^*(q_\mkbeta)_!p_\mkbeta^*\cL_\Iga = i^*\Delta_{\mkbeta}(L_\gamma).
	\]
	Therefore it suffices to check that $i^*\Delta_{\mkbeta}\cL_\Iga$ is even for all $\mkbeta\in\mkp(\vecn)$.
	Applying \cref{prop:Mackey-splits}, we are reduced to checking the evenness of $i^*\cL_{\mkbeta'}$ for $\mkbeta'\in \Comp(\vecn')$, $\vecn' \in \bbZ_{>0}\Root_\bullet^+$, which follows from \cref{lem:base-of-induction}.
\end{proof}

\subsection{Orders on Kostant partitions}
The set $\mkp(\vecn)$ of marked Kostant partitions admits two natural orders.
First, each $\mkbeta\in\mkp(\vecn)$ gives rise to a locally closed stratum $\Rep_\vecn^\mkbeta\subset \Rep_\alpha$ by~\eqref{eq:stratas}.
The \textit{closure order} $\leqslant_C$ on $\mkp(\vecn)$ is defined by 
\[
	\mkbeta \leqslant_C \mkgamma \qquad \Leftrightarrow \qquad \Rep_\vecn^\mkbeta \subset \overline{\Rep_\vecn^\mkgamma}.
\]
Given $\mkbeta=(i^1\beta^1,\ldots,i^r\beta^r)$, let us define a closed substack $\Rep_\mkbeta^\qs\subset \Rep_\mkbeta$.
Write $\Rep_\mkbeta^\qs \coloneqq \prod_k \Rep_{i^k\beta^k}^\qs$, where $\Rep_{i^k\beta^k}^\qs = \Rep_{i^k\beta^k}$ if $\beta^k \neq \delta^\bullet$, and $\Rep_{i\delta^\bullet}^\qs$ is the image of the map $\Fl^\m_{\delta,\delta,\ldots,\delta}\to \Rep_{i\delta^\bullet}$.
Define 
\[
	\Fl_\mkbeta^\qs\coloneqq \Rep_\mkbeta^\qs\times_{\Rep_\mkbeta} \Fl^\s_\mkbeta.
\]
It is clear from \cref{def:stratas} that $\Rep^\reg_\mkbeta \subset \Rep_\mkbeta^\qs$, and so the inclusion $\Rep_\vecn^\mkbeta\subset \Fl_\mkbeta^\s$ factors through $\Fl_\mkbeta^\qs$. 
We will say that a representation $V\in \Rep_\vecn$ \textit{supports a flag of type $\mkbeta$} if $V$ lies in the image of the proper map $\Fl_\mkbeta^\qs\to \Rep_\vecn$.

\begin{lem}\label{lem:closure-vs-sub}
	We have $\mkbeta \leqslant_C \mkgamma$ if and only if every representation $V \in \Rep_\vecn^\mkbeta$ supports a flag of type $\mkgamma$.
\end{lem}
\begin{proof}
	The argument is contained in the proof of~\cite[Prop.~3.4]{reineke2003harder}.
	To summarize, $\Fl^\qs_\mkgamma\to \Rep_\vecn$ is a proper map with irreducible source, which restricts to an isomorphism over the open dense $\Rep_\vecn^\mkgamma$.
	This implies that the closure $\overline{\Rep_\vecn^\mkgamma}$ is precisely the image of $\Fl_\mkgamma^\qs$.
\end{proof}

Second, we have total \textit{reverse lexicographic} order.
Recall that the set $\Root^+_\bullet$ is totally ordered.
Then given $\mkbeta,\mkgamma\in\mkp(\vecn)$, $\mkbeta=(i^1\beta^1,\ldots,i^r\beta^r)$, $\mkgamma=(j^1\gamma^1,\ldots,j^s\gamma^s)$, we say that $\mkbeta \leqslant_{L} \mkgamma$ if there exists $t\leq \min(r,s)$ such that $i^l\beta^l = j^l\gamma^l$ for all $l<t$, and either $\beta^t > \gamma^t$, or $\beta^t = \gamma^t$ and $i^t>j^t$.

\begin{prop}\label{prop:lex-refines-closure}
	Reverse lexicographic order refines closure order, i.e. $\mkbeta \leqslant_C \mkgamma$ implies $\mkbeta \leqslant_{L} \mkgamma$.
\end{prop}
\begin{proof}
	Let $\mkgamma >_L \mkbeta$.
	By \cref{lem:closure-vs-sub}, it suffices to show that a representation $V\in \Rep_\vecn^\mkbeta$ does not support a flag of type $\mkgamma$.
	We reason by contradiction.
	Assume $V_1\subset V_2\subset\ldots\subset V$ is a flag of type $\mkgamma$, and denote by $t$ the smallest index such that $\mkgamma^t > \mkbeta^t$.
	Recall that by definition (and \cref{lem:Hom-Ext-ordered}), $V$ supports the unique flag $V'_1\subset\ldots\subset V$ of type $\mkbeta$; in particular, we have $V_i = V'_i$ for all $i<t$.
	Passing from $V$ to $V/V_{t-1}$, we can assume that $t=1$.
	If $\gamma^1 > \beta^1$, then any map $V_1\to V$ has to vanish by \cref{lem:Hom-Ext-ordered}.
	If $\gamma^1 = \beta^1$ and $j^1>i^1$, then for the same reason the map $V_1\to V$ factors through $V'_1$.
	However, the dimension of $V'_1$ is strictly smaller than that of $V_1$, so that $V_1\to V$ cannot be an injection.
	Thus $V$ cannot support a flag of type $\mkgamma$.
\end{proof}

\begin{cor}\label{cor:flag-support}
	Let $\mkbeta\in\mkp(\vecn)$, and pick an arbitrary composition $\beta\in\Comp(\vecn)$ refining $\mkbeta$.
	Then the support of the flag sheaf $\cL_\beta$ belongs to $\overline{\Rep_\vecn^\mkbeta}$.
\end{cor}
\begin{proof}
	The closure $\overline{\Rep_\vecn^\mkbeta}$ is precisely the image of the projection $p:\Fl^\qs_\mkbeta\to \Rep_\vecn$ by \cref{lem:closure-vs-sub}.
	The flag sheaf $\cL_\beta$ is the pushforward of the constant sheaf along $\Fl^\m_\beta\to \Rep_\vecn$.
	Since the latter map factors through $p$, we may conclude.
\end{proof}

\begin{rmk}
	One can write down an intermediate partial order between $\leqslant_C$ and $\leqslant_L$ similarly to~\cite[Prop.~3.7]{reineke2003harder}.
	We do not need this.
\end{rmk}

\subsection{Ersatz-completeness}

\begin{prop}\label{prop:ersatz-comp}
	The evenness theory $\Ev_\vecn^\mkp$ is ersatz-complete.
\end{prop}
\begin{proof}
	In view of \cref{prop:flag-is-even,cor:flag-support}, for each $\mkbeta\in \mkp(\vecn)$ and indecomposable $F\in\Ev_\mkbeta$ it suffices to find a composition $\Ibe\in\Comp(\vecn)$ refining $\mkbeta$, such that $\cL_\Ibe|_{\Rep_\vecn^\mkbeta} = F$.
	Let us write $\mkbeta = (i^1\beta^1,\ldots,i^r\beta^r)$, and $\Ibe$ is a concatenation of compositions $\Iga_k\in\Comp(i^k\Ibe^k)$.
	Recall the map $q:\Rep^\mkbeta_\vecn\to \Rep_\mkbeta^\reg$.
	Since any representation $V\in \Rep_\vecn^\mkbeta$ supports a unique flag of type $\mkbeta$, we have 
	\[
		\cL_\Ibe|_{\Rep_\vecn^\mkbeta} = q^*(\cL_{\Iga_1} \boxtimes\ldots \boxtimes \cL_{\Iga_r}).
	\]
	By \cref{rmk:comp-to-tensor} we are therefore reduced to the length $1$ case $\mkbeta = (k\alpha')$, $k>0$, $\alpha'\in \Root^+_\bullet$.

	Assume $\alpha'\in\Root^+_{\rm re}$, then $V = M_{\alpha'}^{\oplus k}$.
	Explicitly, $M_{\alpha'}$ are as follows:
	\[
		M_{\alpha_1+i\delta} = 
		\begin{tikzcd}[column sep = large]
			\bbC^{i+1}\ar[r,shift left,"(\id_{\bbC^i}\text{ }0)"]\ar[r,shift right,"(0\text{ }\id_{\bbC^i})"'] & \bbC^i
		\end{tikzcd},\qquad 
		M_{\alpha_0+i\delta} = 
		\begin{tikzcd}[column sep = large]
			\bbC^{i}\ar[r,shift left,"\begingroup
			\renewcommand*{\arraystretch}{0.6}\begin{pmatrix}\id_{\bbC^i}\\0\end{pmatrix}\endgroup"]\ar[r,shift right,"\begingroup
			\renewcommand*{\arraystretch}{0.6}\begin{pmatrix}0\\\id_{\bbC^i}\end{pmatrix}\endgroup"'] & \bbC^{i+1}
		\end{tikzcd}.
	\]
	One checks directly that $M_{\alpha_1+i\delta}^{\oplus k}$ has the unique flag of type
	\begin{equation}\label{eq:idem-1}
		\Ibe_{k(\alpha_1+i\delta)} \coloneqq (k\delta,k\delta,\ldots,k\delta,k\alpha_1),
	\end{equation}
	and $M_{\alpha_0+i\delta}^{\oplus k}$ has the unique flag of type 
	\begin{equation}\label{eq:idem-0}
		\Ibe_{k(\alpha_0+i\delta)} \coloneqq (k\alpha_0,k\delta,k\delta,\ldots,k\delta).
	\end{equation}
	Therefore $\cL_{\Ibe_{k\alpha'}}|_{\Rep_{k\alpha'}^\mkbeta} = \bbk$ for any $\alpha'\in \Root^+_{\rm re}$.

	Let $\alpha' = \delta^\bullet$; then $\Rep_{k\alpha'}^\mkbeta = [\cN_{\fkgl_k}/GL_k]$, and the evenness theory is piecewise constant on the nilpotent orbits.
	For a partition $\lambda\vdash k$, $\lambda = (\lambda^1\geq\ldots \geq \lambda^r)$, consider the flag type 
	\begin{equation}\label{eq:idem-bul}
		\Ibe^\bullet_\lambda \coloneqq (\lambda^1\delta,\ldots,\lambda^r\delta)
	\end{equation}
	and the nilpotent orbit $\bbO_\lambda$ with sizes of Jordan blocks given by $\lambda$.
	By \cref{prop:gen-images}, the restriction of $\Fl_{\Ibe^\bullet_\lambda}^\m$ to $\Rep_{k\alpha'}^\mkbeta$ is simply the partial Springer resolution $T^*(GL_k/P_\lambda)$.
	It is known that the image of projection $p: T^*(GL_k/P_\lambda)\to \cN_{\fkgl_k}$ lies in the closure $\overline{\bbO_{\lambda^t}}$, and $p$ restricts to an isomorphism over $\bbO_{\lambda^t}$.
	Therefore $\cL_{\Ibe^\bullet_{\lambda}}$ contains the ersatz parity sheaf corresponding to $\lambda^t$ as a direct summand.

	Finally, let $\alpha' = \delta^\circ$, so that $\Rep_{k\alpha'}^\mkbeta = [\fkgl_k/GL_k]$.
	For a partition $\mu\vdash k$, consider the flag type 
	\begin{equation}\label{eq:idem-circ}
		\Ibe^\circ_\mu \coloneqq (\mu^1\alpha_0,\mu^1\alpha_1,\ldots,\mu^r\alpha_0,\mu^r\alpha_1).
	\end{equation}
	It follows from~\cite[Lem.~8.7]{MakMin_KLR2023} that the restriction of $\Fl_{\Ibe^\circ_\mu}^\m$ to $\Rep_{k\alpha'}^\mkbeta$ is the partial Grothendieck-Springer alternation $GL_k\times_{P_\mu} \mathfrak{p}_\mu$.
	Under Fourier-Sato transform on $[\fkgl_k/GL_k]$, the flag sheaves $\cL_{\Ibe^\circ_{\mu}}$ go to the flag sheaves $\cL_{\Ibe^\bullet_{\mu}}$ on the nilpotent cone, so we may conclude by the previous case.
\end{proof}

\begin{cor}\label{cor:even-are-flag}
	The category of ersatz parity sheaves on $\Rep_\vecn$ with respect to $\Ev_\vecn^\mkp$ coincides with category $\FShv_\vecn$ of direct sums of shifted direct summands of flag sheaves $\cL_\Ibe$, $\Ibe\in\Comp(\vecn)$. 
\end{cor}

\begin{rmk}\label{rmk:simpler-idems}
	We can replace the idempotents in (\ref{eq:idem-1},~\ref{eq:idem-0}) by 
	\[
		\Ibe'_{k(\alpha_1+i\delta)} \coloneqq (ik\alpha_0,(i+1)k\alpha_1), \quad \Ibe'_{k(\alpha_0+i\delta)} \coloneqq ((i+1)k\alpha_0,ik\alpha_1).
	\]
\end{rmk}

\begin{rmk}\label{rmk:geom-Ext}
Hone--Williamson introduced the notion of geometric extension sheaves in~\cite{hone2025geometric}.
Given a smooth irreducible open $U\subset Y$ and a smooth proper map\footnote{Note that $\mathcal{L}$ is automatically surjective, since it is both open and closed.} $\mathcal{L}:V\to U$, the geometric extension $\mathscr{E}(Y,\mathcal{L})$ is the universal extension of $\mathcal{L}_*\bbk_V$ to $Y$, contained as a direct summand of $f_*\bbk_X$ for any proper $f:X\to Y$ which restricts to $\mathcal{L}$ over $U$.
This definition recovers IC-sheaves, as well as parity sheaves in many situations, with the additional virtue of being agnostic about the choice of stratification of $Y$.
On the other hand, the theory of geometric extensions has little control over the behaviour of $\mathscr{E}(Y,\mathcal{L})$ at the boundary $Y\setminus U$; in particular, there is currently no well-behaved notion of a \textit{category} of geometric extensions.

In view of \cref{cor:even-are-flag,prop:flag-is-even}, ersatz parity sheaves on $\Rep_\vecn$ are the direct summands of geometric extensions obtained from partial flag varieties.
However, we can only observe this \textit{a posteriori}, following the inductive argument in \cref{subs:evenness-flag}. 
One may enquire in which situations geometric extensions can be arranged into a polyhereditary evenness theory; we hope to return to this question in future work.
\end{rmk}

\subsection{Polyheredity structure on $A^\m(\vecn)$}
By \cref{prop:ersatz-comp}, we can apply \cref{thm:quasi-her-inher} to the evenness theory $\Ev_\vecn^\mkp$ on $\Rep_\vecn$.
Since the evenness theory $\Ev_\mkbeta$ on each stratum is equivalent to piecewise constant evenness theory, it is polyhereditary by \cref{rmk:classical-parity}.
Hence $\Ev_\vecn^\mkp$ is polyhereditary as well.
By \cref{cor:even-are-flag}, the direct sum of all flag sheaves $\cL_\vecn = \sum_{\Ibe\in\Comp(\vecn)} \cL_\beta$ is a full sheaf in $\Par_\vecn^\mkp$.
Therefore the algebra $\Ext^*(\cL_\vecn,\cL_\vecn)$ is polyhereditary by \cref{cor:polyher-case}.
We obtain:

\begin{thm}\label{thm:Schur-poly-qher}
	Let $\Gamma = 1 \rightrightarrows 0$ be the Kronecker quiver, and $\vecn \in Q^+_I$.
	The seminilpotent quiver Schur algebra $A_\Gamma^\m(\alpha)$ is polyhereditary.\qed
\end{thm}

Let us be more explicit about the sequence of idempotents in $A_\Gamma^\m(\alpha)$ yielding the polyhereditary structure.
Recall the notations in~\eqref{eq:mkbeta-vs-kbeta}, and consider the set 
\begin{equation}\label{eq:indexing-set}
	\Lambda_\vecn = \left\{ (\mkbeta,\lambda,\mu) : \mkbeta\in \mkp(\vecn), \lambda\vdash k^\bullet, \mu\vdash k^\circ \right\}.
\end{equation}
We equip the triples in $\Lambda_\vecn$ with the lexicographic order, where each coordinate is arranged by reverse lexicographic order.
In other words,
\[
	(\mkbeta_1,\lambda_1,\mu_1) \leqslant (\mkbeta_2,\lambda_2,\mu_2) \Leftrightarrow 
	\left[
	\begin{array}{l}
	\mkbeta_1>_L \mkbeta_2; \\
    \mkbeta_1 = \mkbeta_2, \lambda_1 > \lambda_2; \\
	\mkbeta_1 = \mkbeta_2, \lambda_1 = \lambda_2, \mu_1 > \mu_2.
	\end{array}
	\right.
\]
To each element $(\mkbeta,\lambda,\mu)\in \Lambda_\vecn$, we associate a composition $\gamma(\mkbeta,\lambda,\mu)\in \Comp(\vecn)$, defined via the following procedure:
\begin{enumerate}
	\item Write $\mkbeta$ as $(k^1\alpha^1,\ldots,k^r\alpha^r,k^\bullet \delta^\bullet,k^\circ \delta^\circ, k^{r+1}\alpha^{r+1},\ldots,k^s\alpha^s)$, $\alpha_i\in \Root^+_{\rm re}$;
	\item To each $\alpha^i\in \Root^+_{\rm re}$ associate the composition $\Ibe_{k^i\alpha^i}$ as defined in (\ref{eq:idem-1},~\ref{eq:idem-0});
	\item For $\delta^\bullet$ resp. $\delta^\circ$, take the composition $\Ibe^\bullet_{\lambda^t}$, resp. $\Ibe^\circ_{\mu^t}$ as defined in (\ref{eq:idem-bul},~\ref{eq:idem-circ});
	\item Concatenate compositions in (2), (3) in the order given by (1). Denote the result by $\gamma(\mkbeta,\lambda,\mu)\in \Comp(\vecn)$. 
\end{enumerate}

Combining \cref{thm:quasi-her-inher,prop:lex-refines-closure,prop:ersatz-comp}, we obtain

\begin{thm}\label{thm:Schur-polyher-order}
	The polyhereditary structure on $A_\Gamma^\m(\alpha)$ is given by the sequence of idempotents $1_{\gamma(\mkbeta,\lambda,\mu)}$, $(\mkbeta,\lambda,\mu)\in\Lambda_\vecn$ in reverse lexicographic order.\qed
\end{thm}

\subsection{Another polyhereditary structure}\label{ssec:another-polyher}
Note that we have the following simple corollary of the discussion in \cref{rmk:classical-parity}. 
\begin{prop}
	Let $\Gamma$ be a cyclic quiver, and $\vecn\in Q_I^+$.
	The nilpotent Schur algebra $A^\n_\Gamma(\vecn)$ is polyhereditary.
\end{prop}
\begin{proof}
	All nilpotent flag sheaves are supported on the locus $\Rep_\vecn^{\mathrm{nil}}\subset \Rep_\vecn$ of nilpotent representations.
	It is well known that the cyclic quiver has finitely many nilpotent representations of fixed dimension vector $\vecn$, and the stabilizer of each such representation is homotopic to a product of general linear groups.
	Every flag sheaf has even fibers by~\cite{Mak_CBKA2015,zhou2022affine}.
	Moreover, for each nilpotent representation $V$ there exists $\Ibe\in \Comp(\vecn)$ such that $V$ admits exactly one flag of type $\Ibe$~\cite[Prop.~2.7]{SW_QSAF2014}.

	Summarizing, the piecewise constant evenness theory $\Ev^{\mathrm{pc}}$ on $\Rep_\vecn^{\mathrm{nil}}$ is ersatz-complete, therefore polyhereditary, and the sum of flag sheaves is full.
	We conclude by \cref{cor:McNa-result}.
\end{proof}

In particular, we get another polyhereditary structure on $A^\n_{1\leftrightarrows 0}(\alpha)\simeq A^\m_{1\rightrightarrows 0}(\alpha)$.
Note, however, that the two polyheredity orders are not compatible already for $\alpha = \delta$ by \cref{expl:compare-hw-orders}.
Only the polyheredity order of \cref{thm:Schur-polyher-order} is compatible with the categorification of Beck-Damiani basis of $U^-_q(\widehat{\fksl}_2)$ in~\cite{McN_RKAI2017}.

\medskip
\section{Semicuspidal quotients}\label{sec:semicusp-quot}
In this section, we unwind \cref{thm:Schur-polyher-order} and relate our results to some existing constructions.
In order to do so, we consider a coarser properly stratified structure on $A^\m(\xi)$, $\xi\in Q^+_I$.

\subsection{Semicuspidal quotients}\label{subs:def-semicusp}
Let us regroup the idempotents in \cref{thm:Schur-polyher-order}. For any $\kbeta\in \kp(\xi)$, consider
\[
	1_{\gamma(\kbeta)} \coloneqq \sum_{k,\lambda,\mu} 1_{\gamma((\kbeta,k),\lambda,\mu)}.
\]
On the sheaf side, let $\parsh_{\kbeta}$ be the direct sum of all ersatz parity sheaves in $\Ev(\Rep_\xi^\kbeta)$.
\cref{lem:Ev-to-alg-qher} implies that the (reverse) lexicographic order on idempotents $1_{\gamma(\kbeta)}$ induces a $\mathscr{C}$-properly stratified structure on $A^\m(\vecn)$, where $\mathscr{C}$ is the collection of algebras $\Ext^*(\parsh_\kbeta,\parsh_\kbeta)$ for all $\xi$ and $\mkbeta\in\mkp(\xi)$.
Since the evenness theory on $\Rep^\kbeta_\xi$ is a product of evenness theories, in order to understand $\mathscr{C}$ it suffices to consider $\xi = n\alpha$, $\kbeta = (n\alpha)$ for some $\alpha\in \Root^+$, $n>0$.

Let us denote the sum of all (seminilpotent) flag sheaves on $\Rep_{n\alpha}$ by $\cL_{n\alpha}$, and the sum of all ersatz parity sheaves supported away from $\Rep^\reg_{n\alpha}$ by $\parsh'$.
By \cref{lem:Par-open-surj}, we have 
\begin{equation}\label{eq:Ext-is-semicusp}
\begin{aligned}
	\Ext^*(\parsh_\kbeta,\parsh_\kbeta) &= \Ext^*(\cL_{n\alpha},\cL_{n\alpha})/\Ext^*(\cL_{n\alpha},\parsh')\circ \Ext^*(\parsh',\cL_{n\alpha}) \\
	&= \Ext^*(\cL_{n\alpha},\cL_{n\alpha})/\Ext^*(\cL_{n\alpha},\cL')\circ \Ext^*(\cL',\cL_{n\alpha}),
\end{aligned}
\end{equation}
where $\cL'$ is the direct sum of all flag sheaves on $\Rep_{n\alpha}$ supported away from $\Rep^\reg_{n\alpha}$.
Note that the second equality holds by ersatz-completeness.

\begin{defn}
	A composition $\Ibe=(\Ibe^1,\ldots,\Ibe^k)\in\Comp(\xi)$ is called \emph{noncuspidal} if we have $\theta(\sum_{t=1}^r\Ibe^t)>\theta(\sum_{t=r+1}^k\Ibe^t)$ for some $1\leqslant r<k$.
	We say that the idempotent $1_\Ibe\in A(\xi)$ is noncuspidal if $\Ibe$ is, and denote the sum of all noncuspidal idempotents by $1_{\nc}$.
\end{defn}

By \cref{lem:closure-vs-sub} a flag sheaf $\cL_\Ibe$, $\Ibe\in \Comp(\xi)$ is supported away from $\Rep^\reg_{n\alpha}$ for $\Ibe$ noncuspidal.
Moreover, by construction in \cref{prop:ersatz-comp} the collection of such flag sheaves contains all ersatz parity sheaves supported away from $\Rep^\reg_{n\alpha}$ as direct summands.
Thus
\begin{align*}
	\Ext^*(\cL_{n\vecn},\cL_{n\vecn})/\Ext^*(\cL_{n\vecn},\cL')\circ \Ext^*(\cL',\cL_{n\vecn})
	= A(n\alpha)/A(n\alpha)1_{\nc}A(n\alpha).
\end{align*}

\begin{defn}\label{defn:semicusp-alg}
	We call $\CA(n\vecn) \coloneqq A(n\vecn)/A(n\vecn)1_{\nc}A(n\vecn)$ the \textit{semicuspidal algebra}.
\end{defn}

When $\alpha\in\Root^+_{\rm re}$, the evenness theory on $\Rep^{(n\alpha)}_{n\alpha}$ is the constant one, so that $\CA(n\alpha)$ is Morita-equivalent to a polynomial algebra.
When $\alpha = \delta$, the algebra is more complicated, but the results of \cref{sec:rep-cS} supply us with a projective generator.
Namely, let $\cI = \{\tau,\varepsilon\}$, and recall the set of $\cI$-coloured compositions $\cI^{(n)}$ from \cref{def:comps}.
For each $n>0$, consider the map $\cI^{(n)}\to \Comp(n\delta)$, $\cla\mapsto \beta_\cla$ uniquely determined by:
\begin{itemize}
\item $(n\varepsilon)\mapsto (n\delta)$, $(n\tau)\mapsto (n\alpha_0,n\alpha_1)$,
\item The map $\coprod_j\cI^{(n_j)}\to \coprod_j\Comp(n_j\delta)$ commutes with concatenation.
\end{itemize}
Note that for a composition of pure colour $\varepsilon$, resp. $\tau$ we have $\beta_\cla=\beta_\lambda^\bullet$, resp. $\beta_\cla=\beta_\lambda^\circ$, as defined in (\ref{eq:idem-bul},~\ref{eq:idem-circ}).
We abbreviate $e_\cla \coloneqq 1_{\beta_\cla}\in A(n\delta)$.

\begin{prop}\label{prop:proj-gen}
	Let 
	\[
		e=\sum_{\cla\in\cI^{(n)}}e_\cla, \qquad e_0 = \sum_{\cla\in\cI^{n}}e_\cla.
	\]
	The algebra $\CA(n\delta)$ is Morita-equivalent to $e\CA(n\delta)e$.
	Moreover, if $\bbk$ is of characteristic zero, then $\CA(n\delta)$ is Morita-equivalent to $e_0\CA(n\delta)e_0$.
\end{prop}
\begin{proof}
	By the proof of \cref{prop:ersatz-comp}, every ersatz parity sheaf on $\Rep^\reg_{n\delta}$ can be obtained from a flag sheaf of type $\beta_\cla$, where 
	\begin{equation}\label{eq:small-idems}
		\cla = (\lambda^1\varepsilon,\ldots,\lambda^k\varepsilon,\lambda^{k+1}\tau,\ldots,\lambda^{k+l}\tau).
	\end{equation}
	This proves the first claim.

	Recall the stratification $\Rep^\reg_{n\delta} = \bigsqcup_{k=0}^n \Rep^\reg_{k\delta^\bullet}\times \Rep^\reg_{(n-k)\delta^\circ}$.
	We know by Springer theory that when $\bbk$ is of characteristic zero, pushforwards of partial Springer resolutions have the same direct summands as the Springer sheaf $p_*\bbk$, $p:T^*(GL_k/B)\to \cN_{\mathfrak{gl}_k}$.
	This means that we only need the flag sheaves of type $\beta_\cla$ where all $\lambda^i$'s are equal to $1$.
	This proves the second claim.
\end{proof}

\begin{rmk}\label{rmk:flag-sheaves-semicusp}
	Denote restriction of flag sheaf of type $\Ibe_\cla$ to $\Rep^\reg_{n\delta}$ by $\cL^\reg_\cla$, and $\cL^\reg \coloneqq \bigoplus_\cla \cL^\reg_\cla$.
	Then by the discussion above that $\CA(n\delta) \simeq \Ext^*(\cL^\reg,\cL^\reg)$.
\end{rmk}

The proof of \cref{prop:proj-gen} shows that $e$ is somewhat redundant.
While we could choose a smaller idempotent, this one produces more convenient diagrammatics.
On the other hand, we can use the smaller set of idempotents to describe the polyhereditary structure on $\CA(n\delta)$ inherited from \cref{thm:Schur-polyher-order}. 
Denote $\Lambda'_{n\delta}\coloneqq \{ (\lambda,\mu) \}\subset \Lambda_{n\delta}$, where $(\lambda,\mu)\coloneqq ((k^\bullet\delta^\bullet,k^\circ\delta^\circ),\lambda,\mu)$ in the notations of~\eqref{eq:indexing-set}.
This is precisely the subset of images of $\cla$'s appearing in~\eqref{eq:small-idems}.

Restricting polyhereditary evenness theory from $\Rep_{n\delta}$ to $\Rep^\reg_{n\delta}$, \cref{thm:Schur-polyher-order} implies
\begin{thm}\label{thm:curve-Schur-poly}
	The semicuspidal algebra $\CA(n\delta)$ is polyhereditary.
	The polyhereditary structure is given by the sequence of idempotents $1_{\lambda,\mu}$, $(\lambda,\mu)\in\Lambda'_{n\delta}$ in the lexicographic order.\qed
\end{thm}

\begin{expl}
	Let $\xi=2\delta$. \cref{thm:Schur-polyher-order} yields the following sequence of idempotents:
	\begin{align*}
		1_{(2\alpha_1,2\alpha_0)}&<1_{(\alpha_1,\delta,\alpha_0)}<1_{(\alpha_1,\alpha_0,\alpha_1,\alpha_0)}<1_{(\alpha_1,\alpha_0,\delta)}<1_{(\delta,\alpha_1,\alpha_0)}\\
		&<1_{(2\delta)}<1_{(\delta,\delta)}<1_{(\delta,\alpha_0,\alpha_1)}<1_{(2\alpha_0,2\alpha_1)}<1_{(\alpha_0,\alpha_1,\alpha_0,\alpha_1)},
	\end{align*}
	where the second line consists of idempotents in $\Lambda'_{2\delta}$, and can be equivalently rewritten as
	\[
		e_{(2\varepsilon)}< e_{(\varepsilon,\varepsilon)} < e_{(\varepsilon,\tau)} < e_{(2\tau)} < e_{(\tau,\tau)}.
	\]
\end{expl}

\subsection{Diagrammatics of $e\CA(n\delta)e$}\label{subs:thickcals-cusp-QS}
In view of \cref{prop:proj-gen}, let us introduce new Schur diagrammatics for images of certain elements of $eA(n\delta)e$ in $e\CA(n\delta)e$. 
First, denote 
\[
\tikz[thick,xscale=.2,yscale=.2,font=\footnotesize]{
	\draw (0,0) -- (0,5);
	\ntxt{0}{6}{$n\tau$}
	
	\ntxt{2}{2.5}{$\coloneqq$}
	
	\draw (4.5,0) -- (4.5,5);
	\ntxt{4.5}{6}{$n\alpha_0$}
	
	\draw (7.5,0) -- (7.5,5);
	\ntxt{7.5}{6}{$n\alpha_1$}
	
	\ntxt{8.5}{2.5}{,}
	}
\qquad\qquad
\tikz[thick,xscale=.2,yscale=.2,font=\footnotesize]{
	\draw [color=cyan!90!white] (0,0) --(0,5);
	\ntxt{0}{6}{$n\varepsilon$}
	
	\ntxt{2}{2.5}{$\coloneqq$}
	
	\draw (4.5,0) -- (4.5,5);
	\ntxt{4.5}{6}{$n\delta$}
	
	\ntxt{6}{2.5}{.}
}
\]
For $\cla=(\lambda_1c_1,\ldots,\lambda_rc_r)\in \cI^{(n)}$, we draw the idempotent $e_\cla$ as $r$ parallel vertical lines with labels $\lambda_1c_1,\ldots,\lambda_rc_r$.
We allow our strands to change colour via splits and merges:
\[
	\tikz[thick,xscale=.2,yscale=.2,font=\footnotesize]{
    \draw (0,0) -- (0,2);
	\draw [color=cyan!90!white] (0,2) -- (0,4);
    \ntxt{0}{2}{$-$}
    \ntxt{0}{-1}{$a\tau$}
    \ntxt{0}{5}{$a\varepsilon$}
    
    \ntxt{3}{2}{$\coloneqq$}
    
    \dmerge{6}{0}{10}{4}
	\ntxt{6}{-1}{$a\alpha_0$}	
	\ntxt{10}{-1}{$a\alpha_1$}	
	\ntxt{8}{5}{$a\delta$}	

    \ntxt{12}{1}{,}	
	}\qquad\qquad
	\tikz[thick,xscale=.2,yscale=.2,font=\footnotesize]{
	\draw [color=cyan!90!white] (0,0) -- (0,2);
	\draw (0,2) -- (0,4);
    \ntxt{0}{2}{$-$}
    \ntxt{0}{-1}{$a\varepsilon$}
    \ntxt{0}{5}{$a\tau$}
    
    \ntxt{3}{2}{$\coloneqq$}
    
    \dsplit{6}{0}{10}{4}
	\ntxt{6}{5}{$a\alpha_0$}	
	\ntxt{10}{5}{$a\alpha_1$}	
	\ntxt{8}{-1}{$a\delta$}	
	
	\ntxt{12}{1}{.}	
	}
\]

For strands of colour $\tau$, we write
$$
\tikz[thick,xscale=.2,yscale=.2,font=\footnotesize]{
	\draw (3,0) -- (3,1);
	\ntxt{3}{-1.5}{$(a+b)\tau$}
	\dsplit{0}{1}{6}{5.5}
	\ntxt{0}{6.5}{$a\tau$}
	\ntxt{6}{6.5}{$b\tau$}
	
	\ntxt{8}{2}{$\coloneqq$}

	\draw (10.5,2) -- (10.5,5.5);
	\crosin{13.5}{2}{16.5}{5.5}
	\draw (19.5,2) -- (19.5,5.5);

	\dsplit{10.5}{0}{13.5}{2}
	\ntxt{10.5}{6.5}{$a\alpha_0$}	
	\ntxt{13.5}{6.5}{$a\alpha_1$}	
	\ntxt{11.5}{-1.5}{$(a+b)\alpha_0$}	
	
	\dsplit{16.5}{0}{19.5}{2}
	\ntxt{16.5}{6.5}{$b\alpha_0$}	
	\ntxt{19.5}{6.5}{$b\alpha_1$}	
    \ntxt{18.5}{-1.5}{$(a+b)\alpha_1$}	

	\ntxt{21}{2}{,}
	}\qquad
	\tikz[thick,xscale=.2,yscale=.2,font=\footnotesize]{
	\draw (3,4.5) -- (3,5.5);
	\ntxt{3}{6.5}{$(a+b)\tau$}
	\dmerge{0}{0}{6}{4.5}
	\ntxt{0}{-1.5}{$a\tau$}
	\ntxt{6}{-1.5}{$b\tau$}
	
	\ntxt{8}{2}{$\coloneqq$}
	
	\draw (10.5,0) -- (10.5,3.5);
	\crosin{13.5}{0}{16.5}{3.5}
	\draw (19.5,0) -- (19.5,3.5);
	
	\dmerge{10.5}{3.5}{13.5}{5.5}
	\ntxt{10.5}{-1.5}{$a\alpha_0$}	
	\ntxt{13.5}{-1.5}{$a\alpha_1$}	
	\ntxt{11.5}{6.5}{$(a+b)\alpha_0$}	
	
	\dmerge{16.5}{3.5}{19.5}{5.5}
	\ntxt{16.5}{-1.5}{$b\alpha_0$}	
	\ntxt{19.5}{-1.5}{$b\alpha_1$}	
	\ntxt{18.5}{6.5}{$(a+b)\alpha_1$}	
	
	\ntxt{21}{2}{.}
}
$$
For strands of colour $\varepsilon$, splits and merges are defined as images of $S_{(a\delta,b\delta),(a+b)\delta}$, $M_{(a+b)\delta,(a\delta,b\delta)}$ from \cref{def:quiver-Schur-SM} respectively.

\begin{lem}
	For each $c\in\cI$ splits and merges of colour $c$ are associative, i.e.~\eqref{eq:pic-asso} holds. 
\end{lem}
\begin{proof}
	In this proof, we will use black for thick strands of colour $\alpha_0$ and red for $\alpha_1$ to remove the labels.
	For $c=\varepsilon$, the claim follows directly from associativity of splits/merges in $A(n\delta)$.
	For $c=\tau$, consider the following equations in $A(n\delta)$:
	$$
	\tikz[thick,xscale=.2,yscale=.2,font=\footnotesize]{
		\draw (4,-1) -- (4,2);
		\dsplit{2}{2}{6}{5}
		\draw [color=red!60!white] (7,-1) -- (-1,5);
		
		\ntxt{10}{2}{$=$}
		
		\dsplit{15}{-1}{19}{5}
		\draw [color=red!60!white] (20,-1) -- (12,5);
		
		\ntxt{21}{0}{,}
	}\quad 
	\tikz[thick,xscale=.2,yscale=.2,font=\footnotesize]{
		{\color{red!60!white}\draw (4,-1) -- (4,2);
		\dsplit{2}{2}{6}{5}}
		\draw (7,-1) -- (-1,5);
		
		\ntxt{10}{2}{$=$}
		
		{\color{red!60!white}\dsplit{15}{-1}{19}{5}}
		\draw (20,-1) -- (12,5);
		
		\ntxt{21}{0}{,}
	}
	$$
	together with their reflections with respect to vertical and horizontal axes.
	It is straightforward to check these relations on the polynomial representation of $A(n\delta)$.
	Namely, the formulae in \cref{prop:poly-rep-quiver-Schur} show that the crossings act by multiplication with a symmetric polynomial, which commutes past single-colour splits and merges.
	Thus we can check the associativity relations already in $eA(n\delta)e$:
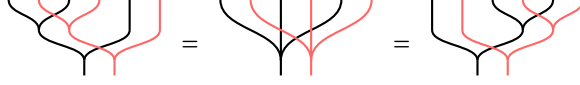
	\[
	\tikz[thick,xscale=.2,yscale=.2,font=\small]{
	\dsplit{0}{3}{4}{5}
	\draw (8,3) -- (8,5);
	\dsplit{2}{1}{8}{3}
	\draw (5,0) -- (5,1);
	{\color{red!60!white}\dsplit{2}{3}{6}{5}
	\draw (10,3) -- (10,5);
	\dsplit{4}{1}{10}{3}
	\draw (7,0) -- (7,1);}
	
	\ntxt{12}{2}{$=$}

	\dsplit{14}{1}{22}{5}
	\draw (18,0) -- (18,5);
	{\color{red!60!white}\dsplit{16}{1}{24}{5}
	\draw (20,0) -- (20,5);}

	\ntxt{26}{2}{$=$}
	
	\dsplit{32}{3}{36}{5}
	\draw (28,3) -- (28,5);
	\dsplit{28}{1}{34}{3}
	\draw (31,0) -- (31,1);
	{\color{red!60!white}\dsplit{34}{3}{38}{5}
	\draw (30,3) -- (30,5);
	\dsplit{30}{1}{36}{3}
	\draw (33,0) -- (33,1);}
	}
	\]
	In particular, the same relations hold in $e\CA(n\delta)e$.
\end{proof}

For two strands of the same colour, we define the crossing as merge followed by split, as in~\eqref{eq:dumb-cross}. 
However, since we do not allow idempotents of mixed colour in $\cI$-coloured diagrammatics, we have to define the crossing of two different colours separately:
$$
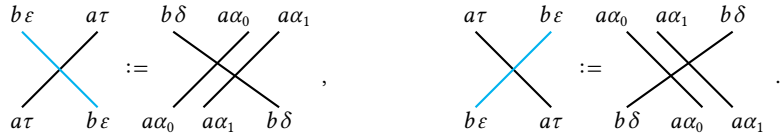
\tikz[thick,xscale=.2,yscale=.2,font=\footnotesize]{
    \draw (0,0) -- (5,5);
    \draw [color=cyan!90!white] (5,0) -- (0,5);
    \ntxt{0}{-1}{$a\tau$}
    \ntxt{5}{-1}{$b\varepsilon$}
    \ntxt{0}{6}{$b\varepsilon$}
    \ntxt{5}{6}{$a\tau$}
    
    \ntxt{7.5}{2.5}{$\coloneqq$}

    \draw (10,0) -- (15,5);
    \draw (12,0) -- (17,5);    
    \draw (10,5) -- (17,0);
    \ntxt{9}{-1}{$a\alpha_0$}
    \ntxt{14}{6}{$a\alpha_0$}
    \ntxt{13}{-1}{$a\alpha_1$}
    \ntxt{18}{6}{$a\alpha_1$}
    \ntxt{17}{-1}{$b\delta$}
    \ntxt{10}{6}{$b\delta$}
	\ntxt{20}{2}{,}
}
\qquad\qquad
\tikz[thick,xscale=.2,yscale=.2,font=\footnotesize]{
    \draw (0,5) -- (5,0);
    \draw [color=cyan!90!white] (5,5) -- (0,0);
    \ntxt{0}{6}{$a\tau$}
    \ntxt{5}{6}{$b\varepsilon$}
    \ntxt{0}{-1}{$b\varepsilon$}
    \ntxt{5}{-1}{$a\tau$}
    
    \ntxt{7.5}{2.5}{$\coloneqq$}

    \draw (10,5) -- (15,0);
    \draw (12,5) -- (17,0);    
    \draw (10,0) -- (17,5);
    \ntxt{9}{6}{$a\alpha_0$}
    \ntxt{14}{-1}{$a\alpha_0$}
    \ntxt{13}{6}{$a\alpha_1$}
    \ntxt{18}{-1}{$a\alpha_1$}
    \ntxt{17}{6}{$b\delta$}
    \ntxt{10}{-1}{$b\delta$}
	\ntxt{20}{2}{.}
}
$$

Finally, we define coupons labelled by polynomials $P\in \Pol_{a\delta}^{\fkS_a\times \fkS_a}$.
For strands of colour $\varepsilon$, we define them simply as images of coupons in $eA(n\delta)e$.
For colour $\tau$, given a polynomial of the form $P=QR$, $Q\in \Bbbk[u_1,\ldots,u_a]^{\fkS_a}$, $R\in \Bbbk[v_1,\ldots,v_a]^{\fkS_a}$, we set
$$
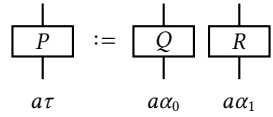
\tikz[thick,xscale=.2,yscale=.2,font=\footnotesize]{
    \draw (0,0) -- (0,1.5);
    \opbox{-2}{1.5}{2}{3.5}{$P$}
    \draw (0,3.5) -- (0,5);
    \ntxt{0}{-1.5}{$a\tau$}
    
    \ntxt{4}{2.5}{$\coloneqq$}
 
    \draw (8,0) -- (8,1.5);
    \opbox{6}{1.5}{10}{3.5}{$Q$}
    \draw (8,3.5) -- (8,5);
    \ntxt{8}{-1.5}{$a\alpha_0$}

    \draw (13,0) -- (13,1.5);
    \opbox{11}{1.5}{15}{3.5}{$R$}
    \draw (13,3.5) -- (13,5);
    \ntxt{13}{-1.5}{$a\alpha_1$}
}
$$
For other polynomials, we extend this definition by linearity.

Let $\cla,\cmu\in \cI^{(n)}$, with underlying uncoloured partitions $\lambda,\mu\in\Comp(n)$. 
Consider the basis $\{\psi^P_{w}\}$ in \cref{prop-basis-QS-nocol} for $e_{\cla} A(\vecn)e_{\cmu}$.
We have an obvious bijection 
$$
\fkS_{\beta_\cla}\backslash\fkS_{n\delta}/\fkS_{\beta_\cmu}\simeq \fkS_{\lambda}\backslash\fkS_{n}/\fkS_\mu\times \fkS_{\lambda}\backslash\fkS_{n}/\fkS_\mu, \qquad w\to (w_0,w_1);
$$
here $w_i$ corresponds to the colour $\alpha_i\in I$.
Let us denote $\psi^P_{w_0,w_1}\coloneqq\psi^P_{w}$.

\begin{prop}\label{cor:span-set-eCAe}
Let $\cla,\cmu\in \cI^{(n)}$, and define $\widetilde{B}_{\cmu'}$ as in \cref{subs:basis-curve-Schur-2c}.
Then the set 
\[
	\{\psi^P_{x,x}: x\in \fkS_{\lambda}\backslash \fkS_{n}/\fkS_{\mu}, P\in B_{\cmu'}\}
\] is a basis of the $\Bbbk$-module $e_\cla \CA(n\delta) e_\cmu$.
\end{prop}
We postpone the proof of this claim until \cref{subs:basis-curve-Schur-2c}.

\begin{expl}\label{ex:diag-new-vs-old}
	Assume $\cla=(2\tau,2\varepsilon,2\tau)$, $\cmu=(3\tau,3\varepsilon)$, and let $x = \bigl(\begin{smallmatrix}
		1 & 2 & 3 & 4 & 5 & 6 \\
		1 & 3 & 5 & 2 & 4 & 6
	  \end{smallmatrix}\bigr)\in \fkS_6$.
	Here is how $\psi_{x,x}^P$ is drawn in quiver Schur diagrammatics and the new diagrammatics:
	
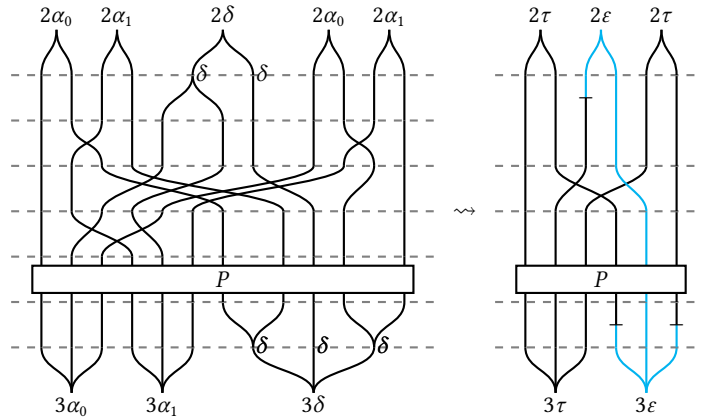
\[
	\tikz[thick,xscale=0.4,yscale=0.6,font=\footnotesize]{
		
    \dsplit{1}{0}{3}{1}
    \draw (2,0) -- (2,1);
    \dsplit{4}{0}{6}{1}
    \draw (5,0) -- (5,1);
    \dsplit{8}{0}{12}{1}
    \draw (10,0) -- (10,1);
    \draw (0,1)-- (14,1)[dashed,color=gray];
    
    \draw (1,1) -- (1,2);
    \draw (2,1) -- (2,2);
    \draw (3,1) -- (3,2);
    \draw (4,1) -- (4,2);
    \draw (5,1) -- (5,2);
    \draw (6,1) -- (6,2);
    \dsplit{7}{1}{9}{2}
    \draw (10,1) -- (10,2);
    \dsplit{11}{1}{13}{2}
    \draw (0,2)-- (14,2)[dashed,color=gray];
    
    \draw (1,2) -- (1,2.2);
    \draw (2,2) -- (2,2.2);
    \draw (3,2) -- (3,2.2);
    \draw (4,2) -- (4,2.2);
    \draw (5,2) -- (5,2.2);
    \draw (6,2) -- (6,2.2);
    \draw (7,2) -- (7,2.2);
    \draw (9,2) -- (9,2.2);
    \draw (10,2) -- (10,2.2);
    \draw (11,2) -- (11,2.2);
    \draw (13,2) -- (13,2.2);
    \draw (1,2.8) -- (1,3);
    \draw (2,2.8) -- (2,3);
    \draw (3,2.8) -- (3,3);
    \draw (4,2.8) -- (4,3);
    \draw (5,2.8) -- (5,3);
    \draw (6,2.8) -- (6,3);
    \draw (7,2.8) -- (7,3);
    \draw (9,2.8) -- (9,3);
    \draw (10,2.8) -- (10,3);
    \draw (11,2.8) -- (11,3);
    \draw (13,2.8) -- (13,3);
    \opbox{0.7}{2.2}{13.3}{2.8}{$P$}
    \draw (0,3)-- (14,3)[dashed,color=gray];

    \draw (1,3) -- (1,4);
    \draw (2,3) .. controls (2,3.4) and (3,3.6) .. (3,4);
    \draw (3,3) .. controls (3,3.4) and (5,3.6) .. (5,4);
    \draw (4,3) .. controls (4,3.4) and (2,3.6) .. (2,4);
    \draw (5,3) .. controls (5,3.4) and (4,3.6) .. (4,4);
    \draw (6,3) -- (6,4);
    \draw (7,3) -- (7,4);
    \draw (9,3) -- (9,4);
    \draw (10,3) -- (10,4);
    \draw (11,3) -- (11,4);
    \draw (13,3) -- (13,4);
    \draw (0,4)-- (14,4)[dashed,color=gray];

    \draw (1,4) -- (1,5);
    \draw (2,4) -- (2,5);
    \draw (3,4) .. controls (3,4.4) and (5,4.6) .. (5,5);
    \draw (4,4) .. controls (4,4.4) and (7,4.6) .. (7,5);
    \draw (5,4) .. controls (5,4.4) and (10,4.6) .. (10,5);
    \draw (6,4) .. controls (6,4.4) and (11,4.6) .. (11,5);
    \draw (7,4) .. controls (7,4.4) and (3,4.6) .. (3,5);
    \draw (9,4) .. controls (9,4.4) and (4,4.6) .. (4,5);
    \draw (10,4) .. controls (10,4.4) and (8,4.6) .. (8,5);
    \draw (11,4) .. controls (11,4.4) and (12,4.6) .. (12,5);
    \draw (13,4) -- (13,5);
    \draw (0,5)-- (14,5)[dashed,color=gray];

    \draw (1,5) -- (1,6);
    \crosin{2}{5}{3}{6}
    \draw (4,5) -- (4,6);
    \draw (5,5) -- (5,6);
    \draw (7,5) -- (7,6);
    \draw (8,5) -- (8,6);
    \draw (10,5) -- (10,6);
    \crosin{11}{5}{12}{6}
    \draw (13,5) -- (13,6);
    \draw (0,6)-- (14,6)[dashed,color=gray];

    \draw (1,6) -- (1,7);
    \draw (2,6) -- (2,7);
    \draw (3,6) -- (3,7);
    \draw (4,6) -- (4,7);
    \dmerge{5}{6}{7}{7}
    \draw (8,6) -- (8,7);
    \draw (10,6) -- (10,7);
    \draw (11,6) -- (11,7);
    \draw (12,6) -- (12,7);
    \draw (13,6) -- (13,7);
    \draw (0,7)-- (14,7)[dashed,color=gray];

    \dmerge{1}{7}{2}{8}
    \dmerge{3}{7}{4}{8}
    \dmerge{6}{7}{8}{8}
    \dmerge{10}{7}{11}{8}
    \dmerge{12}{7}{13}{8}

    \ntxt{2}{-0.3}{$3\alpha_0$}
    \ntxt{5}{-0.3}{$3\alpha_1$}
    \ntxt{10}{-0.3}{$3\delta$}
    
    \ntxt{8.3}{1}{$\delta$}
    \ntxt{10.3}{1}{$\delta$}
    \ntxt{12.3}{1}{$\delta$}
    
    \ntxt{1.5}{8.3}{$2\alpha_0$}
    \ntxt{3.5}{8.3}{$2\alpha_1$}
    \ntxt{7}{8.3}{$2\delta$}
    \ntxt{10.5}{8.3}{$2\alpha_0$}
    \ntxt{12.5}{8.3}{$2\alpha_1$}

    \ntxt{6.3}{7}{$\delta$}
    \ntxt{8.3}{7}{$\delta$}
    
    \ntxt{8.3}{1}{$\delta$}
    \ntxt{10.3}{1}{$\delta$}
    \ntxt{12.3}{1}{$\delta$}

    \ntxt{15}{4}{$\rightsquigarrow$}

    \dsplit{17}{0}{19}{1}
    \draw (18,0) -- (18,1);
    {\color{cyan!90!white}\dsplit{20}{0}{22}{1}
    \draw (21,0) -- (21,1);}
    \draw (16,1)-- (23,1)[dashed,color=gray];

    \draw (17,1) -- (17,2);
    \draw (18,1) -- (18,2);
    \draw (19,1) -- (19,2);
    \draw [color=cyan!90!white] (20,1) -- (20,1.5);
	\draw (20,1.5) -- (20,2);
    \draw [color=cyan!90!white] (21,1) -- (21,2);
    \draw [color=cyan!90!white] (22,1) -- (22,1.5);
	\draw (22,1.5) -- (22,2);
    \ntxt{20}{1.5}{$-$}
    \ntxt{22}{1.5}{$-$}
    \draw (16,2)-- (23,2)[dashed,color=gray];

    \draw (17,2) -- (17,2.2);
    \draw (18,2) -- (18,2.2);
    \draw (19,2) -- (19,2.2);
    \draw (20,2) -- (20,2.2);
    \draw [color=cyan!90!white] (21,2) -- (21,2.2);
    \draw (22,2) -- (22,2.2);
    \draw (17,2.8) -- (17,4);
    \draw (18,2.8) -- (18,4);
    \draw (19,2.8) -- (19,4);
    \draw (20,2.8) -- (20,4);
    \draw [color=cyan!90!white] (21,2.8) -- (21,4);
    \draw (22,2.8) -- (22,4);
    \opbox{16.7}{2.2}{22.3}{2.8}{$P$}
    \draw (16,4)-- (23,4)[dashed,color=gray];

    \draw (17,4) -- (17,5);
    \draw (18,4) .. controls (18,4.4) and (19,4.6) .. (19,5);
    \draw (19,4) .. controls (19,4.4) and (21,4.6) .. (21,5);
    \draw (20,4) .. controls (20,4.4) and (18,4.6) .. (18,5);
    \draw [color=cyan!90!white] (21,4) .. controls (21,4.4) and (20,4.6) .. (20,5);
    \draw (22,4) -- (22,5);
    \draw (16,5)-- (23,5)[dashed,color=gray];

    \draw (17,5) -- (17,7);
    \draw (18,5) -- (18,7);
    \draw (19,5) -- (19,6.5);
	\draw [color=cyan!90!white] (19,6.5) -- (19,7);
    \draw [color=cyan!90!white] (20,5) -- (20,7);
    \draw (21,5) -- (21,7);
    \draw (22,5) -- (22,7);
    \ntxt{19}{6.5}{$-$}
    \draw (16,7)-- (23,7)[dashed,color=gray];

    \dmerge{17}{7}{18}{8}
    {\color{cyan!90!white}\dmerge{19}{7}{20}{8}}
    \dmerge{21}{7}{22}{8}


    \ntxt{18}{-0.3}{$3\tau$}
    \ntxt{21}{-0.3}{$3\varepsilon$}

    \ntxt{17.5}{8.3}{$2\tau$}
    \ntxt{19.5}{8.3}{$2\varepsilon$}
    \ntxt{21.5}{8.3}{$2\tau$}

	}
	\]
\end{expl}

\subsection{Restriction to KLR algebra}\label{subs:KLR-prop-str}
Let $\vecn = n_0\alpha_0 + n_1\alpha_1$, and consider the KLR algebra $A_0(\vecn)$ as in \cref{rmk:KLR-not-Schur}.
A stratification of this algebra was obtained in~\cite{KM_SKAA2017} under the assumption $\ona{char}\bbk > \min(n_0,n_1)$.
The theory we developed in \cref{sec:rep-cS} allows us to remove this condition.

Let us define a evenness sub-theory of $\Ev^\mkp_\vecn$.
Since $\Ev^\mkp_\vecn$ is ersatz-complete, this amounts to choosing a subset of ersatz parity sheaves.
For any $\mkbeta\in\mkp(\vecn)$ and an indecomposable $F\in \Ev(\Rep_\vecn^\mkbeta)$, we denote the corresponding ersatz parity sheaf by $\parsh(\mkbeta,F)$.
Recalling the identification~\eqref{eq:mkbeta-vs-kbeta}, we write
\[
	\Ev^\kp_\vecn \coloneqq \{ \parsh((\kbeta,0),F) : \kbeta\in\kp(\vecn), F\in \Ev(\Rep_\vecn^{(\kbeta,0)}) \}.
\]

\begin{prop}
	The category $\Ev^\kp_\vecn$ is the category $\FShv^0_\vecn$ of all direct summands of flag sheaves $\cL_\Ibe$, $\beta\in\Comp_0(\vecn)$. 
\end{prop}
\begin{proof}
	Consider the natural action of $GL_2$ on $\Rep \Gamma$ by linear combinations of arrows.
	For any $\Iga\in\Comp(\vecn)$, the flag variety $\Fl^\m_\Iga$ has an induced $GL_2$-action, and so $\cL_\Iga$ acquires a natural $GL_2$-equivariant structure if and only if $\Iga\in \Comp_0(\vecn)$.
	On the other hand, the support of an ersatz parity sheaf $\parsh((\kbeta,k),F)$ is not preserved by $GL_2$ for $k>0$.
	Therefore, $\cL_\Iga$ must belong to $\Ev^\kp_\vecn$.

	For the opposite inclusion, the proof of \cref{prop:ersatz-comp}, \cref{rmk:simpler-idems} and the usual nilHecke argument to replace sequences $k\alpha_i$ with $(\alpha_i,\ldots,\alpha_i)$ show that every ersatz parity sheaf $\parsh((\kbeta,0),F)$ occurs as a direct summand of a KLR flag sheaf.
\end{proof}

In particular, we see that $\Ev^\kp_\vecn$ can be seen as an induced evenness theory from \cref{cor:gluing-evenness}, for the decomposition $\Rep_\vecn = \bigsqcup_{\kbeta\in\kp(\vecn)} \Rep^\kbeta_\vecn$.
It is also ersatz-complete.

Analogously to \cref{defn:semicusp-alg}, we can define the semicuspidal quotient $\CA_0(n\delta)$ of $A_0(n\delta)$.
Consider the class of algebras $\mathscr{C}_0$ containing polynomial algebras $\mathscr{P}$, the algebras $\CA_0(n\delta)$ for all $n>0$, as well as their tensor products.
Let 
\[
	\Lambda^0_\vecn = \left\{ (\kbeta,\mu) : \kbeta\in \kp(\vecn), \mu\vdash k^\delta \right\},
\]
where $k^\delta$ is the coefficient of $\delta$ in $\kbeta$.
To each element $(\kbeta,\mu)\in \Lambda^0_\vecn$ we associate a composition $\gamma(\kbeta,\mu)\in \Comp_0(\vecn)$ by concatenating compositions from \cref{rmk:simpler-idems} for real roots, and \eqref{eq:idem-circ} for $\delta$.
Applying \cref{thm:quasi-her-inher}, we obtain an analogue of \cref{thm:Schur-polyher-order}, which removes the characteristic assumption of~\cite[Th.~3]{KM_SKAA2017} for the Kronecker quiver.

\begin{cor}\label{cor:KLR-strat-order}
	The KLR algebra $A_0(\vecn)$ admits a $\mathscr{C}_0$-properly stratified structure, given by the sequence of idempotents $1_{\gamma(\kbeta,\mu)}$, $(\kbeta,\mu)\in\Lambda^0_\vecn$ in the lexicographic order.\qed
\end{cor}

Since convolution algebras are local in the base, the open embedding $\cT_n(\bbP^1)\subset \Rep_{n\delta}$ gives rise to a restriction map $A_0(n\delta)\to A_0(n\delta)/A_0(n\delta)1_{\nc}A_0(n\delta)$.
By \cref{lem:Par-open-surj}, this is a surjection.
In particular, after idempotent truncation with respect to $1_{(n\delta)}$ we obtain:
\begin{cor}\label{cor:surj-taut-classes}
	The restriction map $H^*(\Rep_{n\delta})\to H^*(\cT_n)$ is surjective.\qed
\end{cor}
This is in sharp contradiction with \cite[§8.9]{MakMin_KLR2023}; see \cref{subs:erratum} for discussion and erratum.

\subsection{Other affine quivers}
As mentioned above, a proper stratification for KLR algebras of affine type was constructed in~\cite{KM_SKAA2017} under the assumption that $\ona{char}\bbk$ is not too small.
We expect this result can be improved. 
\begin{conj}
	Let $Q$ be an affine quiver without oriented cycle, $\vecn$ dimension vector, and $A_0(\vecn)$ the corresponding KLR algebra. 
	The preorder of~\cite[Lm.~3.2]{KM_SKAA2017} is properly stratifying in arbitrary characteristic.
\end{conj}
Note that in type $A$, we have a stratification on $A_0(\vecn)$ by \cref{ssec:another-polyher}, but it is not compatible with the required preorder.

For the Kronecker quiver, this conjecture holds by \cref{cor:KLR-strat-order}.
For other quivers, our strategy of constructing a properly stratified preorder on the category $\FShv_\vecn^0$ is complicated by the following fact: in general, such preorder cannot be obtained by an application of \cref{thm:quasi-her-inher}.

\begin{expl}
 Let $Q$ be a cyclic quiver on 3 vertices without an oriented cycle:
 \[
 \begin{tikzcd}
 &[-10pt] 2\ar[dr,"h_1"] &[-10pt] \\[-10pt]
 1\ar[ur,"h_0"]\ar[rr,"h_2"] & & 0
 \end{tikzcd}
 \]
 Fix the dimension vector $\delta \coloneqq \alpha_0+\alpha_1+\alpha_2$, and $\ona{char}\bbk=0$.
 Let $\Ibe_{1}=(\alpha_0,\alpha_2,\alpha_1)$ and $\Ibe_{2}=(\alpha_0,\alpha_1,\alpha_2)$.
 The largest ideal $J\subset A_0(\delta)$ in the proper stratification of $A_0(\delta)$ constructed in~\cite{KM_SKAA2017} is generated by the idempotents $1_\beta$, $\beta\in \Comp_0(\delta)\setminus\{\beta_1,\beta_2\}$.
 The quotient $A_0(\delta)/J$ is $Z_{A_2}[x]$, where $Z_{A_2}$ is the zigzag algebra of type $A_2$:
 \begin{equation}\label{eq:zigzag-A2-def}
 Z_{A_2} \coloneqq \bbk Q/(yzy,zyz),\qquad Q = \begin{tikzcd}\tau\ar[r,shift left=.5ex,"y"] & \varepsilon\ar[l,shift left=.5ex,"z"]\end{tikzcd}.
 \end{equation}
 
 Consider the KLR sheaf $\cL=\bigoplus_{\beta\in \Comp_0(\delta)}\cL_\beta$, so that $A_0(\delta)=\Ext^*(\cL,\cL)$. 
 Suppose that the ideal $J\subset A_0(\delta)$ comes from a properly stratified evenness theory. 
 More precisely, we assume that there is an evenness theory on $\Rep_\delta$ obtained by gluing two evenness theories on an open $j:U\hookrightarrow \Rep_\delta$ and its complement $\Rep_\delta\backslash U$, such that the restriction morphism $\Ext^*(\cL,\cL)\to \Ext^*(j^*\cL,j^*\cL)$ coincides with $A_0(\delta)\to A_0(\delta)/J$.
 Then the supports of flag sheaves $\cL_{\Ibe}$, $\Ibe\not\in \{\Ibe_1,\Ibe_2\}$ must be disjoint from $U$.
 Computing these supports, we obtain that 
 \begin{align*}
 U \subset U'\coloneqq \mathsf{Rep}_\delta \setminus {[\{h_1=0\text{ or }h_0=h_2=0\}/\bbG_m^3]} \simeq \bbP^1 \times B\bbG_m.
 \end{align*}

 Denote by $j'$ the inclusion $U'\hookrightarrow\Rep_\delta$. 
 We have a chain of algebra homomorphisms $\Ext^*(\cL,\cL)\to \Ext^*(j'^*\cL,j'^*\cL)\to \Ext^*(j^*\cL,j^*\cL)$ and the identification $\Ext^*(j^*\cL,j^*\cL)= Z_{A_2}[x]$ is such that the two idempotents of the zigzag algebra correspond to the decomposition $j^*\cL=j^*\cL_{\beta_1}\oplus j^*\cL_{\beta_2}$. 
 Since $yz\ne 0$ in $Z_{A_2}[x]$, we must exhibit a non-zero composition of extensions $j^*\cL_{\beta_2}\to j^*\cL_{\beta_1}\to j^*\cL_{\beta_2}$. 
 On the other hand, it is easy to see that $j'^*\cL_{\Ibe_{1}}$, resp. $j'^*\cL_{\Ibe_{2}}$ is the constant sheaf on $\bbP^1$, resp. a point.
 One deduces by a direct computation that every composition of extensions $j'^*\cL_{\beta_2}\to j'^*\cL_{\beta_1}\to j'^*\cL_{\beta_2}$ is zero. 
 We arrived at a contradiction.
\end{expl}

In the example above, the algebra $Z_{A_2}[x]$ can still be recovered geometrically, if one instead restricts $\cL_{\Ibe_{1}}\oplus \cL_{\Ibe_{2}}$ \textit{microlocally} to the complement of singular supports of the non-cuspidal sheaves.
Since gluing of evenness theories can be set up in greater generality (see \cref{rmk:exotic-sheaves}), we expect that our methods will still apply in an appropriate setup of microlocal sheaves with modular coefficients.

\medskip
\section{Schur algebras on ``dotted'' curves}\label{sec:schur-dotted}
In this section, we express the (truncated) semicuspidal algebra $e\CA(n\delta)e$ in terms of another convolution algebra, and in doing so prove \cref{cor:span-set-eCAe}.

\subsection{Curve Schur algebras}
\label{subs:curve-Schur}
Let us quickly recall the geometric setup of~\cite{MakMin_KLR2023}, and refer the reader to \emph{op.cit.} for more details.

Let $C$ be a smooth projective curve, and $\bbk$ a field.
For any $n>0$, let $\cT_n = \cT_n(C)$ be the moduli stack of torsion sheaves on $C$ of length $n$.
For any composition $\lambda\in\Comp(n)$ consider the associated moduli of flags of torsion sheaves:
\begin{align*}
	\cF_{\lambda} = \{ 0 = \cE_0\subset \cE_1\subset \ldots \subset \cE_{\ell(\lambda)} : \cE_j/\cE_{j-1}\in \cT_{\lambda^j} \}, \qquad \cF_n\coloneqq \bigsqcup_{\lambda\in \Comp(n)}\cF_{\lambda}.
\end{align*}
There are two natural maps out of $\cF_{\lambda}$:
\begin{align*}
	p:\cF_{\lambda}&\to \cT_n,  &  p(\cE_1\subset \ldots \subset \cE_k)&= \cE_k, \\
	q:\cF_{\lambda}&\to \cT_\lambda,  &  q(\cE_1\subset \ldots \subset \cE_k)&= (\cE_1,\ldots,\cE_k/\cE_{k-1}),
\end{align*}
where $\cT_\lambda$ stands for the product $\prod_{j=1}^k \cT_{\lambda^j}$.
Both $\cT_n$ and $\cF_\lambda$ are smooth stacks.
Using the maps $p$, we can form \textit{curve Steinberg stacks}:
\[
	\cZ_n = \bigsqcup_{\lambda,\mu\in \Comp(n)} \cZ_{\lambda,\mu},\qquad \cZ_{\lambda,\mu} \coloneqq \cF_{\lambda}\times_{\cT_n} \cF_{\mu}.
\]
As explained in~\cite[Sec.~2]{MakMin_KLR2023}, all of the stacks above are global quotients by $GL_n$:
\begin{gather*}
	\cT_n \simeq [X_n/GL_n],\quad \cF_\lambda\simeq [Y_\lambda/GL_n], \quad \cZ_{\lambda,\mu}\simeq [Z_{\lambda,\mu}/GL_n];\\
	X_n \coloneqq \{ s:\bbC^n\otimes \cO\twoheadrightarrow \cE : H^0(s)\text{ is an iso} \},\\
	Y_\lambda \coloneqq \{ (V^\bullet, s) : H^0(s|_{V^i})\text{ is an iso for all }i \}\subset \bfF_\lambda \times X_n,\quad Z_{\lambda,\mu} \coloneqq Y_\lambda\times_{X_n} Y_\mu.
\end{gather*}
In particular $H_*(\cZ_n) \simeq H^{GL_n}_*(Z_n)$, so the formalism of \cref{subs:conv-algs} applies and we get a convolution product on $\cS_n\coloneqq H_*(\cZ_n)$.
We call $\cS_n$ the \textit{curve Schur algebra}.

The algebra $\cS_n$ acts naturally on $H^*(\cF_n)$; we call this $\cS_n$-module the \textit{polynomial representation}.
We have the following isomorphisms of cohomology rings:
\begin{align*}
	H^*(\cF_{\lambda},\bbk)\simeq \bigotimes\nolimits_j H^*(\cT_{\lambda^j},\bbk),\qquad H^*(\cT_1) \simeq H^*(C,\bbk)[x].
\end{align*}
Let $\lambda = 1^n= (1,1,\ldots,1)$.
If $\bbQ\subset \bbk$, pullback along $\cF_{1^n}\to \cT_n$ induces an isomorphism
\begin{align}\label{eq:tor-vs-poly}
	H^*(\cT_n,\bbk) \xrightarrow{\sim} \left( H^*(C^n,\bbk)[x_1,\ldots,x_n] \right)^{\fkS_n}.
\end{align}
If $C = \bbP^1$, the map~\eqref{eq:tor-vs-poly} is a monomorphism over $\bbZ$, and its image can be explicitly computed; see \cref{subs:erratum} for details.
As a consequence, $H^*(\cT_n,\bbk)$ is a free $\bbk$-module for any $\bbk$ by universal coefficients.
In view of this, \textit{from now on we always assume the following:}
\begin{equation}\label{eq:assumption-on-curve}
	\text{Either $C = \bbP^1$ or $\bbQ\subset \bbk$.}
\end{equation}
The reader is invited to assume $C = \bbP^1$, since this is the case that bears relevance to the preceding sections.

As with quiver Schur algebras, given $\lambda'\vDash \lambda$ we can define split and merge operators $S_\lambda^{\lambda'}$, resp. $M^\lambda_{\lambda'}$ as the fundamental class of $Y_{\lambda'} = Y_{\lambda'}\times_{Y_{\lambda}}Y_{\lambda}$ inside $Z_{\lambda',\lambda} = Y_{\lambda'}\times_{X_n}Y_{\lambda}$, resp. inside $Z_{\lambda,\lambda'}$.
Further, let $\lambda = (\lambda^1,\ldots, \lambda^k)$, and let $\lambda' = (\lambda^1,\ldots,\lambda^i_{(1)},\lambda^i_{(2)},\ldots, \lambda^k)$, $\lambda'' = (\lambda^1,\ldots,\lambda^i_{(2)},\lambda^i_{(1)}\ldots, \lambda^k)$ be obtained from $\lambda$ by splitting $\lambda^i = \lambda^i_{(1)}+\lambda^i_{(2)}$ into $\lambda^i_{(1)}$, $\lambda^i_{(2)}$ in two different ways.
We define the \textit{elementary crossing} as 
\[
	R_{\lambda'}^{\lambda''} \coloneqq S^{\lambda''}_\lambda M_{\lambda'}^\lambda.
\]
Together, these operators give rise to the elements $\psi_w^P$ defined via the same diagrammatics as in~\cref{subs:qschur-basis}. 
The following statement was proved in~\cite[Sec.~4]{MakMin_KLR2023} using equivariant localization.
\begin{prop}\label{prop:old-basis-curve}
	For any $\lambda$, fix a $\bbk$-basis $B_{\lambda}$ of $H^*(\cF_{\lambda},\bbk)$.
	The set $\{\psi_w^P : w\in \fkS_{\lambda}\backslash \fkS_{n}/\fkS_{\mu}, P\in B_{\lambda'}\}$ is a basis of $\cS_{\lambda,\mu}$.
\end{prop}

\subsection{Extended Schur algebra and diagrammatics}
\label{subs:col-curve-Schur}
Let us slighty enlarge the algebra $\cS_n$.
Fix a point $c\in C$; by abuse of notation we denote the class $[c]\in H^2(C,\bbk)$ by $c$ too.
Consider the closed substack $\cT_{n\varepsilon}=\{\cO_c\otimes \bbC^n\}\subset \cT_{n}$; we have $\cT_{n\varepsilon} \simeq [\pt/GL_n]$.
Let us further denote $\cT_{n\tau} = \cT_n$.

Recall the set $\cI = \{\tau,\varepsilon\}$.
For any $\cI$-coloured partition $\cla\in \cI^{(n)}$ with underlying uncoloured partition $\lambda$, consider the following flag stacks:
\begin{align*}
	\cF_{\cla} = \{ 0 = \cE^0\subset \cE^1\subset \ldots \subset \cE^k : \cE^j/\cE^{j-1}\in \cT_{\cla^j} \}.
\end{align*}

We will sometimes identify uncoloured partitions $(\mu^1,\ldots,\mu^k)$ with the partitions $(\mu^1\tau,\ldots,\mu^k\tau)$ coloured purely with $\tau$.
Denoting $\cT_{\cla} = \prod_j \cT_{\cla^j}$, we have an obvious isomorphism $\cF_{\cla}\simeq \cF_\lambda \times_{\cT_\lambda} \cT_{\cla}$.
We have explicit presentations as quotient stacks:
\[
	\cF_{\cla} \simeq [Y_\cla/GL_n],\quad Y_\cla \coloneqq \{ (V^\bullet, s) : H^0(s|_{V^i})\text{ is an iso}, s(V^i)/s(V^{i-1})\in \cT_{\cla^i}\text{ }\forall i \}\subset \bfF_\lambda \times X_n.
\]

\begin{lem}\label{lem:Y-is-smooth}
	$Y_\cla$ is a smooth variety for any $\cla$; in particular, $\cF_{\cla}$ is a smooth stack. 
\end{lem}
\begin{proof}
	Let $v_i = \lambda^1+\ldots+\lambda^i$.
	It is clear that $Y_\cla$ is a fibration over $\bfF_\lambda$ with fiber isomorphic to
	\[
	Q_\cla = \{ s : H^0(s|_{\bbC^{v_i}})\text{ is an iso}, s(\bbC^{v_i})/s(\bbC^{v_{i-1}})\in \cT_{\cla^i}\text{ for all }i \}\subset X_n.
	\]
	In~\cite[Lem.~3.3]{MakMin_KLR2023} we showed that $Q_\lambda$ is an affine bundle over $X_\lambda$.
	The restriction of this bundle to $X_\cla$ is precisely $Q_\cla$, and thus the latter is smooth.
\end{proof}

As before, define the Steinberg varieties
\[
	\cZ^\cI_n = \bigsqcup_{\lambda,\mu\in \cI^{(n)}} \cZ^\cI_{\cla,\cmu},\qquad \cZ^\cI_{\cla,\cmu} = \cF_{\cla}\times_{\cT_n} \cF_{\cmu}.
\]

\begin{defn}
	The \textit{extended Schur algebra} of $C$ is the convolution algebra $\cS^\cI_n\coloneqq H_*(\cZ^\cI_n,\bbk)$.
\end{defn}

As before, $\cS^\cI_n$ naturally acts on the \textit{polynomial representation} 
\begin{equation}\label{eq:poly-rep-def}
	\bfP_n \coloneqq \bigoplus_{\cla \in \cI^{(n)}} H^*(\cF_\cla,\bbk).
\end{equation}

Let $\cla'\vDash\cla$, and consider the underlying uncoloured compositions $\lambda'\vDash \lambda$.
We have natural proper maps $\cF_{\lambda'}\to \cF_\lambda$, $\cF_{\cla}\hookrightarrow \cF_\lambda$, and so we can define $\cF^{\lambda'}_{\cla} \coloneqq \cF_{\cla}\times_{\cF_\lambda} \cF_{\lambda'}$.
This is a closed subvariety of $\cF_{\lambda'}$; moreover, it is fully contained in $\cF_{\cla'}\subset \cF_{\lambda'}$.
In particular, we have closed embeddings $\iota: \cF^{\lambda'}_{\cla} \to \cZ^\cI_{\cla,\cla'}$, $\iota': \cF^{\lambda'}_{\cla} \to \cZ^\cI_{\cla',\cla}$.

\begin{defn}
	We call $S_{\cla}^{\cla'} \coloneqq \iota'_*[\cF^{\lambda'}_{\cla}]\in \cS^\cI_{\cla',\cla}$ a \textit{split}, and $M_{\cla'}^{\cla} \coloneqq \iota_*[\cF^{\lambda'}_{\cla}]\in \cS^\cI_{\cla,\cla'}$ a \textit{merge}.
\end{defn}

\begin{rmk}
	When $\cla$ is entirely of colour $\tau$, this definition coincides with splits and merges in $\cS_n$.
	In general, the map $\cF^{\lambda'}_{\cla}\to \cF_{\cla'}$ is a closed embedding, and not an isomorphism.
\end{rmk}

One can check analogously to~\cite[Lem.~4.6]{MakMin_KLR2023} that for $\cla''\vDash\cla'\vDash\cla$, we have $S_{\cla'}^{\cla''}S_{\cla}^{\cla'} = S_{\cla}^{\cla''}$, $M_{\cla'}^{\cla}M_{\cla''}^{\cla'} = M_{\cla''}^{\cla}$. 
This allows us to denote $S_{\cla}^{\cla'}$ and $M_{\cla'}^{\cla}$ by split and merge in $\cI$-coloured diagrammatics.
Recalling the notations of~\cref{subs:curve-Schur}, consider $\cla,\cla',\cla''\in\cI^{(n)}$ with underlying compositions $\lambda$, $\lambda'$, $\lambda''$, where $\lambda^i$, $\lambda^i_{(1)}$, $\lambda^i_{(2)}$ are all of colour $\varepsilon$.
Then we can again define elementary crossing of colour $\varepsilon$ as 
\[
	R_{\cla'}^{\cla''}\coloneqq S^{\cla''}_\cla M_{\cla'}^\cla.
\]
Since we do not consider splits and merges of mixed colour, the crossings of strands of different colours have to be defined separately.
Let 
\[
	\cla = (\lambda^1c^1,\ldots,a\tau, b\varepsilon,\ldots, \lambda^kc^k),\qquad \cmu = (\lambda^1c^1,\ldots,b\varepsilon, a\tau,\ldots, \lambda^kc^k).
\]
\begin{defn}\label{def:mc-crossing}
	The \textit{multicoloured elementary crossing} is given by $R_\cmu^\cla \coloneqq [Z_{\cla,\cmu}]\in \cS_{\cla,\cmu}^\cI$.
\end{defn}

Let $\cla_1$, $\cla_2$ be two $\cI$-compositions with the same underlying composition $\lambda = (\lambda^1,\ldots,\lambda^k)$.
Define $\cla_{12} = (\lambda^1i^1,\ldots,\lambda^ki^k)$, where $i^j = \tau$ if $j$-th colour in both $\cla_1$ and $\cla_2$ is $\tau$, and $\varepsilon$ otherwise.
Closed embeddings $\cT_{m\varepsilon}\subset \cT_{m\tau}$ yield a closed embedding $\iota:\cF_{\cla_{12}}\subset \cF_{\cla_1}\times_{\cT_n}\cF_{\cla_2} = \cZ^{\cI}_{\cla_1,\cla_2}$.

\begin{defn}
	We call $C_{\cla_2}^{\cla_1} \coloneqq \iota_*[\cF_{\cla_{12}}]\in \cS^\cI_{\cla_1,\cla_2}$ the \textit{colour change}.
\end{defn}

Diagrammatically, we draw strands of colour $\tau$ in black, strands of colour $\varepsilon$ in cyan, and colour change as a horizontal dash. 
Finally, for any $\cla\in\cI^{(n)}$ and $P\in H^*(\cF_\cla)$ we have the element $P\in \cS_{\cla,\cla}^\cI$ corresponding to the image of $P$ under the diagonal map $\cF_\cla\to \cZ^\cI_{\cla,\cla} = \cF_\cla\times_{\cT_n}\cF_\cla$.
In particular, for $\cla = n\tau$ we have $P\in H^*(\cT_n)$, and for $\cla = n\varepsilon$ we have $P\in H^*(\cT_{n\varepsilon}) = H^*_{GL_n}(\pt) = \bbk[x_1,\ldots,x_n]^{\fkS_n}$.
Diagrammatically, we draw these operators as coupons.

\subsection{Comparison with the semicuspidal algebra}\label{ssec:comp-semicups}
In this subsection, let us fix $C=\bbP^1$, and $c = \infty\in \bbP^1$.
By the argument in \cref{subs:def-semicusp}, the semicuspidal algebra $\CA(n\delta)$ is precisely the restriction of $A^\m(n\delta)$ to $\Rep^\reg_{n\delta} \simeq \cT_n(\bbP^1)$.
Since $\cS_n^\cI$ is also a convolution algebra over $\cT_n(\bbP^1)$, it makes sense to compare them.

\begin{prop}\label{prop:regular-truncate}
	Let $\cla\in\cI^{(n)}$, and $\Ibe_\cla\in \Comp(n\delta)$ the corresponding composition.
	Denote $\Fl^\reg_{\cla} = \Fl_{\Ibe_\cla}\times_{\Rep^\reg_{n\delta}}\Rep_{n\delta}$.
	Then we have an isomorphism $\Fl^\reg_{\Ibe_\cla}\simeq \cF_\cla$.
	In particular, restriction to $\Rep^\reg_{n\delta}$ induces an isomorphism $e\CA(n\delta)e\xra{\sim} \cS^\cI_n$.
\end{prop}
\begin{proof}
	When $\cla=\lambda$ is fully of colour $\tau$, this is the content of~\cite[Lem.~8.19]{MakMin_KLR2023}.
	In general, $\Fl^\reg_{\cla}$ is closed in $\Fl^\reg_{\lambda}$; the defining condition is that for each fragment of the flag $\ldots\subset V^{i-1}\subset V^i\subset V^{i+1}\subset \ldots$ with $\dim V^i/V^{i-1} = k\alpha_0$, $\dim V^{i+1}/V^{i} = k\alpha_1$, the restriction of $h_\n$ to $V^{i+1}/V^{i-1}$ vanishes.
	Under the equivalence~\eqref{eq:tor-iso-reg}, this translates into the requirement that each graded piece of colour $\varepsilon$ has to be isomorphic to $\cO_c^{\oplus k}$. This is precisely the defining condition of $\cF_\cla$ in $\cF_\lambda$, hence the first claim.
	The second claim follows from \cref{rmk:flag-sheaves-semicusp}.
\end{proof}

\begin{cor}\label{thm:End(P+)}
	The semicuspidal algebra $\CA(n\delta)$ is Morita-equivalent to $\cS^\cI_n(\bbP^1)$ over an arbitrary field $\bbk$.\qed
\end{cor}

Denote the isomorphism of \cref{prop:regular-truncate} by $\Phi:eCA(n\delta)e\to \cS^\cI_n$.
Let us study what $\Phi$ does to the diagrammatic generators of $eCA(n\delta)e$.

\begin{prop}\label{prop:gen-images}
	Under $\Phi$, splits go to splits, merges go to merges, elementary crossings to elementary crossings, and colour changes to colour changes.
\end{prop}
\begin{proof}
	For splits and merges of colour $\tau$, this follows from~\cite[Lem.~8.7]{MakMin_KLR2023}.
	For splits and merges of colour $\varepsilon$, the claim follows immediately from~\cref{prop:regular-truncate}. 
	The proof of~\cref{prop:regular-truncate} also shows that colour changes restrict to colour changes. 
	Since the elementary crossing of one colour is defined as composition of split and merge, the claim holds for them as well.
	
	It remains to show that multicoloured crossings are preserved under $\Phi$.
	Let us compute the restriction of their defining correspondence to the regular representations.
	To save space, we will only consider the crossing from $(a\tau,b\varepsilon)$ to $(b\varepsilon,a\tau)$, $n = a+b$, the swapped case being completely analogous.
	Recall the convolution diagram for quiver Schur algebra:
	\[
		\begin{tikzcd}
			Z_{\kappa,\lambda}\times Z_{\lambda,\mu}\ar[d] & Z_{\kappa,\lambda,\mu}\ar[l]\ar[r]\ar[d] & Z_{\kappa,\mu}\ar[d]\\
			\widetilde{\bfF}_\kappa\times \widetilde{\bfF}_\lambda\times \widetilde{\bfF}_\lambda\times \widetilde{\bfF}_\mu & \widetilde{\bfF}_\kappa\times \widetilde{\bfF}_\lambda\times \widetilde{\bfF}_\mu\ar[l]\ar[r] & \widetilde{\bfF}_\kappa\times \widetilde{\bfF}_\mu
		\end{tikzcd}
	\]
	where we denote $\kappa = (a\alpha_0,a\alpha_1,b\delta)$, $\lambda = (a\alpha_0,b\delta,a\alpha_1)$, $\mu = (b\delta,a\alpha_0,a\alpha_1)$.
	Write $\Rep^\reg_{n\delta} = [E_{n\delta}^\reg/G_{n\delta}]$, and $\widetilde{\bfF}_\lambda^\reg = \widetilde{\bfF}_\lambda \cap (E_{n\delta}^\reg \times \bfF_\lambda)$.
	We want to show that after restricting to the regular locus $\widetilde{\bfF}_\kappa^\reg\times \widetilde{\bfF}_\mu^\reg$, the following equality holds:
	\[
		[Z_{\kappa,\lambda}]*[Z_{\lambda,\mu}] = [Z_{\kappa,\mu}].
	\]
	We will denote this restriction by a circle in superscript.
	First of all, one checks from the definitions that 
	\[
		Z_{\kappa,\lambda}^\circ = \left\{ \begin{tikzcd} V\subset\bbC^{a+b}\ar[r,shift right=.5ex,"h_\n"']\ar[r,shift left=.5ex,"h_\s"] & \bbC^{a+b}\supset W \end{tikzcd} \mid \dim V = b, \dim W = a, h_\n(V)\subset W \right\},
	\]
	and $Z_{\lambda,\mu}^\circ = Z_{\kappa,\mu}^\circ = Z_{\kappa,\lambda,\mu}^\circ = \{x(V) = 0\} \subset Z_{\kappa,\lambda}^\circ$.
	Furthermore, $\widetilde{\bfF}_\lambda^\reg = Z_{\kappa,\lambda}^\circ$.
	Putting all of this together, the product $[Z_{\kappa,\lambda}^\circ]*[Z_{\lambda,\mu}^\circ]$ is obtained by Gysin pullback of fundamental class along the following diagram:
	\[
		\begin{tikzcd}
			\widetilde{\bfF}^\reg_\lambda\times Z_{\lambda,\mu}^\circ\ar[d] & Z_{\lambda,\mu}^\circ\ar[l,"g"']\ar[d] \\
			  \widetilde{\bfF}^\reg_\lambda\times \widetilde{\bfF}^\reg_\lambda & \widetilde{\bfF}^\reg_\lambda\ar[l,"\Delta"']  
		\end{tikzcd}
	\]
	where $g$ is the graph of inclusion $Z_{\lambda,\mu}^\circ\hookrightarrow\widetilde{\bfF}_\lambda^\reg$.
	However, it is known that Gysin pullback along graph morphism sends $[\widetilde{\bfF}_\lambda^\reg]\otimes x$ to $x$, so we may conclude.
\end{proof}

\subsection{Diagrammatic relations}
Let us compute some relations in $\cS^\cI_n$.
\begin{prop}\label{prop:color-past-split}
	Let $\lambda = (\lambda^1,\ldots,\lambda^k)\in \Comp(n)$.
	The following local diagrammatic relations hold in $\cS^\cI_n$, together with their upside down versions:
	$$
	\tikz[thick,xscale=.2,yscale=.2,font=\footnotesize]{
	\draw (-2,0) -- (-2,1);
	\draw (-2,4) -- (-2,6);
	\draw [color=cyan!90!white] (-2,6) -- (-2,8);
	\ntxt{-2}{6}{$-$}
	\opbox{-3}{1}{-1}{4}{$P$}
	\ntxt{-2}{-1}{$n\tau$}
	\ntxt{-2}{9}{$n\varepsilon$}
	\ntxt{1}{4}{$=$}

	\draw (4,0) -- (4,2);
	\draw [color=cyan!90!white] (4,2) -- (4,4);
	\draw [color=cyan!90!white] (4,7) -- (4,8);
	\ntxt{4}{2}{$-$}
	\ntxt{4}{-1}{$n\tau$}
	\ntxt{4}{9}{$n\varepsilon$}
	\opbox{3}{4}{5}{7}{$P$}
	\ntxt{6}{4}{,}
	}\quad
	\tikz[thick,xscale=.2,yscale=.2,font=\footnotesize]{
	\draw (3,0) -- (3,1.5);
	\draw [color=cyan!90!white] (3,1.5) -- (3,3);
	{\color{cyan!90!white}\dsplit{0}{3}{6}{8}
	\dsplit{1.5}{3}{4.5}{8}}
	\ntxt{3}{1.5}{$-$}
	\ntxt{3}{6.5}{$\cdots$}
	\ntxt{3}{-1}{$n\tau$}
	\ntxt{0}{9}{$\lambda_1\varepsilon$}
	\ntxt{6}{9}{$\lambda_k\varepsilon$}

	\ntxt{7.5}{4}{$=$}

	\draw (13,0) -- (13,1);
	\dsplit{10}{1}{16}{4}
	\dsplit{11.5}{1}{14.5}{4}
	\ntxt{13}{6.5}{$\cdots$}
	\ntxt{13}{-1}{$n\tau$}
	\ntxt{10}{9}{$\lambda_1\varepsilon$}
	\ntxt{16}{9}{$\lambda_k\varepsilon$}
	\opbox{9.7}{4}{16.3}{6}{$Q_\lambda$}
	\draw (10,6) -- (10,7);
	\draw (11.5,6) -- (11.5,7);
	\draw (14.5,6) -- (14.5,7);
	\draw (16,6) -- (16,7);
	\draw [color=cyan!90!white] (10,7) -- (10,8);
	\draw [color=cyan!90!white] (11.5,7) -- (11.5,8);
	\draw [color=cyan!90!white] (14.5,7) -- (14.5,8);
	\draw [color=cyan!90!white] (16,7) -- (16,8);
	\ntxt{10}{7}{$-$}
	\ntxt{11.5}{7}{$-$}
	\ntxt{14.5}{7}{$-$}
	\ntxt{16}{7}{$-$}
	\ntxt{17}{4}{,}
	}\quad
	\tikz[thick,xscale=.2,yscale=.2,font=\footnotesize]{
	\draw [color=cyan!90!white] (3,0) -- (3,1.5);
	\draw (3,1.5) -- (3,3);
	\dsplit{0}{3}{6}{8}
	\dsplit{1.5}{3}{4.5}{8}
	\ntxt{3}{1.5}{$-$}
	\ntxt{3}{6.5}{$\cdots$}
	\ntxt{3}{-1}{$n\varepsilon$}
	\ntxt{0}{9}{$\lambda_1\tau$}
	\ntxt{6}{9}{$\lambda_k\tau$}

	\ntxt{7.5}{4}{$=$}

	{\color{cyan!90!white}\draw (13,0) -- (13,1);
	\dsplit{10}{1}{16}{4}
	\dsplit{11.5}{1}{14.5}{4}}
	\ntxt{13}{6.5}{$\cdots$}
	\ntxt{13}{-1}{$n\varepsilon$}
	\ntxt{10}{9}{$\lambda_1\tau$}
	\ntxt{16}{9}{$\lambda_k\tau$}
	\opbox{9.7}{4}{16.3}{6}{$(-1)^{|N_\lambda|}Q_\lambda$}
	\draw [color=cyan!90!white] (10,6) -- (10,7);
	\draw [color=cyan!90!white] (11.5,6) -- (11.5,7);
	\draw [color=cyan!90!white] (14.5,6) -- (14.5,7);
	\draw [color=cyan!90!white] (16,6) -- (16,7);
	\draw (10,7) -- (10,8);
	\draw (11.5,7) -- (11.5,8);
	\draw (14.5,7) -- (14.5,8);
	\draw (16,7) -- (16,8);
	\ntxt{10}{7}{$-$}
	\ntxt{11.5}{7}{$-$}
	\ntxt{14.5}{7}{$-$}
	\ntxt{16}{7}{$-$}
	\ntxt{17}{4}{.}
	}
	$$
	Here, $P\in\bbk[x_1,\ldots,x_n]^{\fkS_n}$ is a polynomial, and $Q_{\lambda} \coloneqq \prod_{(i,j)\in N_\lambda}(x_i-x_j)$ is the $GL_\lambda$-equivariant Euler class of $\mathfrak{n}_\lambda = \ona{Lie}(\ona{Rad}(GL_\lambda\cdot B))$.
\end{prop}
\begin{proof}
	Upside down versions of the relations are obtained by inverting the order of indexes.
	Locality follows by base change along the natural forgetful maps $\cF_\cla\to \cF_\cmu$, where $\cla\vDash\cmu$.
	The first relation holds because the product in $\cS_n^\cI$ is linear over $H^*_{GL_n}(\pt) = \bbk[x_1,\ldots,x_n]^{\fkS_n}$.
	For the second relation, a computation analogous to~\cite[Lem.~4.6]{MakMin_KLR2023} shows that the l.h.s. is given by $[\bfF_\lambda]\in \cS^\cI_{n\tau,\cla}$, and the r.h.s. by $Q_\lambda\cdot[Y_{\cla}]\in \cS^\cI_{n\tau,\cla}$.
	The difference between the two fundamental classes is the Euler class of the embedding $\bfF_\lambda\subset Y_{\cla} = T^*\bfF_\lambda$, which is precisely $Q_\lambda$.

	The third relation is slightly more subtle.
	While the r.h.s. is similarly given by $Q_\lambda\cdot[\bfF_{\lambda}]$, the same argument computes the l.h.s. to be $\iota^![X_{n\varepsilon}\times Y_\lambda]$, where $\iota^!$ is Gysin pullback with respect to the following fiber square:
	\[
	\begin{tikzcd}
		X_{n\varepsilon}\times Y_\lambda\ar[d] & \bfF_\lambda\ar[d]\ar[l,"\iota"'] \\
		X_n\times X_n & X_n \ar[l]
	\end{tikzcd}
	\]
	This pullback differs from $[\bfF_\lambda]$ by an excess intersection class
	\[
		e(TX_n|_{\bfF_\lambda}/ N_{\bfF_\lambda}(Y_\lambda)) = e(N_{F_\lambda}(X_n)|_{\bfF_\lambda}) = \prod_{(i,j)\in N_\lambda}(x_j-x_i) = (-1)^{|N_\lambda|}Q_\lambda,
	\]
	so we may conclude.
\end{proof}

\begin{lem}\label{lem:two-colour-changes}
	We have $C^{n\tau}_{n\varepsilon}C_{n\tau}^{n\varepsilon} = \prod_{i=1}^n c_i \prod_{1\leq i\neq j\leq n}(x_i-x_j)^2$, and $C_{n\tau}^{n\varepsilon}C^{n\tau}_{n\varepsilon} = 0$.
\end{lem}
\begin{proof}
	Both products are obtained by pushing and pulling along the closed embedding $\pt \to X_n$, thus amount to multiplication by the Euler class of this embedding.
	This class is equal to 
	\[
		e(N_\pt C^n)e(N_{C^n}X_n) = \prod_{i=1}^n c_i \prod_{1\leq i\neq j\leq n}(x_i-x_j+\Delta_{ij}) = \prod_{i=1}^n c_i \prod_{1\leq i\neq j\leq n}(x_i-x_j).
	\]
	It remains to note that $c_i$'s vanish in the cohomology of $\cT_{n\varepsilon}$.
\end{proof}
Diagrammatically, we draw multiplication by $\prod_i c_i \prod_{i\neq j}(x_i-x_j)^2$ as a cross on a strand.

In order to compute relations between colour changes and crossings, we need some preparation.
Given $\lambda,\mu\in \Comp(n)$, consider the natural projection $Z_{\lambda,\mu}\to \bfF_\lambda\times \bfF_\mu$.
The $GL_n$-orbits $\Omega_w\subset \bfF_\lambda\times \bfF_\mu$ are parameterized by double cosets $w\in \fkS_\lambda\backslash \fkS_n/\fkS_\mu$.
This stratification into orbits induces a stratification $Z_{\lambda,\mu} = \bigsqcup_{w\in \fkS_\lambda\backslash \fkS_n/\fkS_\mu}Z_{\lambda,\mu}^w$.
Fixing a total order on $\fkS_\lambda\backslash \fkS_n/\fkS_\mu$ which refines the closure order in $\bfF_\lambda\times \bfF_\mu$, we obtain a filtration $H^{\leq w}_*(\cZ_{\lambda,\mu})$ on $H_*(\cZ_{\lambda,\mu})$ by~\cite[Sec.~3.3]{MakMin_KLR2023}.
Moreover, we have isomorphisms $H_*^{GL_n}(Z_{\lambda,\mu}^w)\simeq H_*(\cT_{\mu'})$, where $\mu'$ is determined by $\fkS_{\mu'} = \fkS_\mu\cap w^{-1}\fkS_\lambda w$.
All of this works just as well for the $\cI$-coloured version:
\begin{prop}\label{prop:filtr-on-Steinberg}
	The stratification $\bfF_\lambda\times \bfF_\mu = \bigsqcup_w \Omega_w$ induces a filtration on $H_*(\cZ_{\cla,\cmu})$.
	The $w$-th graded component is isomorphic to $H_*(\cT_{\cmu'})$, where $\cmu'$ is as in \cref{subs:sp-mer-diag}.
\end{prop}
\begin{proof}
	Analogously to the proof of~\cref{lem:Y-is-smooth}, $Z_{\cla,\cmu}^w$ is an affine fibration over $\Omega_w\times X_{\cmu'}$.
	In particular, we have $H_*(\cZ_{\cla,\cmu}^w) \simeq H_*(\cT_{\cmu'})$.
	Since each piece $H_*(\cT_{\cmu'})$ is either pure or even by~\eqref{eq:assumption-on-curve}, the long exact sequences associated to each successive step of stratification split.
	This gives rise to a filtration with desired graded components.
\end{proof}

Now let $\cla$, $\cmu$ be as in \cref{def:mc-crossing}. We denote 
\begin{gather*}
	\cla_\tau = (\lambda^1c^1,\ldots,a\tau, b\tau,\ldots, \lambda^kc^k),\quad \cla_\varepsilon = (\lambda^1c^1,\ldots,a\varepsilon, b\varepsilon,\ldots, \lambda^kc^k),\\
	\cmu_\tau = (\lambda^1c^1,\ldots,b\tau, a\tau,\ldots, \lambda^kc^k),\quad \cmu_\varepsilon = (\lambda^1c^1,\ldots,b\varepsilon, a\varepsilon,\ldots, \lambda^kc^k).
\end{gather*}

\begin{lem}\label{lem:color-slide}
	Up to lower terms with respect to the length filtration, we have $C_{\cla}^{\cla_\tau}R_\cmu^\cla = R_{\cmu_\tau}^{\cla_\tau} C_{\cmu}^{\cmu_\tau}$, $C^\cmu_{\cmu_\varepsilon}R^{\cmu_\varepsilon}_{\cla_\varepsilon} = R_\cla^\cmu C^\cla_{\cla_\varepsilon}$.
	Diagrammatically, colour change slides past crossings:
	\[
		\tikz[thick,xscale=.2,yscale=.2,font=\footnotesize]{
		\draw [color=cyan!90!white] (0,0) -- (3,4.5);
		\draw (3,4.5) -- (4,6);
		\draw (4,0) -- (0,6);
		\ntxt{3}{4.5}{$-$}
		\ntxt{0}{-1}{$b\varepsilon$}
		\ntxt{4}{-1}{$a\tau$}
		\ntxt{0}{7}{$a\tau$}
		\ntxt{4}{7}{$b\tau$}
		\ntxt{6}{3}{$=$}
		\draw [color=cyan!90!white] (8,0) -- (9,1.5);
		\draw (9,1.5) -- (12,6);
		\draw (12,0) -- (8,6);
		\ntxt{9}{1.5}{$-$}
		\ntxt{8}{-1}{$b\varepsilon$}
		\ntxt{12}{-1}{$a\tau$}
		\ntxt{8}{7}{$a\tau$}
		\ntxt{12}{7}{$b\tau$}
		\ntxt{15}{3}{$+\ldots,$}
		}\qquad
		\tikz[thick,xscale=.2,yscale=.2,font=\footnotesize]{
		\draw [color=cyan!90!white] (0,0) -- (3,4.5);
		\draw (3,4.5) -- (4,6);
		\draw [color=cyan!90!white] (4,0) -- (0,6);
		\ntxt{3}{4.5}{$-$}
		\ntxt{0}{-1}{$a\varepsilon$}
		\ntxt{4}{-1}{$b\varepsilon$}
		\ntxt{0}{7}{$b\varepsilon$}
		\ntxt{4}{7}{$a\tau$}
		\ntxt{6}{3}{$=$}
		\draw [color=cyan!90!white] (8,0) -- (9,1.5);
		\draw (9,1.5) -- (12,6);
		\draw [color=cyan!90!white] (12,0) -- (8,6);
		\ntxt{9}{1.5}{$-$}
		\ntxt{8}{-1}{$a\varepsilon$}
		\ntxt{12}{-1}{$b\varepsilon$}
		\ntxt{8}{7}{$b\varepsilon$}
		\ntxt{12}{7}{$a\tau$}
		\ntxt{15}{3}{$+\ldots$.}
		}
	\]
\end{lem}
\begin{proof}
	We only prove the first equality, the second one being analogous.
	By base change, we can further restrict our attention to the case $\cla = (a\tau, b\varepsilon)$, $\cmu = (b\varepsilon, a\tau)$, $n = a+b$.

	Since we only consider the equality up to lower terms, we can restrict everything to the open stratum $\Omega_w$, $w = w_{a,b}$ as in~\eqref{eq:shuffle-perm}.
	In terms of sheaves, we require our pairs of subsheaves $\cE_1,\cE_2\subset \cE$ to have trivial intersection.
	Consider the following diagram with fiber square:
	\[
		\begin{tikzcd}
			Y_\cmu\times Z_{\cmu_\tau,\cla_\tau}^w\ar[d,hook,"\iota"] & Z_{\cmu,\cla_\tau}^w\ar[l,"\Delta"']\ar[d,hook,"\iota"]\ar[dr,hook,"\iota'"]\arrow[dl, phantom, "\scalebox{1.3}{$\llcorner$}", at start, shift left=1.5ex, color=black] & \\
			(Y_\cmu\times_{X_n}Y_{\cmu_\tau})\times (Y_{\cmu_\tau}\times_{X_n}Y_{\cla_\tau}) & Y_\cmu\times_{X_n}Y_{\cmu_\tau}\times_{X_n}Y_{\cla_\tau}\ar[l,"q"']\ar[r,"p"] & Y_\cmu\times_{X_n}Y_{\cla_\tau} 
		\end{tikzcd}
	\]
	The product $R_{\cmu_\tau}^{\cla_\tau} C_{\cmu}^{\cmu_\tau}$ is equal to $p_*q^!\iota_*([Y_\cmu]\otimes[Z_{\cmu_\tau,\cla_\tau}^w])$ by definition. 
	Using base change for Gysin pullback, this equals to $\iota'_*\Delta^!([Y_\cmu]\otimes[Z_{\cmu_\tau,\cla_\tau}^w])$. 
	By the proof of \cref{prop:filtr-on-Steinberg} both $Z_{\cmu,\cla_\tau}^w$ and $Z_{\cmu_\tau,\cla_\tau}^w$ are smooth, and $\codim_{Y_\cmu\times Z_{\cmu_\tau,\cla_\tau}^w}Z_{\cmu,\cla_\tau}^w = \dim Y_{\cmu_\tau}$.
	Thus $\Delta^! = \Delta^*$ by~\cite[Rem.~6.2.1]{Ful_IT1998}, and so $R_{\cmu_\tau}^{\cla_\tau} C_{\cmu}^{\cmu_\tau} = [Z_{\cmu,\cla_\tau}^w]\in \cS_{\cmu,\cla_\tau}^w$.
	The same argument shows that $C_{\cla}^{\cla_\tau}R_\cmu^\cla$ is given by the fundamental class $[Z_{\cmu,\cla}^w]\in \cS_{\cmu,\cla_\tau}^w$. 
	Note that
	\[
		\cZ_{\cmu,\cla_\tau}\setminus \cZ_{\cmu,\cla} = \{ \cE_1,\cE_2\subset \cE : \cE_1\simeq \cO_c^{\oplus b}, \cE/\cE_2\not\simeq \cO_c^{\oplus a} \}.
	\] 
	However, these two conditions are contradictory when $\cE_1\cap\cE_2 = 0$.
	Thus $(\cZ_{\cmu,\cla_\tau}\setminus \cZ_{\cla,\cmu})^w$ is empty, and we conclude that $[Z_{\cmu,\cla_\tau}^w]=[Z_{\cla,\cmu}^w]$.
\end{proof}

\subsection{Basis theorem}
\label{subs:basis-curve-Schur-2c}
Given $\cla,\cmu\in \cI^{(n)}$, consider the underlying uncoloured compositions $\lambda,\mu\in \Comp(n)$ together with their refinements $\lambda'\vDash\lambda$, $\mu'\vDash\mu$ as in~\cref{subs:sp-mer-diag}.
These refinements can be naturally promoted to $\cI$-coloured partitions $\cla'\vDash\cla$, $\cmu'\vDash\cmu$.
Given $w\in \fkS_{\lambda}\backslash \fkS_{n}/\fkS_{\mu}$, we define another $\cI$-colouring $\cla'_w$ of $\lambda'$, where $(\lambda')^i$ is of colour $\tau$ if and only if both $(\lambda')^i$ is of colour $\tau$ in $\cla'$ and $(\mu')^{w(i)}$ is of colour $\tau$ in $\cmu'$.
For any $P\in H^*(\cF_{\cla'_w},\bbk)$, we consider a diagrammatic element $\psi^{P}_w$, defined as follows:
\begin{itemize}
	\item Take the uncoloured Schur diagram $D_w$;
	\item Colour each split and merge with their colours in $\cla$, $\cmu$. If some thin strand has different colours on two ends, put the appropriate colour change on it between the crossings and the end coloured by $\varepsilon$;
	\item Put a coupon with $P$ between splits and crossings, above the possible colour changes.
\end{itemize}
For an illustration, see the diagram on the right of~\cref{ex:diag-new-vs-old}.

\begin{thm}\label{thm:curve-Schur-basis}
	For every $\cla\in \cI^{(n)}$, fix a $\bbk$-basis $B_{\cla}$ of $H^*(\cT_{\cla},\bbk)$.
	Then the set 
	\[
		\{\psi_w^P : w\in \fkS_{\lambda}\backslash \fkS_{n}/\fkS_{\mu}, P\in B_{\cla'_w}\}
	\]
	is a basis of $\cS^\cI_{\cla,\cmu}$.
\end{thm}
\begin{proof}
	It suffices to prove the claim after passing to the associated graded with respect to the length filtration.
	Let us identify the uncoloured partitions with fully $\tau$-coloured partitions as usual.
	Consider the map $\Psi: \cS^\cI_{\cla,\cmu} \to \cS^\cI_{\lambda,\mu} = \cS_{\lambda,\mu}$ defined by $\Psi(x) = C_\cmu^\mu xC^{\cla}_\lambda$.
	Thanks to~\cref{prop:color-past-split,lem:two-colour-changes,lem:color-slide}, we can remove all colour changes in the diagram of $\Psi(\psi_w^P)$, for the price of multiplying $P$ by some polynomials $Q_\nu\in \bbk[x_1,\ldots,x_n]$ independent of $P$ and adding some lower terms. 
	For instance, for the diagram from~\cref{ex:diag-new-vs-old} we have:
	
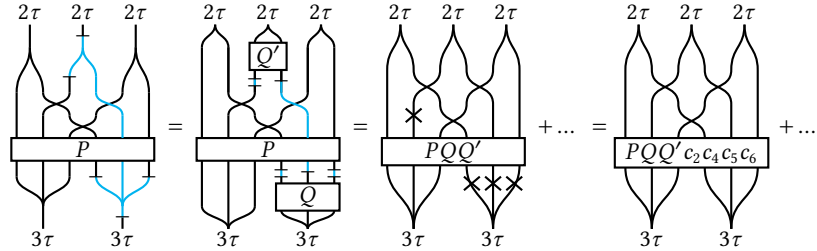
\[
		\tikz[thick,xscale=.35,yscale=.3,font=\footnotesize]{
		\ntxt{1}{-0.5}{$3\tau$}
		\ntxt{4}{-0.5}{$3\tau$}
		\ntxt{0.5}{9.5}{$2\tau$}
		\ntxt{2.5}{9.5}{$2\tau$}
		\ntxt{4.5}{9.5}{$2\tau$}
		\draw (1,0) -- (1,3);
		\draw (4,0) -- (4,0.5);
		\draw [color=cyan!90!white] (4,0.5) -- (4,3);
		\dsplit{0}{1}{2}{2.3}
		{\color{cyan!90!white}\dsplit{3}{1}{5}{2.3}}
        \ntxt{4}{0.5}{$-$}
        \ntxt{3}{2.3}{$-$}
        \ntxt{5}{2.3}{$-$}
		\draw (0,2.3) -- (0,3);
		\draw (2,2.3) -- (2,3);
		\draw (3,2.3) -- (3,3);
		\draw (5,2.3) -- (5,3);
        \opbox{-0.2}{3}{5.2}{4}{$P$}
        \draw (0,4) -- (0,6);
        \draw (1,4) -- (1,5);
        \crosin{2}{4}{3}{5}
        \crosin{1}{5}{2}{6}
        \draw (3,5) .. controls (3,5.6) and (4,5.4) .. (4,6);
		\draw [color=cyan!90!white] (4,5) .. controls (4,5.6) and (3,5.4) .. (3,6);
        \draw [color=cyan!90!white] (4,4) -- (4,5);
        \draw (5,4) -- (5,6);
        \dmerge{0}{6}{1}{8}
        {\color{cyan!90!white}\dmerge{2}{6.7}{3}{8}}
        \dmerge{4}{6}{5}{8}
        \draw (0.5,8) -- (0.5,9);
		\draw (2,6) -- (2,6.7);
		\draw [color=cyan!90!white] (3,6) -- (3,6.7);
        \draw [color=cyan!90!white] (2.5,8) -- (2.5,8.5);
		\draw (2.5,8.5) -- (2.5,9);
        \draw (4.5,8) -- (4.5,9);
        \ntxt{2.5}{8.5}{$-$}
        \ntxt{2}{6.7}{$-$}

        \ntxt{6}{4.5}{$=$}

		\ntxt{8}{-0.5}{$3\tau$}
		\ntxt{11}{-0.5}{$3\tau$}
		\ntxt{7.5}{9.5}{$2\tau$}
		\ntxt{9.5}{9.5}{$2\tau$}
		\ntxt{11.5}{9.5}{$2\tau$}
        \dsplit{7}{0}{9}{1}
        \dsplit{10}{0}{12}{0.8}
        \draw (11,0) -- (11,0.8);
        \opbox{9.8}{0.8}{12.2}{2}{$Q$}
        \draw (7,1) -- (7,3);
        \draw (8,0) -- (8,3);
        \draw (9,1) -- (9,3);
        \draw (10,2) -- (10,2.3);
		\draw [color=cyan!90!white] (10,2.3) -- (10,2.6);
		\draw (10,2.6) -- (10,3);
        \draw (11,2) -- (11,2.5);
		\draw [color=cyan!90!white] (11,2.5) -- (11,3);
		\draw (12,2) -- (12,2.3);
		\draw [color=cyan!90!white] (12,2.3) -- (12,2.6);
		\draw (12,2.6) -- (12,3);
        \ntxt{10}{2.3}{$-$}
        \ntxt{10}{2.6}{$-$}    
        \draw (8,6) -- (8,8);
        \draw (9,6) -- (9,7);
        \draw (10,6) -- (10,7);
        \draw (11,6) -- (11,8);
        \ntxt{11}{2.5}{$-$}
        \ntxt{12}{2.3}{$-$}
        \ntxt{12}{2.6}{$-$}
        \opbox{6.8}{3}{12.2}{4}{$P$}
        \draw (7,4) -- (7,8);
        \draw (8,4) -- (8,5);
        \crosin{9}{4}{10}{5}
        \crosin{8}{5}{9}{6}
		\draw (10,5) .. controls (10,5.6) and (11,5.4) .. (11,6);
		\draw [color=cyan!90!white] (11,5) .. controls (11,5.6) and (10,5.4) .. (10,6);
        \draw [color=cyan!90!white] (11,4) -- (11,5);
        \draw (12,4) -- (12,8);
        \draw (8,6) -- (8,8);
		\draw (9,6) -- (9,6.3);
		\draw [color=cyan!90!white] (9,6.3) -- (9,6.6);
		\draw (9,6.6) -- (9,7);
        \draw [color=cyan!90!white] (10,6) -- (10,6.5);
		\draw (10,6.5) -- (10,7);
        \draw (11,6) -- (11,8);
        \ntxt{9}{6.3}{$-$}
        \ntxt{9}{6.6}{$-$}
        \ntxt{10}{6.5}{$-$}
        \opbox{8.8}{7}{10.2}{8.2}{$Q'$}
        \dmerge{7}{8}{8}{9}
        \dmerge{9}{8.2}{10}{9}
        \dmerge{11}{8}{12}{9}

        \ntxt{13}{4.5}{$=$}

		\ntxt{15}{-0.5}{$3\tau$}
		\ntxt{18}{-0.5}{$3\tau$}
		\ntxt{14.5}{9.5}{$2\tau$}
		\ntxt{16.5}{9.5}{$2\tau$}
		\ntxt{18.5}{9.5}{$2\tau$}
        \draw (15,0) -- (15,2.8);
        \draw (18,0) -- (18,2.8);
        \dsplit{14}{0}{16}{2.8}
        \dsplit{17}{0}{19}{2.8}
        \strex{17.2}{2}
        \strex{18}{2}
        \strex{18.8}{2}
        \opbox{13.8}{2.8}{19.2}{4}{$PQQ'$}
        \draw (14,4) -- (14,7);
        \draw (15,4) -- (15,5.5);
        \crosin{16}{4}{17}{5.5}
        \crosin{15}{5.5}{16}{7}
        \crosin{17}{5.5}{18}{7}
        \draw (18,4) -- (18,5.5);
        \draw (19,4) -- (19,7);
        \dmerge{14}{7}{15}{9}
        \dmerge{16}{7}{17}{9}
        \dmerge{18}{7}{19}{9}
        \strex{15}{5}

        \ntxt{21}{4.5}{$+\ldots\text{ }=$}

		\ntxt{24}{-0.5}{$3\tau$}
		\ntxt{27}{-0.5}{$3\tau$}
		\ntxt{23.5}{9.5}{$2\tau$}
		\ntxt{25.5}{9.5}{$2\tau$}
		\ntxt{27.5}{9.5}{$2\tau$}
        \draw (24,0) -- (24,2.7);
        \draw (27,0) -- (27,2.7);
        \dsplit{23}{0}{25}{2.7}
        \dsplit{26}{0}{28}{2.7}
        \opbox{22.6}{2.7}{28.4}{4}{$PQQ'c_2c_4c_5c_6$}
        \draw (23,4) -- (23,7);
        \draw (24,4) -- (24,5.5);
        \crosin{25}{4}{26}{5.5}
        \crosin{24}{5.5}{25}{7}
        \crosin{26}{5.5}{27}{7}
        \draw (27,4) -- (27,5.5);
        \draw (28,4) -- (28,7);
        \dmerge{23}{7}{24}{9}
        \dmerge{25}{7}{26}{9}
        \dmerge{27}{7}{28}{9}
        \ntxt{29.5}{4.5}{$+\ldots$}
		}
	\]
	We see that $\Gr\Psi(\psi_w^P)$ is of length $w$.
	The polynomials $Q_\nu$ are non-zero divisors and monic, and $P$ does not contain any nilpotents arising from colour changes by our choice of $\cla'_w$, so the elements $\Gr\Psi(\psi_w^P)$, $P\in B_{\cla'_w}$ are linearly independent by~\cref{prop:old-basis-curve}.
	Since the graded dimensions of $H^*(\cT_{\cla'_w},\bbk)$ and $\Gr^w \cS^\cI_{\cla,\cmu}$ coincide by~\cref{prop:filtr-on-Steinberg}, we conclude that $\{\psi_w^P\}$ is a basis of $\cS^\cI_{\cla,\cmu}$ over any field $\bbk$.
\end{proof}

\begin{cor}
	The polynomial representation $\bigoplus_{\cla \in \cI^{(n)}} H^*(\cF_\cla,\bbk)$ of $\cS^\cI_n$ is faithful.
\end{cor}
\begin{proof}
	Consider a non-trivial linear combination of the basis elements of $\cS^\cI_{\cla,\cmu}$ that acts trivially the polynomial representation.
	Then its image under $\Psi$ acts trivially as well.
	However, we showed that $\Psi$ is injective, so that it suffices to check that the restriction of polynomial representation to $\cS_n$ is faithful. 
	This holds by~\cite[Th.~4.15]{MakMin_KLR2023}.
\end{proof}

By \cref{cor:surj-taut-classes}, we have a surjection
\[
	\Pol_{\Ibe_\cla} = \bbk[u_1,\ldots,u_n,v_1,\ldots,v_n]^{\fkS_{\Ibe_\cla}}\twoheadrightarrow H^*(\cT_\cla);
\]
see \cref{cor:image-phi} for a precise description of this map.
In particular, let the collection of elements $\widetilde{B}_\cla\subset \Pol_{\Ibe_\cla}$ be an arbitrary lift of $B_\cla\subset H^*(\cT_\cla)$.
Then \cref{cor:span-set-eCAe} trivially follows from \cref{thm:curve-Schur-basis}.

\subsection{Polynomial representation}
Recall the polynomial representation $\bfP_n$ of $\cS_n^\cI$ defined by~\eqref{eq:poly-rep-def}.
Let $\cla = (\lambda^1c^1,\ldots,\lambda^kc^k)$, and let $H_\cla\subset H^*(C,\bbk)^{\otimes n}$ be the tensor product $\bigotimes_{i=1}^k H_i$, where $H_i = H^*(C,\bbk)^{\otimes \lambda^i}$ if $c^i = \tau$ and $\bbk$ otherwise. 
By~\eqref{eq:tor-vs-poly}, we have an embedding
\[
	H^*(\cF_\cla,\bbk) \subset \bfP_\cla\coloneqq (H_\cla[x_1,\ldots,x_n])^{\fkS_\cla}.
\] 
We will use it to describe the polynomial representation explicitly.
Let $\Delta\in H^2(C\times C,\bbk)$ be the class of the diagonal.
For any $1\leq i<j\leq n$ we have an embedding $H^*(C,\bbk)^{\otimes 2} \subset H^*(C,\bbk)^{\otimes n}$ sending $x\times y$ to $1^{\otimes i-1}\otimes x \otimes 1^{\otimes j-i-1}\otimes y\otimes 1^{\otimes n-j}$; we denote by $\Delta_{ij}\in H^2(C,\bbk)^{\otimes n}$ the image of $\Delta$.

\begin{prop}\label{prop:new-poly-rep}
	We have the following (local) formulas for the action of $\cS^\cI_n$ on the polynomial representation:
	\begin{itemize}
		\item Coupons act by multiplication;
		\item $S^{(a\tau,b\tau)}_{n\tau}(P) = P$, and $S^{(a\varepsilon,b\varepsilon)}_{n\varepsilon}(P) = P\prod_{i=1\leq i\leq a < j \leq n} (x_i-x_j)$;
		\item $M_{(a\varepsilon,b\varepsilon)}^{n\varepsilon}(P) = \partial_{a,b}(P)$, and $M_{(a\tau,b\tau)}^{n\tau}(P) = \partial_{a,b}\left( P\prod_{i=1\leq i\leq a < j \leq n} (x_i-x_j-\Delta_{ij}) \right)$;
		\item $C_{n\varepsilon}^{n\tau}: \bbk[x_1,\ldots,x_n]^{\fkS_n}\to \bfP_{n\tau}$ acts by multiplication by $\prod_{i} c_i\prod_{i\neq j}(x_i-x_j)$;
		\item $C^{n\varepsilon}_{n\tau}: \bfP_{n\tau}\to \bbk[x_1,\ldots,x_n]^{\fkS_n}$ acts by setting all classes in $H^{>0}(C^n)$ to zero.
	\end{itemize}
\end{prop}
\begin{proof}
	The formulas in colour $\tau$ were proven in~\cite[Th.~4.15]{MakMin_KLR2023}.
	$C_{n\varepsilon}^{n\tau}$ is defined as the pushforward along the closed embedding $\pt = X_{n\varepsilon} \to X_n$, and so is given by the Euler class of this embedding:
	\[
		e(N_\pt X_n) = e(N_\pt C^n)e(N_{C^n} X_n) = \prod_{i} c_i\prod_{i\neq j}(x_i-x_j+\Delta_{ij}) = \prod_{i} c_i\prod_{i\neq j}(x_i-x_j).
	\]
	For $C^{n\varepsilon}_{n\tau}$, this does nothing on polynomials in $x_i$'s by linearity, and kills anything containing classes in $H^{>0}(C^n)$, since $H^*_{GL_n}(\pt)$ has no nilpotent elements.
	The rest of the formulas follow from the relations in~\cref{prop:color-past-split}.
\end{proof}

Let $D_n = x_1^{n-1}x_2^{n-2}\ldots x_{n-1}$. 
Using the identity $\partial_{w_0}(D_n) = 1$, it is easy to check in the polynomial representations that 
\begin{equation}\label{eq:thick-thin-colour}
	C_{n\tau}^{n\varepsilon} = M^{n\varepsilon}_{\bm{\varepsilon}}D_nC_{\bm{\tau}}^{\bm{\varepsilon}}S_{n\tau}^{\bm{\tau}},\quad C_{n\varepsilon}^{n\tau} = \pm M^{n\tau}_{\bm{\tau}}D_nC_{\bm{\varepsilon}}^{\bm{\tau}}S_{n\varepsilon}^{\bm{\varepsilon}};\qquad \bm{\tau} = (\tau,\ldots,\tau), \quad \bm{\varepsilon} = (\varepsilon,\ldots, \varepsilon).
\end{equation}
In words, thick colour changes are expressed in terms of thin colour changes.

\begin{cor}
	The algebra $e\CA(n\delta)e$ is generated by (single-coloured) splits and merges, multicoloured crossings, thin colour changes, symmetric polynomials $\bbk[u_1,\ldots,u_k]^{\fkS_k}$ on strands of thickness $k$, and polynomial $(v-u)$ on thin strands of colour $\tau$.
\end{cor}
\begin{proof}
	Follows from \cref{cor:span-set-eCAe,cor:image-phi} and~\eqref{eq:thick-thin-colour}.
\end{proof}

The computation of the action of multicoloured crossings is somewhat more involved.
\begin{prop}\label{prop:thin-mc-cros-formula}
	Let $P\in H^*(\cT_{(\varepsilon,\tau)}) = H^*(C)[x_1]\otimes \bbk[x_2]$.
	Write $P = P_1+P_2$, where $P_1\in \bbk[x_1,x_2]$, and $P_2\in H^{>0}(C)[x_1,x_2]$.
	Then $R_{(\varepsilon,\tau)}^{(\tau,\varepsilon)}(P) = s(P) + c_1\partial P_1$.
\end{prop}
\begin{proof}
	Consider the Steinberg variety $Z_2 = Y_2 \times_{X_2} Y_2$.
	$Z_2$ has two irreducible components, the diagonal $Y_2$ and the closure of its complement; let us call the latter $T_2$.
	It is well known that $T_2$ is smooth\footnote{One can check this étale-locally, where it becomes a claim about the Grothendieck-Springer alteration for $\fkgl_2$. In fact, one can show that $T_2$ is the blow up of $Y_2$ along the subvariety of points with underlying torsion sheaf of the form $\cO_x^{\oplus 2}$, $x\in C$}.
	Now look at $Z'\coloneqq Z_{(\varepsilon,\tau),(\tau,\varepsilon)} = Y_{(\varepsilon,\tau)}\times_{X_2} Y_{(\tau,\varepsilon)}\subset Z_2$ from~\cref{def:mc-crossing}.
	We have $Z'\cap Y_2 = Y_{(\varepsilon,\varepsilon)}$, so $\dim Z'\cap Y_2 = 2$.
	On the other hand, $T'_2\coloneqq Z'\cap T_2$ has a $3$-dimensional open, covering the locus $X_2^{\tau+\varepsilon}\subset X_2$ of quotients of the form $\cO_p \oplus \cO_x$, $x\neq p$.
	This shows that $[Z'] = [T'_2]$, and $T'_2$ is smooth\footnote{Again, étale-locally this boils down to Grothendieck-Springer for $\fksl_2$}.

	Consider the following commutative diagram:
	\[
		\begin{tikzcd}
			 Y_{(\varepsilon,\tau)}\times T_2\ar[d,hook,"\iota"] & Y_{(\varepsilon,\tau)}\times T'_2\ar[l,hook',"j"'] & T'_2\ar[l,"q'"']\ar[r,"p'"]\ar[d,hook,"\iota"] & Y_{(\tau,\varepsilon)}\ar[d,hook,"\iota"]\\
			Y_2\times T_2 & &T_2\ar[ll,"q"']\ar[r,"p"] & Y_2
		\end{tikzcd}
	\]
	We want to compute $p'_*(q')^!(P\otimes[T'_2])$, for $P\in H^*(Y_{(\varepsilon,\tau)})$.
	Since $j$ is a closed embedding of smooth varieties, we have $j^!(P\otimes[T_2]) = (P\otimes[T'_2])$.
	Furthermore, since $T_2$ and $Y_{(\varepsilon,\tau)}\times T_2$ intersect transversally, we have base change by~\cite[Prop.~2.2]{MakMin_KLR2023}:
	\[
		\iota_*p'_*(q')^!(P\otimes[T'_2]) = p_*\iota_*(q')^!j^!(P\otimes[T_2]) = p_*q^!\iota_*(P\otimes[Z_2]-[Y_2]) = R_2(c_1P) - c_1P.
	\]
	Using~\cref{prop:new-poly-rep} and the fact that $\Delta(c_1\otimes\gamma) = 0$ for all $\gamma\in H^{>0}(C)$, we have:
	\[
		c_2R_{(\varepsilon,\tau)}^{(\tau,\varepsilon)}(P) = R_2(c_1P) - c_1P = c_2s(P) + \Delta\partial(c_1P_1+c_1P_2) = c_2(s(P)+c_1\partial(P_1)).
	\]
	We conclude by removing $c_2$ from both sides.
\end{proof}

We do not give the formulas for the action of thick multicoloured crossings. 
However, we expect that a dimension count argument as in the proof of~\cref{prop:thin-mc-cros-formula} shows that splits slide past multicoloured crossings.
One can then derive explicit formulas by expressing thick crossings in terms of thin crossings:
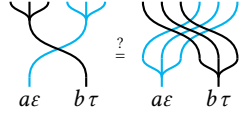
\[
	\tikz[thick,xscale=.25,yscale=.25,font=\footnotesize]{
		\draw [color=cyan!90!white] (1,0) .. controls (1,1.8) and (4,1.2) .. (4,3);
		\draw (4,0) .. controls (4,1.8) and (1,1.2) .. (1,3);
		\dsplit{0}{3}{2}{4}
		{\color{cyan!90!white}\dsplit{3}{3}{5}{4}}
		\draw (1,3) -- (1,4);
		\draw [color=cyan!90!white] (4,3) -- (4,4);
		\draw [color=cyan!90!white] (1,-0.5) -- (1,0);
		\draw (4,-0.5) -- (4,0);

		\ntxt{1}{-1.2}{$a\varepsilon$}
		\ntxt{4}{-1.2}{$b\tau$}

		\ntxt{5.8}{1.5}{$\substack{?\\=}$}

		{\color{cyan!90!white}\dsplit{7}{0}{9}{1}}
		\dsplit{10}{0}{12}{1}
		\draw [color=cyan!90!white] (8,-0.5) -- (8,1);
		\draw (11,-0.5) -- (11,1);
		\draw [color=cyan!90!white] (7,1) .. controls (7,2.8) and (10,2.2) .. (10,4);
		\draw [color=cyan!90!white] (8,1) .. controls (8,2.8) and (11,2.2) .. (11,4);
		\draw [color=cyan!90!white] (9,1) .. controls (9,2.8) and (12,2.2) .. (12,4);
		\draw (10,1) .. controls (10,2.8) and (7,2.2) .. (7,4);
		\draw (11,1) .. controls (11,2.8) and (8,2.2) .. (8,4);
		\draw (12,1) .. controls (12,2.8) and (9,2.2) .. (9,4);

		\ntxt{8}{-1.2}{$a\varepsilon$}
		\ntxt{11}{-1.2}{$b\tau$}
	}
\]

\subsection{Relation to Savage's algebras}\label{subs:Savage}
Let us recall the definition of affine wreath product algebras after~\cite{Sav_AWPA2018}. Our setting is slightly different from \emph{op.cit.}, in that we assume that the Nakayama automorphism is trivial, but do not require the Frobenius element to be non-degenerate.
\begin{defn}
	Let $F$ be a unital algebra, and $\Delta\in F\otimes F$ an element satisfying $(f_1\otimes f_2)\Delta = \Delta(f_2\otimes f_1)$ for all $f_1,f_2\in F$; we call such $\Delta$ \textit{weak Frobenius}.
	For $n>0$, the \textit{affine wreath product algebra} $\cA_n(F) = \cA_n(F,\Delta)$ is the quotient of the free product $F^{\otimes n}[x_1,\ldots,x_n] \star \bbk\fkS_n$ modulo the following relations:
	\begin{gather*}
		s_if = f^{s_i}s_i\text{ for all }f\in F^{\otimes n},\\
		s_ix_j = x_{s_i(j)}s_i + (\delta_{i,j}-\delta_{i+1,j})\Delta_{i,i+1}. 
	\end{gather*} 
\end{defn}

Let $C$ be a smooth projective curve.
Denote by $M_\pt$ the one-dimensional ideal $H^2(C)\subset H^*(C)$, and write $V_C = H^*(C)\oplus M_\pt$.
Define $Z^e_C \coloneqq \End_{H^*(C)}(V_C)$, where we consider everything with cohomological grading, e.g. $M_\pt$ lives in degree $2$.
It is clear that 
\[
	\sum_i t^i \dim Z^e_C[i] = 2 + 2g(C)t + 3t^2,
\]
where $g(C)$ is the genus of $C$.
Observe that $\cS^\cI_1\simeq Z^e_{C}[x]$, $\bfP_1\simeq V_C[x]$.
Let $y = c\cup- :H^*(C)\to M_\pt$, and $z: M_\pt = H^2(C)\hookrightarrow H^*(C)$ the inclusion.
Then $\Delta'\coloneqq\Delta_C + y\otimes z + z\otimes y$ is a weak Frobenius element.

\begin{expl}\label{ex:zigzag-A1-ext}
	Let $C = \bbP^1$. In this case $Z^e_{\bbP^1}$ is isomorphic to the extended zigzag algebra of type $A_1$, given by $Z^e_{A_1} \coloneqq Z_{A_2}/(yz)$, see~\eqref{eq:zigzag-A2-def} for notation, and $\Delta' = 1_\tau\otimes zy + zy\otimes 1_\tau + y\otimes z + z\otimes y$. 
\end{expl}

Consider the subalgebra $\cR^\cI_n = \sum_{\cla,\cmu\in \cI^{n}} \cS^\cI_{\cla,\cmu}\subset \cS^\cI_n$.
Diagrammatically, this means we only allow thin strands at the ends.

\begin{prop}
	We have an isomorphism $\cR^\cI_n\simeq \cA_n(Z^e_C)$. 
\end{prop}
\begin{proof}
	Rewrite the polynomial representation of $\cR^\cI_n$ in the following way:
	\[
		\sum_{\cla\in\cI^{n}} H^*(\cT_\cla) \simeq \bigotimes_{i=1}^n \left(H^*(\cT_\tau)\oplus H^*(\cT_\varepsilon)\right) \simeq V_C^{\otimes n}[x_1,\ldots,x_n].
	\]
	Similarly, we can identify the subalgebra of $\cR^\cI_n$ generated by colour changes and all $H^*(\cT_\cla)$, $\cla\in \cI^{n}$ with $(Z^e_C)^{\otimes n}[x_1,\ldots,x_n]$.
	Consider the elements
	\[
		\underline{s}_i \coloneqq R_{(\tau,\tau)}^{(\tau,\tau)} + R_{(\varepsilon,\varepsilon)}^{(\varepsilon,\varepsilon)} + R_{(\varepsilon,\tau)}^{(\tau,\varepsilon)} + R^{(\varepsilon,\tau)}_{(\tau,\varepsilon)} - 1.
	\]	
	Let $n=2$, $i=1$ for simplicity.
	By~\cref{prop:new-poly-rep,prop:thin-mc-cros-formula}, the action of $\underline{s}$ on the polynomial representation $(V_C)^{\otimes n}[x_1,\ldots,x_n]$ is as follows:
	\begin{align*}
		P\in H^*(\cT_{(\tau,\tau)}): & \underline{s}(P) = s(P) + \Delta_C\partial(P), \\
		P\in H^*(\cT_{(\varepsilon,\tau)}): & \underline{s}(P) = s(P) + c_1\partial(P)|_{1\otimes H^{>0}(C) = 0} = s(P) + (z\otimes y)\partial(P),\\
		P\in H^*(\cT_{(\tau,\varepsilon)}): & \underline{s}(P) = s(P) + (y\otimes z)\partial(P),\\
		P\in H^*(\cT_{(\varepsilon,\varepsilon)}): & \underline{s}(P) = s(P).
	\end{align*}
	Collecting these formulas together, we see that $\underline{s}(P) = s(P) + (\partial^{\Delta'} P)$, where $\partial^{\Delta'}$ is the $\Delta'$-twisted Demazure operator:
	\[
		\partial^{\Delta'}(P) = \frac{\Delta' P - P^s \Delta'}{x_1-x_2}.
	\]
	It is easy to check it satisfies Leibniz rule: $\partial^{\Delta'}(PQ) = P^s\partial^{\Delta'}(Q) + \partial^{\Delta'}(P)Q$.
	In particular:
	\begin{align*}
		\underline{s}(fP) & = f^s s(P) + f^s \partial^{\Delta'}(P) = f^s\underline{s}(P),\\
		\underline{s_i}(x_jP) & = x_{s_i(j)}s_i(P) + \partial^{\Delta'}_i (x_jP) = x_{s_i(j)}s_i(P) + x_{s_i(j)}\partial^{\Delta'}_i(P) + (\delta_{i,j}-\delta_{i+1,j})\Delta'_i P. \\
	\end{align*}
	We leave checking $\fkS_n$-relations to the interested reader (or see~\cite[Sec.~3]{lai2025schurification}); in any case, we obtain an algebra homomorphism $\cA_n(Z^e_C)\to \cR^\cI_n$.
	As the basis of $\cA_n(Z^e_C)$ in~\cite[Th.~4.6]{Sav_AWPA2018} gets sent to a basis of $\cR^\cI_n$ from~\cref{thm:curve-Schur-basis}, this is a bijection.
\end{proof}

\begin{rmk}
	While Savage requires $\Delta$ to be non-degenerate, the proof of \cite[Th.~4.6]{Sav_AWPA2018} does not use this.
	See~\cite[Sec.~3]{lai2025schurification} for another argument. 
\end{rmk}

\begin{cor}
	\label{cor:End(P0)}
	Over a field of characteristic $0$, the algebra $\CA(n\delta)$ is Morita-equivalent to $\cR^\cI_n\simeq \cA_n(Z_{A_1}^e)$.
\end{cor}
\begin{proof}
	Follows from~\cref{prop:proj-gen}.
\end{proof}

\appendix
\medskip
\section{Integral cohomology of $\cT_n$}\label{subs:erratum}

One of the technical results we needed in~\cite{MakMin_KLR2023} was the computation of integral cohomology groups $H^*(\cT_n,\bbZ)$, $\cT_n = \cT_n(\bbP^1)$.
We approached it from three angles.
First, writing $\cT_n = [X_n/GL_n]$ as in \cref{subs:curve-Schur}, restriction to torus fixed points $i:C^n\hookrightarrow X_n$ induces a map
\begin{equation*}
	i^*: H^*(\cT_n,\bbZ) \to \Sym^n\left( H^*(\bbP^1,\bbZ)[x] \right) = \left(\bbZ[x_1,\ldots,x_n,c_1,\ldots,c_n]/(c_i^2)\right)^{\fkS_n}.
\end{equation*}
Second, let $\pi_*:\cT_n\times \bbP^1\to \cT_n$ be the projection, and consider the universal sheaf $\cE$ on $\cT_n\times \bbP^1$. Consider the subring generated by Künneth-Chern classes $c_{i,0} = \pi_*(c_i(\cE)\cup [\pt])$, $c_{i,1} = \pi_*(c_i(\cE))$:
\begin{equation*}
	\jmath: TH^*(\cT_n,\bbZ) = \langle c_{i,0}, c_{i,1} : i\geq 0 \rangle \hookrightarrow H^*(\cT_n,\bbZ).
\end{equation*}
Finally, let $\Gamma = {1 \rightrightarrows 0}$. Then pullback along the open embedding $\cT_n\simeq \Rep^\reg_{n\delta}\Gamma\subset \Rep_{n\delta}\Gamma$ induces a map
\begin{equation*}
	\phi: H^*(\Rep_{n\delta},\bbZ) = \bbZ[u_1,\ldots,u_n,v_1,\ldots,v_n]^{\fkS_n\times \fkS_n}\to H^*(\cT_n,\bbZ).
\end{equation*}
Let $\Lambda\subset H^*(\cT_n,\bbZ)$ be the subring generated by classes $c_i(\pi_*\cE)$, $i\geq 0$.
It easily follows from definitions that under the identification $\cT_n = [X_n/GL_n]$ it corresponds precisely to the pullback of $H^*_{GL_n}(\pt)$.
We claimed in~\cite{MakMin_KLR2023} that $i^*$ is an isomorphism (Prop.~6.2), $\jmath$ is not onto (Ex.~6.7), and the image of $\phi$ is $\Lambda\cdot TH^*(\cT_n)$ (Prop.~8.37).
Unfortunately, the first claim is incorrect. 
However, the proofs of the other two claims were algebraic in nature, and as such relied solely on $i^*$ being injective.

\begin{lem}\label{lem:nilcone-no-torsion}
	The equivariant Borel-Moore homology $H_*^{\BM}([\cN_{\fkgl_n}/GL_n],\bbZ)$ has no torsion as a $H^*_{GL_n}(\pt)$-module.
\end{lem}
\begin{proof}
	We write $\bfH_R = H_*^{\BM}([\cN_{\fkgl_n}/GL_n],R)$ for a ring $R$.
	Decompose $[\cN_{\fkgl_n}/GL_n]$ into the union of nilpotent orbits $\bigsqcup_{\lambda\vdash n}[\bbO_\lambda/GL_n]$.
	It is known	that 
	\[
		H_*^{\BM}([\bbO_\lambda/GL_n],\bbZ) = H^*_{GL_n}(\bbO_\lambda,\bbZ) \simeq H^*_{G_\lambda}(\pt,\bbZ),
	\]
	where $G_\lambda$ is a product of general linear groups.
	In particular, $H_*^{\BM}([\bbO_\lambda/GL_n])$ is even and free over $\bbZ$ for all $\lambda$.
	Long exact sequence in Borel-Moore homology implies that the same holds for $\bfH_\bbZ$.

	Assume that we have a torsion $H^*_{GL_n}(\pt,\bbZ)$-submodule $M\subset \bfH_\bbZ$.
	By universal coefficients, 
	$M\otimes_\bbZ \bbQ$ is a torsion $H^*_{GL_n}(\pt,\bbQ)$-submodule of $\bfH_\bbQ$.
	But $\cN_{\fkgl_n}$ is rationally smooth~\cite{borho1983partial}, so that $\bfH_\bbQ\simeq H^*_{GL_n}(\cN_{\fkgl_n},\bbQ) \simeq H^*_{GL_n}(\pt,\bbQ)$ is a free module.
	Hence necessarily $M = 0$.
\end{proof}

\begin{prop}\label{prop:loc-is-inj}
	$i^*$ is injective.
\end{prop}
\begin{proof}
	Let $T_n\subset GL_n$ be a maximal torus.
	The map $i^*$ is induced by restriction to $T_n$-fixed points. 
	In order to apply equivariant localization theorem, we need to show that no element of $H^*(\cT_n,\bbZ)\subset H^*_{T_n}(X_n,\bbZ)$ is annihilated by the Euler class $e = c_{\mathrm{top}}(N_{C^n}(X_n))$.
	We know by~\cite[Eq.~(8)]{MakMin_KLR2023} that the leading term of $e$ is the symmetric polynomial $\prod_{i<j}(x_i-x_j)^2$; therefore it suffices to check that $H^*(\cT_n,\bbZ)$ has no $\Lambda = H^*_{GL_n}(\pt)$-torsion.
 
	Write $\cT_n^{k\infty} = \Rep^\reg_{k\delta^\bullet}\times \Rep^\reg_{(n-k)\delta^\circ}$; we have $\cT_n = \bigsqcup_{k=0}^n \cT_n^{k\infty}$ as in \cref{subs:reps-Kronecker-quiver}.
	As $\cT_n$ is smooth and both $\Rep^\reg_{k\delta^\bullet}$ and $\Rep^\reg_{k\delta^\circ}$ have even Borel-Moore homology, by long exact sequence it is enough to show that both $H_*^{\BM}([\fkgl_k/GL_k])$ and $H_*^{\BM}([\cN_{\fkgl_k}/GL_k])$ have no $H^*_{GL_k}(\pt)$-torsion.
	This is obvious for the former, and follows from \cref{lem:nilcone-no-torsion} for the latter.
\end{proof}

\begin{expl}
	Let $n=2$, and consider the class of $\cT_2^{2\infty}$ in $H^*(\cT_2,\bbZ)$.
	We can deform $\cT_2^{2\infty}$ to a smooth substack $[N/GL_2]\subset [\fkgl_2/GL_2] = \cT_2^{0\infty}$, where
	\[
		N = \{ \begin{psmallmatrix}a & b\\ c & -a\end{psmallmatrix} : a^2+bc = 1 \}.
	\]
	$N$ intersects $(\bbP^1)^2$ inside $X_2$ transversally in two points.
	Therefore $i^*[\cT_2^{2\infty}] = 2c_1c_2$, and by the proof of \cref{prop:loc-is-inj} $c_1c_2$ cannot lie in the image of $i^*$.
\end{expl}

Recall the merge operator $M_{1^n}^n$ from~\cite[Thm.~4.15]{MakMin_KLR2023} and $\cF_{1^n}$ from \cref{subs:curve-Schur}.
We write $\underline{u}$ instead of $u_1,\ldots,u_n$ for brevity.
\begin{cor}\label{cor:image-phi}
	We have $H^*(\cT_n,\bbZ) = \Lambda\cdot TH^*(\cT_n,\bbZ)$, that is the integral cohomology of $\cT_n$ is generated by tautological classes.
	The following diagram commutes:
	\[
	\begin{tikzcd}
		\bbZ[\underline{u},\underline{v}]^{\fkS_n\times \fkS_n}\ar[d,hook]\ar[r,two heads,"\phi"] & H^*(\cT_n,\bbZ)\ar[d,hook,"i^*"] & H^*(\cF_{1^n},\bbZ)\ar[d,equal]\ar[l] \\
		\bbZ[\underline{u},\underline{v}]\ar[r,"\widetilde\phi"] & \bbZ[\underline{x},\underline{c}]/(c_i^2) & \bbZ[\underline{x},\underline{c}]/(c_i^2)\ar[l,"M_{1^n}^n"']
	\end{tikzcd}
	\]
	where $\widetilde\phi$ is given by $u_i\mapsto x_i$, $v_i\mapsto x_i+c_i$.
	Moreover, $\Lambda\cdot\Im(M_{1^n}^n) = \Im(i^*)$.
\end{cor}
\begin{proof}
	The left square commutes by \cref{cor:surj-taut-classes} and~\cite[Prop.~8.37]{MakMin_KLR2023}.
	The rest follows from \cite[Thm.~4.15, Lm.~8.31]{MakMin_KLR2023}.
\end{proof}
Since $H^*(\cT_n,\bbZ)$ is manifestly free over $\bbZ$, we deduce a description of $H^*(\cT_n,R)$ for any ring $R$ by universal coefficients.
For a general smooth curve $C$, we do not have an analogue of the map $\phi$, so our argument does not apply.
\begin{conj}
	Let $C$ be a smooth complex curve, and consider the subring $TH^*(\cT_n(C),\bbZ)\subset H^*(\cT_n(C),\bbZ)$ generated by classes $\pi_*(c_i(\cE)\cup \gamma)$, $i\geq 0$, $\gamma\in H^*(C,\bbZ)$.
	Then 
	\begin{enumerate}
		\item $H^*(\cT_n(C),\bbZ)$ is generated by $TH^*(\cT_n(C),\bbZ)$ over $\Lambda$;
		\item The image of the pushforward map 
		\[
			H^*(\cF_n(C),\bbZ)\to H^*(\cT_n(C),\bbZ)
		\]
		from full flags is precisely $TH^*(\cT_n(C),\bbZ)$.
	\end{enumerate}
\end{conj}

In particular, the conjecture would imply an explicit description of $H^*(\cT_n(C),\bbZ)$, as the $\Lambda$-submodule of $H^*(C^n,\bbZ)[x_1,\ldots,x_n]$ generated by the image of $M_{1^n}^n$.

\subsection{Erratum for~\cite{MakMin_KLR2023}}
Because of \cref{prop:loc-is-inj}, the mistake in~\cite[Prop.~6.2]{MakMin_KLR2023} does not propagate into the rest of our computations, which are formal manipulations with polynomials.
However, we need to slightly tweak our conclusions.
\begin{itemize}
	\item Every time we mentioned the polynomial representation of $\cS^{\bbP^1}_n$, we need to replace $\bbZ[\underline{x},\underline{c}]/(c_i^2)$ with $H^*(\cT_n,\bbZ)$;
	\item We distinguished between two diagrammatic algebras $\cS^{(1)}$, resp. $\cS^{(2)}$, where thick strands could support coupons valued in $H^*(\cT_n,\bbZ)$, resp. $\bbZ[\underline{x},\underline{c}]/(c_i^2)$. The algebra $\cS^{(2)}$ described the semicuspidal algebra, and $\cS^{(1)}$ was isomorphic to the curve Schur algebra.
	\cref{cor:image-phi} tells us that both roles are played by $\cS^{(2)}$, and $\cS^{(1)}$ is a formal algebraic gadget;
	\item As a consequence, the homomorphism $\Phi_n$ from \cite[Prop.~8.35]{MakMin_KLR2023} is actually bijective over any field or $\bbZ$;
	\item \cite[Conj.~8.40]{MakMin_KLR2023} becomes trivially true.
\end{itemize}

Finally, let us pay special attention to~\cite[Prop.~A.5]{MakMin_KLR2023}.
It is wrong as stated; indeed, the stratification~\eqref{eq:stratas}, refined by further stratifying nilpotent cones into orbits, satisfies the requirements of~\cite{JMW_PS2014}.
This means that one could in principle consider the usual evenness theory $\Ev^{\mathrm{pc}}$.
However, it does \textit{not} contain every flag sheaf $\cL_\beta$, because they are not constructible with respect to this stratification.
On the other hand, our ``ersatz'' theory $\Ev^\mkp$ does contain every $\cL_\beta$ as per \cref{prop:flag-is-even}.

\begingroup
\setstretch{0.9}
\bibliography{zot-bib}{}
\bibliographystyle{alphaabbr}
\endgroup

\end{document}